\documentclass[11pt, reqno]{amsart}
\usepackage[T1]{fontenc}
\usepackage{amsthm}
\usepackage{amsfonts}
\usepackage{amsmath}
\usepackage[margin=1in]{geometry}
\usepackage{xcolor}
\usepackage{nameref}
\numberwithin{equation}{section}

\let\oldsection\section
\makeatletter
\def\section{%
\@ifstar{\@Starred}{\@nonStarred}%
}
\def\@Starred{%
\@ifnextchar[%
{\GenericWarning{}{Warning: A starred chapter can not have parameters. I am going to ignore them!}\@StarredWith}%
{\@StarredWithout}%
}      
\def\@StarredWith[#1]#2{%
\xdef\mysec{#2}\oldsection{#2}
}
\def\@StarredWithout#1{%
\xdef\mysec{#1}\oldsection*{#1}%
}
\def\@nonStarred{%
\@ifnextchar[%
{\@nonStarredWith}%
{\@nonStarredWithout}%
}
\def\@nonStarredWith[#1]#2{%
\xdef\mysec{#2}\oldsection[#1]{#2}%
}
\def\@nonStarredWithout#1{%
\xdef\mysec{#1}\oldsection{#1}%
}
\makeatother

\usepackage{fancyhdr}
\usepackage{mathrsfs}  
\usepackage{amssymb}
\usepackage{enumitem}
\usepackage{amsmath}
\usepackage{mathtools}
\usepackage{bbm}

\usepackage[style=alphabetic,sorting=nty, isbn=false, doi = false]{biblatex}
\renewbibmacro*{url+urldate}{}
\makeatletter
\def\thmhead@plain#1#2#3{%
  \thmname{#1}\thmnumber{\@ifnotempty{#1}{ }\@upn{#2}}%
  \thmnote{ {\the\thm@notefont#3}}}
\let\thmhead\thmhead@plain
\makeatother

\AtBeginEnvironment{align}{\setcounter{equation}{0}}
\newtheorem{definition}{Definition}
\newtheorem*{definition*}{Definition}

\newtheorem{theorem}{Theorem}
\newtheorem*{theorem*}{Theorem}
\newtheorem*{corollary*}{Corollary}

\newtheorem{proposition}{Proposition}
\newtheorem*{proposition*}{Proposition}
\newtheorem*{lemma*}{Lemma}

\newtheorem{remarkc}[definition]{Remark}

\newtheorem*{remark}{Remark}
\newtheorem{lemma}[definition]{Lemma}
\newtheorem{corollary}[definition]{Corollary}
\numberwithin{definition}{section}

\newcommand{\T}{\mathbb{T}}
\newcommand{\R}{\mathbb{R}}
\newcommand{\N}{\mathbb{N}}
\newcommand{\Z}{\mathbb{Z}}

\newcommand{\calP}{\mathcal{P}}

\newcommand{\calR}{\mathcal{R}}

\newcommand{\calU}{\mathcal{U}}

\newcommand{\calW}{\mathcal{W}}

\newcommand{\calL}{\mathcal{L}}

\newcommand\numberthis{\addtocounter{equation}{1}\tag{\theequation}}

\newcommand{\De}{\Delta}
\newcommand{\calD}{\mathcal{D}}

\DeclareMathOperator{\dir}{dir}
\DeclareMathOperator{\vol}{Vol}
\DeclareMathOperator{\env}{env}

\newcommand{\e}{\varepsilon}
\newcommand{\de}{\delta}

\begin{document}
\title{The quasi-Assouad dimension of $(1,2t)$-Furstenberg sets in $\R^3$ is extremized by sticky sets.}
\author{Sam Craig}

\begin{abstract}
    A $(1,2t)$-Furstenberg set in $\R^3$ is naturally defined as a set containing a union of unit line segments forming a $2t$-dimensional subset of the affine Grassmannian in $\R^3$ and satisfying a suitable variant of the Frostman Convex Wolff Axiom. Some of these sets have a multi-scale self-similarity property called stickiness. We investigate the extremizers of the quasi-Assouad dimension of $(1,2t)$-Furstenberg sets, a slightly stronger variant of the Assouad dimension. We prove that sticky $(1,2t)$-Furstenberg sets have the least possible quasi-Assouad dimension among all $(1,2t)$-Furstenberg sets.
\end{abstract}
\maketitle
\section{Introduction}

The $(s,2t)$-Furstenberg problem asks for the minimum dimension of a set in $\R^d$ which has $s$-dimensional intersection with each unit line segment coming from a $2t$-dimensional family of lines. The study of such sets in the planar case is a famous problem formulated by Furstenberg in the 1970s, with the most difficult cases of the problem recently resolved in papers by Orponen and Shmerkin and Ren and Wang \cite{OS23,RW23}. Substantial difficulties arise when moving to higher dimensions, including finding the correct formulation of the problem. Beyond $\R^2$, it is possible that unions of tubes could be smaller than expected because of excess concentration near planes. For this reason, we additionally require the line segments in $(s,2t)$-Furstenberg sets to satisfy a non-concetration condition called the \emph{$t$-Frostman Convex Wolff Axiom}, first defined by Wang and Wu in \cite{WW24} as a natural generalization of the Frostman Convex Wolff Axiom developed in Wang and Zahl's recent proof of the Kakeya conjecture \cite{WZ22, WZ24, WZ25}.

A key idea in the proofs of both the planar Furstenberg conjecture and the Kakeya conjecture in $\R^3$ is that we should prove sharp results for strongly self-similar sets and use those results and a multiscale decomposition to prove sharp results in general. In this paper, we provide the second step in that line of reasoning for $(1,2t)$-Furstenberg sets, though with a weaker dimension called the \emph{quasi-Assouad dimension}, generalizing a similar result for Kakeya sets by Wang and Zahl \cite{WZ24}. The quasi-Assouad dimension captures the maximum concentration of a set between two close but well-separated scales. Our strongly self-similar sets are a natural generalization of sticky Kakeya sets, which we call \emph{sticky $(1,2t)$-Furstenberg sets}. 

Before defining the dimensions and non-concentration conditions needed to formally state our main theorem, we give an informal variant which captures the main idea of the result. In the main theorem, we need control over sets which are close to being sticky $(1,2t)$-Furstenberg sets in a quantitative sense. For now, we call such sets ``approximately sticky $(1,2t)$-Furstenberg sets''. 
\begin{theorem*}[\textbf{1} (\textit{informal version})]
    Suppose every approximately sticky $(1,2t)$-Furstenberg set has quasi-Assouad dimension $\ge \sigma$. Then every $(1,2t)$-Furstenberg set has quasi-Assouad dimension $\ge \sigma$.
\end{theorem*}

In terms of more familiar dimensions, we have a mildly weaker corollary.

\begin{corollary*}[(\textit{Theorem 1 informal corollary})]
    Suppose every approximately sticky $(1,2t)$-Furstenberg set has Hausdorff dimension $\ge \sigma$. Then every $(1,2t)$-Furstenberg set has Assouad dimension $\ge \sigma$. 
\end{corollary*}

In the next three subsections, we discuss the history of the Furstenberg problem; the non-concentration conditions for lines in $\R^3$; and the upper spectrum and quasi-Assouad dimension. We finish the introduction with a formal statement of Theorem 1.

\subsection{History of the Furstenberg problem}\label{subsec:history}

Define an $\alpha$-Furstenberg set to be a set $E \subset \R^2$ such that for each direction $e \in S^1$, there exists a line segment $\ell$ in direction $e$ intersecting $E$ in a set of Hausdorff dimension $\ge \alpha$. The original formulation of the Furstenberg problem asked for the smallest Hausdorff dimension of an $\alpha$-Furstenberg set. This problem was likely inspired by Furstenberg's work on variants of Cantor sets in 1970, but first appeared in the literature in Wolff's notes on the Kakeya problem \cite{F70, W03}. Wolff gave a pair of elementary arguments proving that such sets have Hausdorff dimension at least $\max(2\alpha, \alpha+1/2)$ and conjectured that $\alpha$-Furstenberg sets should have Hausdorff dimension at least $\frac{3\alpha}{2} + \frac{1}{2}$. This definition was generalized by Molter and Rela to consider lower dimensional line sets as well, leading to the current definition of planar $(s,t)$-Furstenberg sets \cite{MR12}. As mentioned in the introduction, this problem was resolved in 2023 by Ren and Wang, whose work built on papers of Orponen and Shmerkin and of Guth, Solomon, and Wang proving the conjecture for Furstenberg sets of tubes under stronger assumptions on their structure \cite{OS23, GSW19}. The introduction to \cite{RW23} and the other papers mentioned here contain more background on the progress towards the resolution of this problem.

The planar $(s,t)$-Furstenberg set definition generalizes immediately to higher dimensions. With this definition, the question of $(1,2t)$-Furstenberg sets is essentially trivial. An $L^2$ argument appearing in a paper of H\'era, Shmerkin, and Yavicoli (who attributed it to Guth) tells us that in any dimension, such sets must have Hausdorff dimension at least $\min(2, 1 + 2t)$ \cite[Appendix A]{HSY21}. It is not hard to prove that these bounds are sharp: if $t < 1/2$, we take the product of an interval and any planar $t$-dimensional set; and if $t \ge 1/2$ we put all the line segments in a hyperplane. Regardless of the ambient dimension, under this definition, the $(1,2t)$-Furstenberg set problem degenerates to the planar problem. This phenomenon appears in the discrete setting as well. The Szemeredi-Trotter theorem tells us that \emph{for any ambient dimension} $n$, sets $\calP$ of points in $\R^n$ and $\calL$ of lines in $\R^n$ have $\lessapprox (|\calP||\calL|)^{2/3} + |\calP| + |\calL|$ many incidences. 

However, there are reasonable conditions one could place on the set of line segments which force the inequality to reflect the dimensionality of the ambient space. One such condition for lines in $\R^3$ is that at most $|\calL|^{1/2}$ lines lie in any hyperplane in $\R^3$. In this case, Guth gives a stronger bound on the number of incidences \cite[Theorem 12.1]{G16}\begin{equation}\label{eqn:STR3}I(\calP, \calL) \lesssim |\calP|^{1/2}|\calL|^{3/4} + |\calP|^{2/3}|\calL|^{1/2} + |\calL| + |\calP|.\end{equation}If our line set has $\de^{-2t}$ elements and each element intersects $\de^{-1}$ many points, we can compute that the final term in (\ref{eqn:STR3}) must dominate the sum and hence our set must contain $\gtrsim \de^{-1-2t}$ points. 

We need an additional continuous non-concentration condition which rules out concentration near hyperplanes to find a genuinely three-dimensional variant of the Furstenberg set problem. The $t$-Frostman Convex Wolff Axiom arises as a natural choice. We discuss the $t$-Frostman Convex Wolff Axiom in more detail in the following subsection. 

\subsection{Background on non-concentration conditions for line segments}\label{subsec:def}

Now, we discuss continuous analogs of the Frostman Convex Wolff Axiom. First, we give some context for this condition in the discrete case. The Frostman Convex Wolff Axiom is a recent creation and its motivation may not be familiar to many readers, so we describe how it naturally arises when working on incidence problems in dimension $\ge 3$. 

The classical definition of Kakeya sets in $\R^3$ is a union of unit line segments containing at least one in each direction in $S^2$. The condition of having one line segment in each direction in $S^2$ implies that the set of line segments constituting the Kakeya set must have Hausdorff dimension $\ge 2$ in the affine Grassmannian of lines in $\R^3$ (after identifying line segments with the affine line they lie on). On the other hand, if $L$ is a set of line segments which is $2$-dimensional in the affine Grassmannian in the same sense, $\bigcup_{\ell \in L} \ell$ can have Hausdorff dimension $2$ instead of the conjectured dimension $3$, for example if all lines fall in a hyperplane. The classical Kakeya condition rules this case out, but is not well-adapted to the Furstenberg problem, since unlike in the Kakeya set case, a $2t$-dimensional subset $D \subset S^2$ is not necessarily evenly distributed in or near geodesics on the sphere, which means that a set of line segments pointing in the directions of $D$ could be overly concentrated in or near a hyperplane. 

Fortunately, the literature on the Kakeya problem indicates a more suitable condition, at least after discretization. In 1995, Wolff used a brush argument to prove that Kakeya sets in $\R^3$ have Hausdorff dimension at least $\frac{5}{2}$ \cite{W95}. Wolff quickly discretized the Kakeya problem to a problem about unions of $\de$-tubes with one pointing in each direction. He wanted an argument which applied in greater generality than just Kakeya sets, so he found a minimal pair of conditions on the concentration of tubes for which his argument applied. The first non-concentration condition is that the tubes must satisfy a discertized variant of the $2$-Frostman measure condition in the affine Grassmannian, which implies that for any $\rho \in [\de, 1]$ and any $\rho$-tube $T_{\rho}$, \[|\{T \in \T : T \subset T_{\rho}\}| \le \rho^2 |\T|.\]The second non-concentration conditon is now called the \emph{Wolff Axiom}. We say that a set of tubes $\T$ satisfies the Wolff Axiom with error $C$ if for any $\rho \in [\de, 1]$ and any $2\de \times \rho \times 1$ slab $V$, then \[|\{T \in \T : T \subset V\}| \le C\frac{\sigma}{\de}.\]This is suitable for arguments that only require information on the structure of the line segments at scale $\de$, but this paper and recent breakthroughs on Kakeya in $\R^3$ and the planar Furstenberg problem have required information on the structure of the line segments at scale $\rho$ at all $\rho \in [\de, 1]$. The Wolff Axiom for $\T$ does not impart useful information about $\T$ at scales other than $\de$, so it is not suitable for multiscale arguments. 

To address this, we want a non-concentration condition for sets of $\de$-tubes $\T$ that captures the Wolff Axiom condition at many different scales and applies for Kakeya sets. We make the simplifying assumption that for any scale $\rho > \de$, we can cover $\T$ with a set of $\rho$-tubes $\T_{\rho}$, so that for each $T_{\rho} \in \T_{\rho}$, $|\{T \in \T : T \subset T_{\rho}\}|$ is a fixed constant $M_{\rho}$, and that for distinct $T_{\rho}, T'_{\rho} \in \T_{\rho}$, $\{T \in \T : T \subset T_{\rho} \}$ and $\{T \in \T : T \subset T'_{\rho}\}$ are disjoint. Given this assumption, we see that if $\T_{\rho}$ satisfies the Wolff Axiom at scale $\rho$, then for any $2\rho \times \sigma \times 1$ prism $V$, \[|\{T \in \T : T \subset V\} \le C\frac{\sigma}{\rho}M_{\rho}.\]If $\T$ also satisfies the $2$-Frostman condition with error $C$, then $M_{\rho} \le C\rho^2 |\T|$, so \[|\{T \in \T : T \subset V\} \le C^2\sigma \rho |\T| = C^2 \vol(V)|\T|.\]It will be convenient to allow $V$ to be any convex set, rather than just rectangular prisms, and with this modification we arrive at what is now called the Frostman Convex Wolff Axiom. 

\begin{definition*}\cite{WZ24}
    We say a finite set of tubes $\T$ in $\R^3$ satisfies \emph{Frostman Convex Wolff Axiom with error $C$} if for any convex set $W \subset \R^3$, \[|\{T \in \T : T \subset W\}| \le C \vol(W)|\T|.\]
\end{definition*}

Our derivation of the Frostman Convex Wolff Axiom shows that if $\T$ satisfies the Frostman Convex Wolff Axiom with small error, then not only does $\T$ itself essentially satisfy the Wolff Axiom with small error, but any cover of $\T$ by $\rho$-tubes $\T_{\rho}$ essentially satisfies the Wolff Axiom with small error as well. The Frostman Convex Wolff Axiom first appeared in Wang and Zahl's proof that Kakeya sets have full Assouad dimension, where they called it the Convex Wolff Axiom \cite{WZ24}. Wang and Zahl used it again in their proof that Kakeya sets have full Hausdorff dimension, where they referred to it as the Frostman Convex Wolff Axiom, which is the name we use in this paper \cite{WZ25}.

This motivation for the Frostman Convex Wolff Axiom largely applies to the $t$-Frostman Convex Wolff Axiom as well. A naive definition of a $(1,2t)$-Furstenberg set in $\R^3$ is a union of line segments making up a $2t$-dimensional subset of the affine Grassmannian. When $t \le 1$, all line segments can fall in a plane, in which case this problem devolves to the planar Furstenberg problem. For a genuinely three dimensional variant of Furstenberg sets, we would like a $t$-dimensional variant of the Wolff Axiom to hold. Specifically, after discretizing our lines to a set of $\de$-tubes, we would like \[|\{T \in \T : T \subset V\}| \le C\left(\frac{\sigma}{\de}\right)^t.\]As in the Kakeya setting, this does not hold at each scale, but if we instead require that \[|\{T \in \T : T \subset V\}| \le C \vol(V)^t |\T|,\]we have the $t$-dimensional variant of the Wolff Axiom at each scale. This condition is now called the $t$-Frostman Convex Wolff Axiom.

\begin{definition*}\cite{WW24}
    We say a finite set of tubes $\T$ in $[0,1]^3$ satisfies the \emph{$t$-Frostman Convex Wolff Axiom with error $C$} if for any convex set $W \subset \R^3$, \[|\{T \in \T : T \subset W\}| \le C \vol(W)^t|\T|.\]
\end{definition*}

This definition originates from Wang and Wu's work on the Fourier restriction problem, where they predicted such estimates could serve is an intermediate step towards improved restriction estimates in higher dimensions \cite{WW24}. We mildly generalize to replace sets of tubes with other collections of convex sets and to allow for ambient spaces other than the unit cube.

\begin{definition}
    We say a finite set of congruent convex sets $\calR$ satisfies the \emph{$t$-Frostman Convex Wolff Axiom with error $C$} in a convex set $V$ if each $R \in \calR$ is contained in $2V$ and for any convex set $W \subset 2V$, \[|\{R \in \calR : R \subset 2W\}| \le C \frac{\vol(W)^t}{\vol(V)^t}|\calR|.\]
    If $V = [0,1]^3$, we simply say that $\calR$ satisfies the $t$-Frostman Convex Wolff Axiom with error $C$.
\end{definition}

\begin{remark}
    When discussing the $t$-Frostman Convex Wolff Axiom and its relatives informally, we often omit the error and simply say $\calR$ satisfies the $t$-Frostman Convex Wolff Axiom. This should be interpreted as $\calR$ satisfying the $t$-Frostman Convex Wolff Axiom with some sufficiently small error.
\end{remark}

These conditions were originally developed for unions of $\de$-tubes, since standard discretization arguments reduce continuum incidence problems like the Kakeya problem to lower bounds on the volume of unions of $\de$-tubes. We would like a continuous variant of the Frostman Convex Wolff Axiom to capture the original spirit of the Furstenberg problem. Our definition is a mild generalization of the Oberlin affine curvature condition, a pleasingly simple non-concentration condition originally intended to capture the curvature of submanifolds in Euclidean space \cite{O00}.

In our definition, we restrict our set of line segments from the affine Grassmannian to the collection of lines $\mathcal{L}$ defined by \begin{equation}\label{eqn:calL}\mathcal{L} = \{\{(a, b, 0) + t(c, d, 1) : t \in \R\} : a, b, c, d \in (-1,1)^4\}.\end{equation}There is no harm in doing so, as the collections of lines intersecting the unit ball is the finite union of isometric copies of $\mathcal{L}$. We equip $\mathcal{L}$ with the flat metric induced by the obvious diffeomorphism with $(-1,1)^4$. We denote the Riemannian volume form on $\mathcal{L}$ by $\text{Vol}_{\mathcal{L}}$ and say a subset $L \subset \mathcal{L}$ is convex if it is geodesically convex (that is, for any pair of points $\ell, \ell' \in L$, if a geodesic $\gamma$ connects $\ell$ and $\ell'$, then $\gamma \subset L$). 

Finally, we are ready to define the continuous $t$-Frostman Convex Wolff Axiom. 
\begin{definition}\label{def:cts}
    We say a collection $L$ of unit line segments, compact in the affine Grassmannian of $\R^3$, satisfies the \emph{continuous $t$-Frostman Convex Wolff Axiom with error $K$} if it supports a probability measure $\mu$ such that for any convex set $W \subset \mathcal{L}$, \begin{equation}\label{eqn:bear}\mu(W) \le K\left(\vol_{\mathcal{L}}(W)\right)^{t/2}. \end{equation}
    We say a set $E \subset [0,1]^3$ is a \emph{$(1,2t)$-Furstenberg set} if there is a collection of line segments $L$ each contained in $E$ and satisfying the $t$-Frostman Convex Wolff Axiom with error $K$ for some $K$.
\end{definition}

For comparison, we state a continuous version of the $\alpha$-Frostman measure condition alluded to previously. 

\begin{definition}\label{def:frostmanmeasure}
    Let $(X, d)$ be a metric space. We say a set $Y \subset X$ is an \emph{$\alpha$-Frostman measure} with error $C > 0$ if $Y$ supports a probability measure $\mu$ such that for any $r \in (0, \text{diam}(X))$ and any $r$-ball $B$, \[\mu(Y \cap B) \le Cr^{\alpha}.\]
\end{definition}

If $L$ satisfies the continuous $t$-Frostman Convex Wolff Axiom, $\mu$ is the probability measure supported on $L$ satisfying (\ref{eqn:bear}), and $W$ is an $r$-ball, we have $\mu(W) \le Kr^{2t}$, so we see that $L$ must be a $2t$-Frostman measure. Definition \ref{def:cts} is stronger than simply being a $2t$-Frostman measure, since it also rules out concentration in convex sets.

We are now ready to define $(1,2t)$-Furstenberg sets.

\begin{definition*}
    A compact set $E \subset [0,1]^3$ is a \emph{$(1,2t)$-Furstenberg set} if there is a set $L$ of unit line segments satisfying the continuous $t$-Frostman Convex Wolff Axiom with some error $K$ such that $\ell \subset E$ for each $\ell \in L$.
\end{definition*}

Similarly to how the $t$-Frostman Convex Wolff Axiom prevents a collection of tubes in $\R^3$ from looking like a collection of tubes in $\R^2$, the continuous $t$-Frostman Convex Wolff Axiom prevents $(1,2t)$-Furstenberg sets in $\R^3$ from resembling a planar Furstenberg set by limiting the mass of the line segments that can lie in the neighborhood of a hyperplane. 

We define sticky Furstenberg sets next. Originally, sticky sets were defined as Kakeya sets which comprise line segments with equal packing and Hausdorff dimension in the affine Grassmannian. Sticky Kakeya sets have a strong self-similarity property, which Wang and Zahl used to prove that they have have full Hausdorff dimension in $\R^3$ \cite{WZ22}. Wang and Zahl later found that a discretized condition which accepts small errors was better suited for applications to proving the general Kakeya problem. They called this condition the \emph{Frostman Convex Wolff Axiom at all scales} and their proof in \cite{WZ22} implied such sets have full Hausdorff dimension. We modify the Frostman Convex Wolff Axiom at all scales to the $t$-Frostman Convex Wolff Axiom at all scales analogously to how Wang and Wu modified the Frostman Convex Wolff Axiom to the $t$-Frostman Convex Wolff Axiom in \cite{WW24}.

\begin{definition}\label{def:FCWAAAS}
    We say $\T$ satisfies the \emph{$t$-Frostman Convex Wolff Axiom at all scales with error $C$} if for any $\rho_0 \in [\de, 1]$, there exists some $\rho \in [\rho_0, C\rho_0]$ such that each $T \in \T$ is contained in a $\rho$-tube $T_{\rho}$ and $\T[T_{\rho}] := \{T \in \T:T \subset 2T_{\rho}\}$ satisfies the $t$-Frostman Convex Wolff Axiom with error $C$ in $T_{\rho}$.

    If $\T$ is a set of $\de$-tubes satisfying the $t$-Frostman Convex Wolff Axiom at all scales with error $C$, then we say $\bigcup_{T \in \T} T$ is a \emph{sticky $(1,2t)$ Furstenberg set at scale $\de$ with error $C$.}
\end{definition}

We define a sticky $(1,2t)$-Furstenberg set to be a set containing a set of $\de$-tubes satisfying the $t$-Frostman Convex Wolff Axiom at all scales. For our application, we need to track both the error in the $t$-Frostman Convex Wolff Axiom at all scales and the scale $\de$ in the definition of sticky $(1,2t)$-Furstenberg sets.

\begin{definition*}
    We say $E$ is a \emph{sticky $(1,2t)$-Furstenberg sets at scale $\de$ with error $C$} if $E$ contains a collection of $\de$-tubes $\T$ satisfying the $t$-Frostman Convex Wolff Axiom at all scales with error $C$.
\end{definition*}

We have already seen that for a set of $\de$-tubes $\T$ satisfying the $t$-Frostman Convex Wolff Axiom, a reasonably uniform cover $\T_{\rho}$ of $\T$ by $\rho$-tubes also satisfies the $t$-Frostman Convex Wolff Axiom. If $\T$ satisfies the $t$-Frostman Convex Wolff Axiom at all scales, then in addition to $\T_{\rho}$ satisfying the $t$-Frostman Convex Wolff Axiom, we also know that $\T[T_{\rho}]$ satisfies the $t$-Frostman Convex Wolff Axiom in $T_{\rho}$. This condition gives dimensional extremizers an additional special structures called grains. The idea of aiming for grains originated from an attempt by Katz and Tao in the early 2000s to solve the Kakeya conjecture in $\R^3$. Katz and Tao's approach can be found in Tao's blog post on their attempt at the Kakeya conjecture \cite{T14}. 

The grains structure gives the best chance to make improvements by applying projection or sum-set theorems. Wang and Zahl used this argument in 2022 to prove that sticky Kakeya sets in $\R^3$ have full Hausdorff dimension \cite{WZ22}. Wang and Zahl later used methods closer to what appears in this paper to prove if sticky Kakeya sets have full Hausdorff dimension, then the general case of the Kakeya conjecture is true as well. They proved this first for the Assouad dimension and later the stronger Hausdorff dimension \cite{WZ24, WZ25}. We follow the second step of Wang and Zahl's reasoning for the quasi-Assouad dimension, proving that if we have good quasi-Assouad dimensional bounds on sets of $\de$-tubes satisfying the $t$-Frostman Convex Wolff Axiom at all scales, then we have good quasi-Assouad dimensional bounds for any $(1,2t)$-Furstenberg set.

\subsection{Background on the quasi-Assouad dimension}\label{subsec:qA}

The quasi-Assouad dimension falls into the large family of dimensions which measure the dimension of a set $E$ by applying the heuristic that $E$ is $\alpha$ dimensional if $|E|_{\de} \ge \de^{-\alpha}$ for any $\de$ sufficiently small. The most famous and strongest member in this family is the Hausdorff dimension. We say a set $E$ has Hausdorff dimension $\ge \alpha$ if \[\lim_{\de \to 0} \inf \left\{ \sum_{i=1}^{\infty} \text{radius}(B_i)^{\alpha} : \{B_i\}_{i=1}^{\infty} \text{ a collection of balls with }\bigcup_{i=1}^{\infty} B_i \supset X, \text{radius}(B_i) < \de\right\} > 0.\]
Hausdorff dimensional bounds implies fine-grained control over the global spacing of $E$ at any sufficiently small scale. In some applications, weaker notions of dimensions from this family are more appropriate. Assouad considered the question of when a metric space $(E,d)$ could be embedded into Euclidean space through a bi-Lipschitz map \cite{A77}. He proved that if there exists a finite value $\alpha$ for which there exists a constant $C$ such that for any $r < R$ and $x \in E$, \[|B(x, R) \cap E|_{r} \le C \left(\frac{R}{r}\right)^\alpha,\]then for any $\e < 1$, $(E, d^{\e})$ has a bi-Lipschitz embedding into Euclidean space. This condition gave rise to what is now called the \textit{Assouad dimension}. 

\begin{definition*}
    We say the \emph{Assouad dimension} of a set $E$ is the infimal value of $\alpha \ge 0$ for which there exists a constant $C$ such that for any $r < R$ and $x \in E$, \begin{equation}\label{eqn:assouad}|B(x, R) \cap E|_{r} \le C \left(\frac{R}{r}\right)^\alpha,\end{equation}where $|F |_r$ denotes the $\de$-covering number of $F$.
\end{definition*}

The Assouad dimension is the weakest reasonable dimension coming from counting the size of a $r$-covering of $E$. In fact, it is too weak to be useful in this paper. It seems very difficult to make any use from Assouad dimensional bounds for sticky $(1,2t)$-Furstenberg sets $E$, because any small modification to $E$ could cause a substantial decrease in the Assouad dimension if the pair of scales $\de, \rho$ extremizing (\ref{eqn:assouad}) are very close together. To address this issue, we need a dimension which forces some separation between the scales $\de, \rho$. We begin by defining the upper spectrum, a dimension parametrized by some parameter $\eta \in (0, 1]$ which forces some separation between the pair of scales witnessing the Assouad dimension. We need both a discrete and continuous variant of this definition.

\begin{definition*}
    A set $E$ comprising a union of $\de$-cubes has \emph{upper spectrum at $\eta$ and scale $\de$} at least $\sigma$ if there exists $\de \le r < R \le 1$ with $r \le\de^{\eta} R$ and an $R$-ball $B_R$ such that $|E \cap B_R|_{r} \ge \left(\frac{R}{r}\right)^{\sigma}$.

    For any $E \subset \R^3$, we say $E$ has \emph{upper spectrum at $\eta$} at least $\sigma$ if there exists some $\de > 0$ for which $E$ has upper spectrum at $\eta$ and scale $\de$ at least $\sigma$
\end{definition*}

We record a more common but plainly equivalent definition of the upper spectrum at $\eta$. This definition originates from work of Fraser, Kathryn Hare, Kevin Hare, Troscheit, and Yu in 2018 \footnote{ The parameter $\eta$ in this paper is more commonly replaced by $1-\eta$ in the literature, in the sense that the dimension in the upper spectrum which we label by $\eta$ is usually labelled by $1-\eta$. It is convenient to use the same parameter for the scale separation as other error terms, motivating our choice of taking $\eta$ closer to $0$.}\cite{FHHTY19}.

\begin{definition*}
    A set $E \subset \R^3$ has upper spectrum at $\eta$ at least $\sigma$ if there exists scales $r, R$ with $r < r^{1-\eta} \le R$ and an $R$-ball $B_R$ such that $|E \cap B_R|_{r} \ge \left(\frac{R}{r}\right)^{\sigma}$.
\end{definition*}

The quasi-Assouad dimension is defined as the infimum of the upper spectrum at $\eta$ as $\eta \searrow 0$. Again, we give discretized and continuous variants of this definition.

\begin{definition*}
    A set $E$ comprising a union of $\de$-cubes has \emph{quasi-Assouad dimension at scale $\de$} at least $\sigma$ if for any $\e > 0$, there exists $\eta > 0$ such that $E$ has upper spectrum at $\eta$ and scale $\de$ at least $\sigma - \e$.

    A set $E \subset \R^3$ has \emph{quasi-Assouad dimension} at least $\sigma$ if for any $\e > 0$, there exists $\eta > 0$ such that $E$ has upper spectrum at $\eta$ at least $\sigma - \e$.
\end{definition*}

This turns out to be a stronger dimension than the Assouad dimension itself. When we try to apply quasi-Assouad dimensional bounds, we see a clear benefit compared with Assouad dimensional bounds. Quasi-Assouad dimensional bounds for a sticky $(1,2t)$-Furstenberg set $E$ implies nearly the same bounds for the upper spectrum of $E$ at small values of $\eta$. In the proof of our main theorem, these good bounds on the upper spectrum of $E$ at small values of $\eta$ play a crucial role in proving good bounds on the upper spectrum of any $(1,2t)$-Furstenberg set for small values of $\eta$, which in turn imply good bounds for the quasi-Assouad dimension of any $(1,2t)$-Furstenberg set.

On the other hand, the quasi-Assouad dimension still allows us mostly free choice of the range of scales where we measure the dimensionality of our Furstenberg sets. In several key points of the proof of the main theorem, we prove upper spectrum bounds for a Furstenberg set by carefully choosing scales $\de \ll \rho$ where we know (\ref{eqn:assouad}) applies. In particular, this idea underlies the extremization argument in Section \ref{sec:proof} and finding the quasi-Assouad dimension of unions of prisms in Section \ref{sec:prismsetup}. 

For further discussion on the Assouad dimension and its variants, we refer to the reader to Fraser's excellent monograph on the topic \cite{F20}.

\subsection{Statement of the main result}
We are now ready to formally state the main theorem of this paper, which says that a quantitative form of a quasi-Assouad dimensional bound for sticky $(1,2t)$-Furstenberg sets imply all $(1,2t)$-Furstenberg set satisfy the same bound.

\begin{theorem}\label{thm:main}
    Fix $t \in (0, 1]$. Suppose there exists $\sigma > 0$ such that the following statement is true. 
    
    \begin{center}For any $\e > 0$, there exists $\eta, \de_0 > 0$ such that for any $\de \in (0, \de_0)$, every sticky $(1,2t)$-Furstenberg set at scale $\de$ with error $\de^{-\eta}$ has quasi-Assouad dimension at scales $\de$ at least $\sigma - \e$.\end{center}
    Then every $(1,2t)$-Furstenberg set has quasi-Assouad dimension at least $\sigma$.
\end{theorem}

\subsection{Paper outline}\label{subsec:PO}

\begin{itemize}
    \item In Section \ref{sec:bpop}, we outline the proof of Theorem \ref{thm:main}. 
    \item In Section \ref{sec:disc}, we carry out an initial discretization, reducing Theorem \ref{thm:main} to a discretized variant Theorem \ref{thm:realmain}.
    \item In Section \ref{sec:multi}, we prove a multiscale decomposition for sets of $\de$-tubes.
    \item In Section \ref{sec:proof}, we prove Theorem \ref{thm:realmain}, given Proposition \ref{prop:prismsetup}, which gives a good bound on the upper spectrum of sets of tubes that look $t$-dimensional in a collection of flat prisms\footnote{A flat prism is a right rectangular $a \times b \times 1$ prism with $a \ll b$.}. The rest of the paper will be devoted to the proof of Proposition \ref{prop:prismsetup}.
    \item In Section \ref{sec:prismsetup}, we use an induction on scales argument to prove Proposition \ref{prop:prismsetup}. This induction on scales argument relies on Proposition \ref{prop:largeprism} and Proposition \ref{prop:smallprism}, which we prove in the subsequent sections.
    \item In Section \ref{sec:largeprism}, we prove Proposition \ref{prop:largeprism}, which states that a collection of tubes that look $1$-dimensional in a collection of flat prisms either arrange to look $1$-dimensional in substantially larger flat prisms or have large upper spectrum.
    \item In Section \ref{sec:smallprism}, we prove Proposition \ref{prop:smallprism}, which states that a collection of tubes that look $t$-dimensional in a collection of flat prisms either arrange to look $1$-dimensional in slightly larger flat prisms or have large upper spectrum.
    \item In Section \ref{sec:lem}, we give proofs for a brush argument and various $L^2$ argument we use earlier in the paper. 
\end{itemize}

\subsection{Notation}\label{subsec:notation}

\begin{enumerate}[label = (\roman*)]
    \item We write $\T$ to denote a set of tubes. For other collections of convex sets, we usually use the calligraphic font, for example $\calR$ or $\calW$. If we denote a collection of congruent sets with a letter in calligraphic font, we may refer to an abstract element of that set by the same letter in regular font. For example, if $\calR$ is a collection of congruent, convex sets, then we denote the volume of an element in $\calR$ by $\vol(R)$. 
    \item We say a pair of tubes $T, T'$ is essentially distinct if $T \not \subset 2T'$. If $\T$ is a set of tubes with each pair of tubes essentially distinct, we say that $\T$ is a set of essentially distinct tubes.
    \item For a set $A \subset \R^3$ and $\de > 0$, we write $N_{\de}A) = \{x \in \R^3, \text{dist}(x, A) < \de\}$. If $\mathcal{A}$ is a collection of subsets of $\R^3$ and $\de > 0$, we write $\calD_{\de}(\mathcal{A}) = \{N_{\de}(A) :A \in \mathcal{A}\}$.
    \item When we say prism, we always mean a right rectangular prism. Any prism $W$ with side lengths $a \times b \times 1$ has a central affine plane $\Pi$ such that $\Pi$ intersects $W$ in a $b \times 1$ rectangle. We denote the normal direction to $\Pi$ by $n(W)$. This value is well-defined up to error of magnitude $\frac{a}{b}$. For a prism $W$, we denote the prism with the same center and direction and dimensions scaled by $C$ as $CW$. 
    \item For a family of convex sets $\calR$ and a convex set $W$, we denote $\calR[W] = \{R \in \calR: R \subset 2W\}$.
    \item We say a dyadic $\de$-cube is a set of the form $[x, x+\de] \times [y, y + \de] \times [z, z+\de]$ for dyadic values $x,y,z \in \R$. Write $|E|_{\de}$ to be the the maximal number of disjoint dyadic $\de$-cubes incident to $E$.  
    \item Let $|\cdot|$ be the counting measure, $\text{Vol}$ the three-dimensional Lebesgue measure, $\text{Vol}_{\R^2}$ the two-dimensional Lebesgue measure and $\text{Vol}_{\R}$ the one-dimensional Lebesgue measure.
    \item We write $a \lesssim_{\alpha, \beta, \dots} b$ if there exists a constant $C > 0$ depending only on $\alpha, \beta, \dots$ such that $a \le Cb$. If $C$ is an absolute constant, we simply write $a \lesssim b$. We write $a \lessapprox_{\de; \alpha, \beta, \dots} b$ if there exist absolute constants $C, N$ depending only on $\alpha, \beta, \dots$ such that $a \le C\log(1/\de)^N b$ when $\de < 1$. If $C, N$ are absolute constants, we write $a \lessapprox_{\de} b$.
    \item We use $\de, \De, \rho, r, R, a, b$ to denote scales. We use $\eta, \e, \omega, \tau$ to denote errors. 
\end{enumerate}

\subsection{Acknowledgement}\label{subsec:Acknowledgements}

The author thanks Betsy Stovall for her advice and support during this project and for many suggestions which greatly improved this manuscript. The author was supported during this project by the National Science Foundation under grant numbers DMS-2037851 and DMS-2246906. 

\section{Outline of the proof of Theorem 1}\label{sec:bpop}

Now we outline in more detail the proof of Theorem \ref{thm:main} and discuss the differences between the $t=1$ case that is already known and the $t < 1$ case that we prove here. We first note that Theorem \ref{thm:main} reduces to the following discretized variant. 

\begin{theorem}\label{thm:realmain}
    Suppose that for any $\e > 0$, there exists $\eta, \de_0 > 0$ such that for any $\de \in (0, \de_0)$ and any set of $\de$-tubes $\T$ satisfying the $t$-Frostman Convex Wolff Axiom at all scales with error $\de^{-\eta}$, $\bigcup_{T \in \T} T$ has upper spectrum at $\eta$ and scale $\de$ at least $\sigma - \e$. 

    Then for any $\omega > 0$, there exists $\eta, \de_0 > 0$ such that for any $\de \in (0, \de_0)$ and any set of $\de$-tubes $\T$ satisfying the $t$-Frostman Convex Wolff Axiom with error $\de^{-\eta}$, $\bigcup_{T \in \T} T$ has upper spectrum at $\eta$ and scale $\de$ at least $\sigma - \omega$.
\end{theorem}

The reduction from Theorem \ref{thm:main} to Theorem \ref{thm:realmain} takes place in Section \ref{sec:disc}. The proof of the reduction is similar to the well-known argument reducing dimensional estimates for Kakeya sets to volume bounds on unions of $\de$-tubes with one in each direction, but in the Furstenberg case, we need more care to use definition \ref{def:cts} to ensure that the union of $\de$-tubes satisfies the $t$-Frostman Convex Wolff Axiom. For the remainder of this section, we focus on the proof of Theorem \ref{thm:realmain}. 

\subsection{Proving Theorem \ref{thm:realmain} in the $t=1$ case}\label{subsec:bpop1}

To motivate our proof of Theorem \ref{thm:realmain} in the general case, we first outline the proof for the $t = 1$ case from \cite{WZ24}. We begin with the set of $\de$-tubes $\T$ satisfying the $1$-Frostman Convex Wolff Axiom with error $\approx 1$.

\begin{enumerate}[label = Step \arabic*]
    \item We consider the structure of the densest quasi-Assouad extremal sets of tubes. By this, we mean sets of $\de$-tubes $\T$ satisfying the $t$-Frostman Convex Wolff Axiom with the smallest possible quasi-Assouad dimension and with the largest possible value of $\log_{1/\de}(|\T|)$ among sets of tubes $\T$ with the smallest possible quasi-Assouad dimension. We prove that these tubes must either be sticky or fill out a set of $a \times b \times 1$ flat prisms $\calR$. If the set of tubes are sticky, we use our assumption on sticky Furstenberg sets to close the argument.
    \item Let $\T_a$ be a cover of $\T$ by $a$-tubes. First, suppose that at each point $x \in \bigcup_{T \in \T_a} T$, all the prisms $R \in \calR$ such that $x \in T$ for some $T \in \T_a[R]$ intersect tangentially. By this, we mean that the normal angle of the center plane of the prisms is contained in a $\lesssim \frac{b}{a}$ ball in the unit sphere. Then the prisms in $\calR$ must arrange themselves into $a' \times b' \times 1$ prisms with $a' \ll b'$ and $b' \gg b$. We repeat this step until we reach a transverse intersection, each time increasing the scales by a fixed amount. In finitely many steps, we have $b' = 1$, in which case the prisms must intersect transversely \footnote{We actually need the prisms to intersect transversally or tangentially after rescaling the prisms through larger tubes. The following steps also need to be carried out after rescaling. This adds some technical complexity to the proof, but the steps here capture the key ideas.}. 
    \item Now, suppose that at each point $x \in \bigcup_{T \in \T_a} T$, all the prisms $R \in \calR$ such that $x \in T$ for some $T \in \T_a[R]$ intersect transversely, that is, with normal angle $\gg \frac{b}{a}$. Using an $L^2$ argument, we see that $\bigcup_{T \in \T_a[R]} T \approx R$, so it suffices to prove $\bigcup_{R \in \calR} R$ has large upper spectrum.
    \item Now, we break the $a \times b \times 1$ prisms in $\calR$ into disjoint unions of $a \times b \times b$ small prisms. We can use a brush argument between the small prisms to find an $c \times b \times b$ thick prism $K$ such that $K \cap \bigcup_{R \in \calR} R \approx K$ and $a \ll c \le b$. Pigeonholing in this prism, we find a ball $B_R$ such that $B_R \cap \bigcup_{R \in \calR} R \approx B_R$. This gives quasi-Assouad dimension $3 \ge \sigma$, completing the argument. 
    
\end{enumerate}
Now, we move to a discussion of the additional difficulties in the proof of Theorem \ref{thm:realmain}. Step 1 carries over without significant changes; we give more details in the following subsection. The other three steps require substantial change, since the reasoning of \cite{WZ24} breaks down in two ways when the direction set has dimension $< 2$.

The first issue is that $\T_a[R]$  no longer sufficiently concentrated to essentially fill out $R$. This means that \textit{a priori} we do not have spacing in the the union of tubes in $a \times b \times b$ subsets of $R$ necessary to carry out the brush and $L^2$ arguments referred to in Step 4. We specifically need $\T_P = \{T \cap P : T \in  \T_a[R]\}$ to satisfy the $t$-Frostman Convex Wolff Axiom in $P$ for many different choices of $a \times b \times b$ prisms $P \subset R$. 

The second problem is that if $K$ is a $c \times b \times b$ prism such that $\bigcup_{R \in \calR} \bigcup_{T \in \T_a[R]} T \cap K$ concentrates like a $\sigma$-dimensional set in $K$, this does not imply that $\bigcup_{R \in \calR} \bigcup_{T \in \T_a[R]} T \cap B_c$ looks $\sigma$-dimensional for any $c$-ball $B_c$, because the mass of $\bigcup_{R \in \calR} \T_a[R] \cap K$ could be widely spread through $K$, rather than concentrating like a $\sigma$-dimensional set in any individual $c$-ball. This phenomenon does not occur when $\sigma$ is the ambient dimension (in this paper, $3$), since in that case, the set must concentrate fully in any subset of $K$. This is how \cite{WZ24} closes the argument, but we must do something different when $t <1$. 

\subsection{Proving Theorem \ref{thm:realmain} in the general case: reducing to tubes concentrating in prisms}

Our proof of Theorem \ref{thm:realmain} in the general case begins with a generalization of Step 1 from the $t=1$ case. As mentioned above, the $t=1$ argument generalizes fairly easily to the $t < 1$ case, although it requires some technical work in either case. This step comprises Sections \ref{sec:multi} and \ref{sec:proof}. In Section \ref{sec:multi}, we prove that for any set of $\de$-tubes $\T$ satisfying the $t$-Frostman Convex Wolff Axiom with small error and $|\T| \approx \de^{-\beta}$, one of four outcomes must happen. 
\begin{enumerate}[label = (\Alph*)]
    \item There exists a scale $\rho$ and a cover of $\T$ by $\rho$-tubes $\T_{\rho}$ such that $\T_{\rho}$ satisfies the $t$-Frostman Convex Wolff Axiom with small error and is denser than $\T$, that is, $|\T_{\rho}| \gg \rho^{-\beta}$.
    \item There exists a scale $\rho$ and a $\rho$-tube $T_{\rho}$ such that $\T[T_{\rho}] = \{T \in \T: T \subset 2T_{\rho}\}$ satisfies the $t$-Frostman Convex Wolff Axiom in $T_{\rho}$ with small error and is denser than $\T$, that is $|\T[T_{\rho}]| \gg \left(\frac{\de}{\rho}\right)^{\beta}$. We actually prove the same conclusion for the image of $\T[T_{\rho}]$ under the affine mapping sending $T_{\rho}$ to the unit cube. The image of $\T[T_{\rho}]$ under this mapping is essentially a set of $\de/\rho$-tubes and we denote it by $\T^{T_{\rho}}$. 
    \item There exists scales $\rho, \De$ with $\de \le \rho \ll \De \le 1$ and a set of $\rho \times \De \times 1$ prisms $\calW$ so that $\T[W] = \{T \in \T : T \subset 2W\}$ satisfies the $t$-Frostman Convex Wolff Axiom in $W$ with small error and $\calW$ itself satisfies the $t$-Frostman Convex Wolff Axiom with small error.
    \item $\T$ satisfies the $t$-Frostman Convex Wolff Axiom at all scales with small error.
\end{enumerate}

We give a precise statement in Proposition \ref{prop:WZ256.3}, whose proof we outline here. The proof follows similarly reasoning as \cite{WZ24}, although arranged differently. Take a set of tubes $\T$ satisfying the $t$-Frostman Convex Wolff Axiom and let $\T_{\rho} = \calD_{\rho}(\T)$ for each $\rho \in [\de, 1]$. Without loss of generality, we can assume that $\T_{\rho}$ is a set of essentially distinct tubes. If there exists a scale $\rho \in [\de, 1]$ for which $|\T_{\rho}| \gg \rho^{-\beta}$ we arrive at (A). If there exists $\rho$ such that $|\T_{\rho}| \ll \rho^{-  \beta}$, then we arrive at (B) or (C), depending on whether $\T^{T_{\rho}}$ usually satisfies the $t$-Frostman Convex Wolff Axiom or not. Suppose $\T_{\rho} \approx \rho^{-\beta}$ for each $\rho \in [\de, 1]$. If $\T$ satisfies the $t$-Frostman Convex Wolff Axiom at all scales, then we have (D). If $\T$ fails to satisfy the $t$-Frostman Convex Wolff Axiom at all scales, then it must be overconcentrated in some cover of $\T$ by congruent convex sets $\calW$. If $\calW$ was a set of tubes $\T_{\rho}$, it would contradict our assumption that $\T_{\rho} \approx \rho^{-\beta}$, so we know that $\calW$ must be a collection of flat prisms, which leads to (C). This argument combines parts of the multiscale decomposition in \cite[Proposition 4.1]{WZ24} and in \cite[Proposition 6.3]{WZ25} and will be given in detail in Section \ref{sec:multi}.

In Section \ref{sec:proof}, we reduce proving Theorem \ref{thm:realmain} to Proposition \ref{prop:prismsetup}. Proposition \ref{prop:prismsetup} states that if $\T$ is a set of $\de$-tubes and $\calR$ is a set of $\de \times \rho \times 1$ prisms such that $\T[R]$ satisfies the $t$-Frostman Convex Wolff Axiom with small error in $R$ for each $R \in \calR$ and $\calR$ itself satisfies the $t$-Frostman Convex Wolff Axiom with small error, then for some small parameter $\eta$, $\bigcup_{T \in \T} T$ has upper spectrum at $\eta$ and scale $\de$ about $2t+1$. If outcome (C) in the multiscale decomposition occurs, then the $\rho$-neighborhood of $\T$ satisfies the assumptions of Proposition \ref{prop:prismsetup}.

To reduce Theorem \ref{thm:realmain} to Proposition \ref{prop:prismsetup}, we let $\sigma^*$ denote the smallest possible quasi-Assouad dimension of a set of tubes satisfying the $t$-Frostman Convex Wolff Axiom with small error. Let $\beta^*$ denote the largest value of $\log_{1/\de} |\T|$ among $\de$-tubes $\T$ satisfying the $t$-Frostman Convex Wolff Axiom with small error and with quasi-Assouad dimension $\sigma^*$. We take a set of $\de$-tubes $\T$ with quasi-Assouad dimension $\sigma^*$ and multiplicity $\de^{-\beta^*}$ and apply the multiscale decomposition to $\T$. If conclusions (A) or (B) occur, then we easily find a denser set of tubes with quasi-Assouad dimension $\le \sigma^*$, contradicting our assumption that $\T$ is the densest possible quasi-Assouad dimension extremizer. Therefore, (A) or (B) cannot occur. If (C) occurs, then Proposition \ref{prop:prismsetup} implies $\sigma^*$ is about $2t+1$. If (D) occurs, then the hypothesis of Theorem \ref{thm:realmain} implies that $\sigma^*$ is at least $\sigma$. In either case, we have the conclusion of Theorem \ref{thm:realmain}. All that remains is proving Proposition \ref{prop:prismsetup}.

\subsection{Proposition \ref{prop:prismsetup}: Upper spectrum bouds for tubes concentrating in prisms} 

Proposition \ref{prop:prismsetup} states that if $\T$ is a set of $\de$-tubes which satisfies the $t$-Frostman Convex Wolff Axiom with small error in a set of $\de \times \rho \times 1$ prisms $\calR$ and $\calR$ satisfies the $t$-Frostman Convex Wolff Axiom with small error, then $\T$ has upper spectrum at small values of $\eta$ and scale $\de$ about $2t+1$. We would like to apply the induction on scales argument from Steps 2, 3, and 4 in the $t=1$ outline to prove this, but as we noted previously, Steps 3 and 4 break down in the $t < 1$ case. The main innovation of this paper is Proposition \ref{prop:smallprism}, a substantially more difficult argument which proves that $\T$ either has the desired upper spectrum or $\T_{\overline{\de}}$, the $\overline{\de}$-neighborhood of $\T$, can be arranged into $\overline{\de} \times \overline{\De} \times 1$ prisms $\overline{\calR}$ so that $\T_{\overline{\de}}$ satisfies the $1$-Frostman Convex Wolff Axiom with small error in $\overline{\calR}$ and $\overline{\calR}$ satisfies the $t$-Frostman Convex Wolff Axiom. 

We discuss the proof of Proposition \ref{prop:smallprism} in more depth in the following subsection, but before doing so, we note how to close the proof of Proposition \ref{prop:prismsetup} given Proposition \ref{prop:smallprism}. If we apply Proposition \ref{prop:smallprism} to $\T$ and we get the desired upper spectrum for Proposition \ref{prop:prismsetup}, we conclude the proof of Proposition \ref{prop:prismsetup}. Otherwise, we are in a position to follow Steps 2, 3, and 4 from the $t=1$ outline, applied to the sets of tubes $\T_{\overline{\de}}$ and the prisms $\overline{\calR}$. We carry out the induction on scales part of Step 2 in the proof of Proposition \ref{prop:prismsetup} itself and the rest of Steps 2, 3, and 4 in Proposition \ref{prop:largeprism}. The proof of Proposition \ref{prop:prismsetup} makes up Section \ref{sec:prismsetup} and the proof of Proposition \ref{prop:largeprism} makes up Section \ref{sec:largeprism}. All that remains is the proof of Proposition \ref{prop:smallprism}.

\subsection{Proposition \ref{prop:smallprism}: Improving dimensionality of tubes concentrating in prisms}\label{subsec:bpop3}

It remains to prove Proposition \ref{prop:smallprism}. We begin with a set of $\de$-tubes $\T$ covered by a set of $\de \times \De \times 1$ prisms $\calR$, so that $\T[R]$ satisfies the $t$-Frostman Convex Wolff Axiom in $R$ for each $R \in \calR$ and $\calR$ itself satisfies the $t$-Frostman Convex Wolff Axiom. We aim to prove that either $\bigcup_{T \in \T} T$ has upper spectrum at a small value $\eta$ and scale $\de$ about $2t+1$, which we call Conclusion A, or $\T$ arranges as a $1$-dimensional set into larger prisms, which we call Conclusion B. We write $x \in \T[R]$ if $x \in T$ for some $T \in \T[R]$, denote the long direction of $R$ by $\dir(R)$ and the normal direction to $R$ by $n(R)$. We now outline the proof of Proposition \ref{prop:smallprism}. To simplify this outline, we suppress all pigeonholing arguments by assuming any value we would pigeonhole on is everywhere constant. Accounting for pigeonholing introduces significant technical challenges in parts of the proof. While our outline here covers the main ideas of the proof of Proposition \ref{prop:smallprism}, we refer the reader to Section \ref{sec:smallprism} to see where these difficulties arise and how we manage them.

We can find a broadness parameter $u$ and a set of tubes $\T_u$ so that for each $x \in \bigcup_{T \in \T} T$, there exists one tube $T_u$ such that $x \in \T[R]$ for some $R \subset T_u$ and $\{\dir(R) : x \in \T[R], R \in \calR\}$ is $\beta_1$-dimensional within a $u$-ball, where $\beta_1 > 0$ is very small depending on our choice of $\omega$. We do this so that for any direction $\ell \in \{\dir(R) \in S^2 : x \in \T[R], R \in \calR\}$, we can find some $\ell' \in \{\dir(R) \in S^2 : x \in \T[R], R \in \calR\}$ with $\angle(\ell, \ell') \sim u$ and, importantly, this property is essentially maintained under small refinements of $\T$ or $\calR$. 

We then break up each $R \in \calR$ into disjoint $\de \times \De \times \frac{\De}{u}$ subprisms $\calP_R$, each containing a collection of tubes $\T_P = \{T \cap P : T \in \T[R]\}$. Later in the proof, we need $\T_P$ to satisfy the $t$-Frostman Convex Wolff Axiom in $P$ for each $P$, so we establish that property now. Suppose instead that $\T_P$ never satisfies the $t$-Frostman Convex Wolff Axiom. Then we can find a larger scale $\alpha$ such that $\T_P$ overconcentrates in $\calD_{\alpha}(\T_P)$, which means that $\mathcal{W}_P = \calD_{\alpha}(\T_P)$ is quantitatively sparse for each $P$. Summing over $\calP_R$, we conclude that \[\left|\bigcup_{T \in \calD_{\alpha}(\T[R])} T\right|_{\alpha} \ll |\calD_{\alpha}(\T[R])|\alpha^{-1}.\footnote{$\calD_{\alpha}(\T[R])$ is defined in Notation (iii).}\]On the other hand, an $L^2$ argument gives that \[\left|\bigcup_{T \in \calD_{\alpha}(\T[R])} T\right|_{\alpha} \approx |\calD_{\alpha}(\T[R])|\alpha^{-1}.\]We arrive at a contradiction and conclude that $\T_P$ satisfies the $t$-Frostman Convex Wolff Axiom with small error for each $P$. We carry this argument out in more detail in Lemma \ref{lem:grain}. 

Define the affine linear map $\phi_{T_u}:\R^3 \to \R^3$ which sends $T_u$ to the unit cube. If $R \in \calR[T_u]$, then $R^{T_u}:=\phi_{T_u}(R)$ is essentially a $\frac{\de}{u} \times \frac{\De}{u} \times 1$ prism, which has normal direction $n(\phi_{T_u}(R))$ and if $T \subset \T[T_u]$, then $T^{T_u}:=\phi_{T_u}(T)$ is essentially a $\frac{\de}{u}$-tube. We denote the $\calR^{T_u} = \{\phi_{T_u}(R) : R \in \calR[T_u] \}$ and similarly denote other rescaled collections of convex sets. We find a transversality paramters $\theta$ so that for each $x \in \bigcup_{R \in \calR^{T_u}} \bigcup_{T \in \T^{T_u}[R]} T$, the set of directions \[\left\{n(R^{T_u}) : x \in \bigcup_{R \in \calR^{T_u}} \bigcup_{T \in \T^{T_u}[R]} T, R^{T_u} \in \calR^{T_u}\right\} \subset S^2\]is $\beta_2$-dimensional within a $\theta$-ball, where $\beta_2$ is a parameter much smaller than $\beta_1$. 

The theorem now breaks into three paths, depending on the values of $u$ and $\theta$. The first path is the \emph{very transverse case}, when $\theta \approx 1$. In that case, we can break the prisms in $\calR^{T_u}$ into disjoint $\frac{\de}{u} \times \frac{\De}{u} \times \frac{\De}{u}$ small slabs $\calP^{T_u}$. We use a brush argument based on some $P_1 \in \calP^{T_u}$ to find a collection $\mathscr{S}^{T_u}_{P_1} \subset \calP^{T_u}$ satisfying the $t$-Frostman Convex Wolff Axiom in the $\frac{\theta \De}{u} \times \frac{\De}{u} \times \frac{\De}{u}$-prism $G$ concentric with $P_1$. We apply an $L^2$ argument to see that $\T^{T_u}$ concentrates like a $2t+1$-dimensional set in $G$. Since $\theta \approx 1$, $G$ is approximately a $\frac{\De}{u}$-ball, so $\T^{T_u}$ has the desired upper spectrum. This property is preserved after undoing the rescaling through $T_u$, so in the very transverse case, we have the desired upper spectrum bound for $\T$.

The second path is the \emph{very narrow case}, when $u \approx \De$. A consequence of our definition of the broadness parameter $u$ and the associated set of tubes $\T_u$ is that $\T[R]$ and $\T[R']$ are essentially disjoint unless $R$ and $R'$ are contained in the same $u$-tube. Similarly, the definition of the transversality parameter $\theta$ implies that if $R, R' \subset T_u$, $\T^{T_u}[R^{T_u}]$ and $\T^{T_u}[{R'}^{T_u}]$ are essentially disjoint unless $R, R'$ are contained in the same $\theta \times 1 \times 1$ prism. We know that $\calR^{T_u}$ is a collection of essentially $\frac{\de}{u} \times 1 \times 1$ prisms, since $u \approx \De$. For any prism $R^{T_u} \in \calR^{T_u}$, we use a brush argument to prove that $\T^{T_u} \cap P$ looks like a discertized $2t+1$-dimensional set in the $\theta \times 1 \times 1$ prism $P$ concentric with $R^{T_u}$. Carrying this out for each $R^{T_u} \in \calR^{T_u}$, we arrive at a cover of $\calR^{T_u}$ by essentially disjoint $\theta \times 1 \times 1$ prisms $\calR_{\text{env}}^{T_u}$, intersecting $\T^{T_u}$ in pairwise disjoint discretized $2t+1$-dimensional sets.

We undo the rescalings through $u$-tubes to arrive at a collection $\calR_{\text{env}} = \bigcup_{T_u \in \T'_u} \phi_{T_u}^{-1}\calR_{\env}^{T_u}$. Each $\T^{T_u} \cap P$ looks $2t+1$-dimensional and they are all disjoint. We can use the fact that $\calR$-satisfies the $t$-Frostman Convex Wolff Axiom to prove $\calR_{\text{env}}$ is large enough to conclude that \[\bigcup_{T_u \in \T'_u} \bigcup_{P \in \calR_{\text{env}}} \bigcup_{T \in \T[T_u]} T \cap P\]looks $2t+1$-dimensional in $B(0,1)$. This gives the desired upper spectrum bound for $\T$ in the narrow case.

If we assume neither the very transverse nor the very narrow case occurs, we must have that our prisms, and the tubes within them, can be arranged into $\overline{\de} \times \overline{\De} \times 1$ prisms $\overline{\calR}$. We call this the \emph{brush case}. We specifically take $\overline{\de} \approx u\theta$ and $\overline{\De} \approx u$. Since we have assumed the very transverse case does not occur, we know that $\theta \ll 1$, so $\overline{\de} \ll \overline{\De}$. Since the very narrow case does not occur, we know $\De \ll \overline{\De}$. This is close to Step 2 from Wang and Zahl's argument outlined previously, but our scale separations $\overline{\de} \ll \overline{\De}$ and $\De \ll \overline{\De}$ are \emph{not sufficient} to push $\overline{\De}$ to $1$ in finitely many steps of the induction on scales argument in Step 2. We must somehow improve the structure of $\T_{\overline{\de}}$ to succesfully complete the induction on scales argument. 

Fortunately, in the $t < 1$ case, the brush argument used to prove that tubes can be arranged into wider prisms also improves the dimensionality of the tubes within the prisms. When we carry out the brush argument, we have a collection of ``bristles'' $\calR_{\text{bristle}} \subset \calR$ all making angle essentially $\overline{\De} = u \gg \De$ with a ``stem'' $R_{\text{stem}}$ and all contained in a $\overline{\de} \times \overline{\De} \times 1$ prism $\overline{R}$ concentric with $R_{\text{stem}}$. These bristles come from taking prisms incident to one fixed tube $T_{\text{stem}}$ in $\T[R_{\text{stem}}]$ so that the prism approximately intersects $T_{\text{stem}}$ with angle $\approx u$. Then $\overline{\T}(R_{\text{stem}}):=\calD_{\overline{\de}}(\calR_{\text{stem}})$ is a union of $\overline{\de}$-tubes, all intersecting $T_{\text{stem}}$ with angle $\approx u$ and length of intersection $\approx \frac{\de}{u}$ Since essentially all of $T_{\text{stem}}$ must intersect with the tubes in $\overline{\T}(R_{\text{stem}})$, the brush arguments tells us that $\overline{\T}(R_{\text{stem}})$ must satisfy the $1$-Frostman Convex Wolff Axiom in $\overline{R}$. Applying an $L^2$ argument, we conclude that $\bigcup_{T \in \overline{\T}(R_{\text{stem}})} T$ essentially fills out $\overline{R}$.  

At this point, we have almost the same structure as in the $t=1$ case, with a worse scale separation and $\overline{\calR}$ satisfying only the $t$-Frostman Convex Wolff Axiom, rather than the $1$-Frostman Convex Wolff Axiom. The worse scale separation is still acceptable in Wang and Zahl's induction on scales process and replacing the $t$-Frostman Convex Wolff Axiom with the $1$-Frostman Convex Wolff Axiom has no impact on the challenging points in the induction on scales argument, so in Proposition \ref{prop:largeprism}, we follow the same induction on scale steps described Subsection \ref{subsec:bpop1} to conclude the argument.

\section{Discretizing the main theorem}\label{sec:disc}

In this section, we prove that Theorem \ref{thm:realmain} implies Theorem \ref{thm:main}.

\begin{proposition}\label{prop:implication}
    Suppose Theorem \ref{thm:realmain} is true. Then Theorem \ref{thm:main} is true.
\end{proposition}

We briefly sketch the argument before proceeding with the proof. We begin by assuming Theorem \ref{thm:realmain} is true and the hypothesis for Theorem \ref{thm:main} is as well. The hypothesis for Theorem \ref{thm:main} essentially implies the hypothesis for Theorem \ref{thm:realmain}, so we have good control on the discretized upper spectrum of a set of tubes satisfying the $t$-Frostman Convex Wolff Axiom. We would like to use this control to lower bound the quasi-Assouad dimension of a set of line segments $L$ satisfying the continuous $t$-Frostman Convex Wolff Axiom. It suffices to find some value of $\de$ for which a subset of the $\de$-neighborhood of $L$ satisfies the $t$-Frostman Convex Wolff Axiom with small error. Since $L$ satisfies the continuous $t$-Frostman Convex Wolff Axiom, it supports a probability measure $\mu$ such that \[\mu(W) \le K\left(\text{Vol}_{\mathcal{L}}(W)\right)^{t/2} \text{ for any geodesically convex } W \subset \mathcal{L}.\]Following similar reasoning to the proof of Frostman's lemma, we can find a collection of $\de$-tubes contained in the $\de$-neighborhood of $L$ which satisfy the $t$-Frostman Convex Wolff Axiom with small error. Ensuring the $\de$-neighborhood of $L$ satisfies the $t$-Frostman Convex Wolff Axiom is where we use the continuous $t$-Frostman Convex Wolff Axiom. This part of the proof involves some technical work. 

Before beginning the proof of Proposition \ref{prop:implication}, we introduce a piece of notation in this proof which we use in the proof and throughout the remainder of the paper. 
\begin{definition}
    Recall that we have defined $\calW[R] = \{W \in \calW : W \subset 2R\}$. We denote by $\calW(R)$ some subset of $\calW[R]$. 
\end{definition}
\begin{remark}
    It is harmless to assume that any $W \in \calW(R)$ is in fact contained in $R$. Accounting for this change would occasionally introduce constant-factor errors in some bounds which never impact any of the final results.
\end{remark}
Now, we prove Proposition \ref{prop:implication}.
\begin{proof}
    Suppose Theorem \ref{thm:realmain} is true. Suppose furthermore that for any $\e > 0$, there exists $\eta, \de_0 > 0$ such that any sticky $(1,2t)$-Furstenberg set at scale $\de < \de_0$ with error $\de^{-\eta}$ has quasi-Assouad dimension at scale $\de$ at least $\sigma - \e$. Unwinding definitions, we see that this implies that for any $\e > 0$, there exists $\eta, \de_0 > 0$ such that if a set of tubes $\T$ satisfies the $t$-Frostman Convex Wolff Axiom at all scales with error $\le \de^{-\eta}$, then $\bigcup_{T \in \T} T$ has discretized upper spectrum at least $\sigma - \e$ at $\eta$ and scale $\de$. 
    
    Let $E$ be a $(1,2t)$-Furstenberg set and $L$ be the associated set of lines which satisfies the continuous $t$-Frostman Convex Wolff Axiom with error $K$. We claim that for any $\eta> 0$, there exists $\de_0 > 0$ such that for any $\de < \de_0$, $N_{\de}(E)$ contains a collection of $\de$-tubes satisfying the $t$-Frostman Convex Wolff Axiom with error $\de^{-\eta}$. Before proving this claim, we use it to complete the proof of Proposition \ref{prop:implication}. 
    
    It suffices to prove that for any $\omega > 0$, we can choose $\eta > 0$ such that for any $(1,2t)$-Furstenberg set $E$ and scales $r < r^{\eta}R$ and an $R$-ball $B_R$ with \[|E \cap B_R|_r > \left(\frac{R}{r}\right)^{\sigma - \omega}.\]
    
    If this is not the case, then for any scales $r, R$ with $r^{1-\eta} \le R$ and any $R$-ball $B_R$, \begin{equation}\label{eqn:origin}|E \cap B_R|_r <\left(\frac{R}{r}\right)^{\sigma-\omega}.\end{equation}By the conclusion of Theorem \ref{thm:realmain}, we know there exist $\eta, \de'_0 > 0$ such that if $\de < \de'_0$ and $\T$ is a collection of $\de$-tubes satisfying the $t$-Frostman Convex Wolff Axiom with error $\de^{-\eta}$, then $\bigcup_{T \in \T} T$ has upper spectrum at $\eta$ and scale $\de'_0$ at least $\sigma - \omega$. By our claim, for any $\de > 0$ with $\de < \min(\de'_0, \de_0)$ such that $N_{\de}(E)$ contains a collection of $\de$-tubes $\T$ satisfying the $t$-Frostman Convex Wolff Axiom with error $\de^{-\eta}$, $\bigcup_{T \in \T} T$ has upper spectrum at $\eta$ at least $\sigma - \omega$. Then there exist scales $r, R$ with $\de \le r \le \de^{-\eta}R$ and an $R$-ball $B_R$ such that \[|E \cap B_R|_r \ge \left(\frac{R}{r}\right)^{\sigma - \omega}.\]Since $r \ge \de$, $r^{1-\eta} \le R$, and we contradict (\ref{eqn:origin}). We conclude that $E$ has upper spctrum at $\eta$ at least $\sigma - \omega$. As $\omega$ was arbitrary, we see that $\dim_{qA}(E) \ge \sigma$.

    It remains to prove the claim. Fix a dyadic scale $\de > 0$, which we eventually require to be small enough so that $\de^{-\eta} \gtrapprox_{\de} 1$. Let $L$ be the set of lines constituting $E$ and $\mu$ be the probability measure supported on $L$ guaranteed in Definition \ref{def:cts}. Cover $\mathcal{L}$ with disjoint dyadic $\de$-cubes. We can refine $E$ to some $E_1$ with $\mu(E_1) \gtrsim \mu(E)$ so that the set of dyadic $\de$-cubes intersecting $\mu(E')$ are pairwise $3\de$-separated. Denote by $\mathcal{Q}$ the set of dyadic $\de$-cubes incident to $E'$. For $k \in \N$, let $Q_k = \{ Q \in \mathcal{Q}: \mu(Q) \in [2^{-k}, 2^{-k+1})\}$. We easily see that since $\mu$ is a probability measure, $\sum_{2^{-k} < \de^{100}} \sum_{Q \in \mathcal{Q}_k} \mu(Q) \le \frac{\mu(E_1)}{2}$. There are $\lessapprox_{\de} 1$ terms in that sum, so after pigeonholing, we can find one choice of $k$ such that $\sum_{Q \in \mathcal{Q}_k} \mu(Q) \gtrapprox_{\de} \mu(E_2) \gtrsim \mu(E) = 1$. For each $Q \in \mathcal{Q}_k$, we have some $\ell_Q \in L \cap Q$. Let $T_Q$ denote the tube formed by taking the $\de$-neighborhood of the line segment in $E$ corresponding to $\ell_Q$ and $\T = \{T_Q : Q \in \mathcal{Q}_k\}$. Clearly, $\T \subset N_{\de}(E)$ and since the elements of $\mathcal{Q}_k$ are $3\de$-separated, $\T$ is an essentially distinct set of tubes. Now, we must prove that it satisfies the $t$-Frostman Convex Wolff Axiom with error $\le \de^{-\eta}$, so long as $\de$ is sufficiently small.  

    If $\T$ fails the $t$-Frostman Convex Wolff Axiom with error $\de^{-\eta}$, then we must have some convex set $W$ with \begin{equation}\label{eqn:lab1}|\T[W]| \gtrsim \de^{-\eta}\text{Vol}(W)^t|\T|.\end{equation}We can assume $W$ is almost minimal, in the sense that any convex set $W' \subset W$ with $\T[W'] = \T[W]$ has $\vol(W') \ge \frac{1}{2} \vol(W)$. Let $\tilde{W} = \{\ell \in \mathcal{L} : \ell \subset W\}$ and \begin{equation*}\tilde{\T} = \{T \in \T : T = T_Q \text{ for some Q such that } Q \cap \tilde{W} \neq \emptyset\}.\end{equation*}Let $\ell, \ell'$ be line segments respectively paramterized as in the definition of $\mathcal{L}$ (see (\ref{eqn:calL})) by $(u,v,w,x)$ and $(u',v',w',x')$. We claim that a convex combination $t_1(u,v,w,x) + (1-t_1)(u',v',w',x') \in \mathcal{L}$ corresponds to a line $\ell_{t_1} \subset W$. To see this, note that each point $z_{t_1} = (t_1(u,v,0) + (1-t_1)(u', v', 0)) + s(t_1(w,x,1) + (1-t_1)(w', x', 0)) \in \ell_{t_1}$ is a convex combination of $z = (u,v,0) + s(w,x,1)$ and $z' = (u', v', 0) + s(w',x', 1)$. As $W$ is convex and $z, z' \in W$, $z_{t_1} \in W$ as well. As $z_{t_1}$ is arbitrary, it follows that $\ell_{t_1} \subset W$. In the flat metric on $\mathcal{L}$, an element $(u^*, v^*, w^*, x^*)$ is on the geodesic between $(u,v,w,x)$ and $(u',v',w',x')$ precisely when $(u^*, v^*, w^*, x^*) = t_1(u,v,w,x) + (1-t_1)(u',v',w',x')$, so we conclude that $\tilde{W}$ is geodesically convex. 
    
    Next, we prove that that $\text{Vol}_{\mathcal{L}}(\tilde{W}) \lesssim \text{Vol}(W)^2$. Let $E(W)$ be the outer John ellipsoid of $W$, $W_{\text{bottom}} = E(W) \cap \{(x,y,z) : z = 0\}$ and $W_{\text{top}} = E(W) \cap \{(x,y,z) : z = 1\}$. Any ellipsoid is the image of the unit ball under a linear map $A$, so parallel slices are formed by $A$ applied to parallel slices of the unit ball, which are always similar. All linear maps commute with translations and dilations, so the image of the slices under $A$ are similar as well. It follows that $W_{\text{bottom}}$ and $W_{\text{top}}$ are similar. Without loss of generality, assume that $W_{\text{bottom}}$ is smaller than $W_{\text{top}}$, in the sense that $W_{\text{bottom}}$ is $W_{\text{top}}$ rescaled by some scalar $\kappa < 1$. Expand $W_{\text{bottom}}$ to $W'_{\text{bottom}}$ which is congruent to $W_{\text{top}}$ and let $x_{\text{bottom}}$ be a point in $W_{\text{bottom}}$. Let $K$ be the convex hull of $W'_{\text{bottom}} \cup W_{\text{top}}$ and $C$ the convex hull of $\{x_{\text{bottom}}\} \cup W_{\text{top}}$. Note that $K$ is a cylinder with base $W_{\text{top}}$ and height $1$ and hence $\vol(K) = \vol_{\R^2}(W_{\text{top}})$, and $C$ is a cone with base $W_{\text{top}}$ and height $1$ and hence $\vol(C) \sim \vol_{\R^2}(W_{\text{top}})$. Finally, we know that $W \supset C$ and since $\T[K] \supset \T[W]$, by the minimality of $W$ in (\ref{eqn:lab1}), we must have that $\vol(K) \gtrsim \vol(W)$. Since $\vol(K) \sim \vol(C)$, we conclude that $\vol(K) \sim \vol(W)$. It is immediate from definitions that \[\text{Vol}_{\mathcal{L}}(\tilde{W}) \lesssim \text{Vol}_{\R^2}(W_{\text{bottom}})\text{Vol}_{\R^2}(W_{\text{top}}) \le \vol(K)^2 \sim \vol(W)^2,\]as desired.

    By our assumption that $\mu$ satisfies (\ref{eqn:bear}), we know that \begin{equation}\label{eqn:lab2}\mu(\tilde{W}) \le K~\text{Vol}_{\mathcal{L}}(\tilde{W})^{t/2} \lesssim \vol(W)^t.\end{equation}On the other hand, $\mu(\tilde{W}) \sim |\T[W]|2^k$ and $|\T| \sim 2^k\sum_{Q \in \mathcal{Q}_k} \mu(Q) \gtrapprox_{\de} 2^k$. Combining this estimate with (\ref{eqn:lab2}), we conclude that $|\T[W]| \lessapprox_{\de, K} \vol(W)^t|\T|$. So long as $\de$ is sufficiently small, we contradict (\ref{eqn:lab1}). We conclude that $\T$ satisfies the $t$-Frostman Convex Wolff Axiom with error $\de^{-\eta}$, completing the proof.
\end{proof}

\section{A multiscale decomposition for unions of tubes}\label{sec:multi}

In this section, we prove Proposition \ref{prop:WZ256.3}, a multiscale decomposition for collections of $\de$-tubes $\T$ satisfying the $t$-Frostman Convex Wolff Axiom with small error. This decomposition has four possible outcomes. One of the outcomes is that $\T$ satisfies the $t$-Frostman Convex Wolff Axiom at all scales with small error, so we before stating the multiscale decomposition, we recall Definition \ref{def:FCWAAAS}.
\begin{definition*}[\textbf{\ref{def:FCWAAAS}}]
    We say $\T$ satisfies the \emph{$t$-Frostman Convex Wolff Axiom at all scales with error $C$} if for any $\rho_0 \in [\de, 1]$, there exists some $\rho \in [\rho_0, C\rho_0]$ such that each $T \in \T$ is contained in a $\rho$-tube $T_{\rho}$ and $\T[T_{\rho}] := \{T \in \T:T \subset 2T_{\rho}\}$ satisfies the $t$-Frostman Convex Wolff Axiom with error $C$ in $T_{\rho}$.
\end{definition*}

Now we state Proposition \ref{prop:WZ256.3}. This proposition starts with a several rounds of quantifier choice. The way we choose quantifiers in the statement is aligned with how we apply Proposition \ref{prop:WZ256.3} in the proof of Theorem \ref{thm:realmain}. 

\begin{proposition}\label{prop:WZ256.3}
    For any $\beta \in [2t,4]$ and $\zeta > 0$, there exists $\alpha > 0$ such that the following holds. For any $\eta_1 > \eta_2 > 0$ with $\eta_1 < \zeta/2$ and any $\de'_0 > 0$, there exists $\eta, \de_0 > 0$ such that the following holds for all $\de \in (0, \de'_0)$. Let $\T$ be a set of $\de$-tubes satisfying the $t$-Frostman Convex Wolff Axiom with error $\le \de^{-\eta_1}$ and with $|\T| \in [\de^{-\beta+\alpha}, \de^{-\beta}]$. Then there exists $\T' \subset \T$ such that at least one of the following is true. 

    \begin{enumerate}[label = (\Alph*)]
        \item There exists a scale $\rho \in (\de, \min(\de'_0, \de^{\eta}))$ and some collection $\T'_{\rho} \subset \calD_{\rho}(\T)$ such that $|\T'_{\rho}| \ge \rho^{-\beta - \alpha}$ and $\T'_{\rho}$ satisfies the $t$-Frostman Convex Wolff Axiom with error $\rho^{-\eta_1}$.
        \item There exists a scale $\rho \in [\max(\de/\de'_0, \de^{1-\eta}), 1]$, a $\rho$-tube $T_{\rho}$, and a set of tubes $\T_{\ast} \subset \T'[T_{\rho}]$ such that $|\T_{\ast}| \ge \left(\frac{\de}{\rho}\right)^{-\beta-\alpha}$ and $\T_{\ast}^{T_{\rho}}$ satisfies the $t$-Frostman Convex Wolff Axiom with error $(\de/\rho)^{-\eta_1}$.
        \item There exist $\de \le \rho < \De \le 1$ with $\rho \le \de^{\eta_1/100}\De$ and $\rho < \de'_0$ and a set $\calW$ of $\rho \times \De \times 1$ prisms which factor $\T'$ from above and below with respect to the $t$-Frostman Convex Wolff Axiom with error $\le \rho^{-\eta_2}$. 
        \item $\T'$ satisfies the $t$-Frostman Convex Wolff Axiom at all scales with error $\de^{-\zeta}$.
    \end{enumerate}
\end{proposition}

In order to prove Proposition \ref{prop:WZ256.3}, we further discuss non-concentration conditions for collections of convex sets in Subsection \ref{subsec:nonconc}. 

\subsection{Non-concentration conditions for collections of convex sets.}\label{subsec:nonconc}

We begin with the $t$-Katz-Tao Convex Wolff Axiom, naturally generalizing the Katz-Tao Convex Wolff Axiom defined in \cite{WZ25}. 

\begin{definition}\label{def:KTCWA}
    A collection $\mathcal{U}$ of congruent convex sets satisfies the $t$-Katz-Tao Convex Wolff Axiom with error $C \ge 1$ if for any convex set $V$\[|\mathcal{U}[V]| \le C\left(\frac{\vol(V)}{\vol(U)}\right)^t.\]

    We denote the smallest constant $C$ for which this inequality holds by $C_{t\text{-KT-CW}}(\mathcal{U})$.
\end{definition}

Similarly to how the $t$-Frostman Convex Wolff Axiom is related to the $2t$-Frostman measure condition, the $t$-Katz-Tao Convex Wolff Axiom is related to the $(\de, 2t, C)$ Katz-Tao set condition, a well-known planar non-concentration condition originally from \cite[Refinement 2.2]{KT02}. These definitions coincide for a family of congruent convex sets $\calU$ when $|\calU| = \vol(U)^{-1}$, while the Frostman variant allows denser configurations of convex sets and the Katz-Tao variant allows sparser configurations. We focus on the Frostman Convex Wolff Axiom in this paper, but there are some points where it is convenient to consider the Katz-Tao Convex Wolff Axiom as well.

We would like to use these conditions at many scales. Before we can do this, we have to carefully define some ways of changing the scale of convex sets. We can most easily do so when the convex sets involved are prisms, so we use convex sets and prisms interchangeably in the remainder of the proof. We can do so because we can approximate convex sets with prisms from both inside and outside with only a constant factor loss in volume, which is harmless in our applications. This approximation follows essentially by the John ellipsoid theorem, which says that we can approximate any convex set $K$ from the inside by an ellipse $E_{\text{I}} \subset K$ and from the outside by an ellipse $E_{\text{O}} \supset K$ such that $|E_{\text{I}}| \sim |E_{\text{O}}|$. We can then approximate $E_{\text{I}}$ from the inside with a prism of approximately equal measure and $E_{\text{O}}$ from the outside with a prism of approximately equal measure. It is easiest to see this when the ellipses are spheres; the general case follows by applying an isotropic dilation, a rotation, and then a translation. For consistency, we also define tubes to be $\de \times \de \times 1$ prisms, which is an irrelevant change from the usual definition as the $\de$-neighborhood of a line segment. 

One way to change the scale of our prisms is to take neighborhoods of prisms. It will be convenient to do so in such a way that the neighborhood of a prism is also a prism. Up to small constant factors, we know that the $r$-neighborhood of an $a \times b \times c$ prism is an $(a+r) \times (b+r) \times (c+r)$ prism, which up to constant factors again is a $\max(a,r)  \times \max(b,r) \times \max(c,r)$ prism. For an $a \times b \times c$ prism $W$, we define $N_r(W)$ to be the concentric (that is, with the same short, medium, and long direction) $\max(a,r) \times \max(b,r) \times \max(c,r)$ prism. We define, for a family of prisms $\calW$, $\calD_{r}(\calW) = \{N_r(W) :W \in \calW\}$. 

Another way to change the scale is to rescale prisms through larger prisms. For a prism $W$, define $\phi_W$ to be the affine linear map sending $3W$ to $[0,1]^3$. For a set of prisms $\calR$, define $\calR^W = \phi_W(\calR[W])$. Since $\phi_W$ is angle-preserving, it sends prisms to prisms. In the cases we consider in this paper, the resulting prisms are congruent and we can easily find their dimensions\footnote{In general, this is not the case; see \cite[Figure 7]{WZ25} for an example.}.

Finally, we say that a set of prisms $\calR$ covers a set $\calW$ if each $W \in \calW$ is contained in $3R$ for some $R \in \calR$. Clearly, $\calD_{\de}(\calR)$ covers $\calR$ for any $\de \ge 0$. We state two conditions regarding covers, both originating from \cite{WZ25}.

\begin{definition}
    Let $\calW$ and $\calR$ be collections of congruent convex sets.
    \begin{itemize}
        \item $\calR$ is a $C$-balanced cover of $\calW$ if $\calR$ is a cover of $\calW$ and $|\calW[R_1]| \le C |\calW[R_2]|$ for any $R_1, R_2 \in \calR$.
        \item $\calR$ is a $C$-partitioning cover of $\calW$ if $\calR$ is a cover of $\calW$ and for each $W \in \calW$, $|\{R \in \calR : W \in \calW[R]\}| \le C$.
    \end{itemize}
\end{definition}

To simplify the relationship between the Convex Wolff Axioms with covers, we use Wang and Zahl's definition of factoring \cite{WZ25}.
\begin{definition}\cite[Definition 4.4]{WZ25}
    Suppose $\calR$ covers $\calW$.
    \begin{itemize}
        \item We say $\calR$ factors $\calW$ from above with respect to the $t$-Frostman (resp. Katz-Tao) Convex Wolff Axiom with error $C$ if $\calR$ satisfies the $t$-Frostman (resp. Katz-Tao) Convex Wolff Axiom with error $C$.
        \item We say $\calR$ factors $\calW$ from below with respect to the $t$-Frostman (resp. Katz-Tao) Convex Wolff Axiom with error $C$ if for each $R \in \calR$, $\calW^R$ satisfies the $t$-Frostman (resp. Katz-Tao) Convex Wolff Axiom with error $C$.
    \end{itemize}
\end{definition}

We have four easy consequences of the definitions of the Convex Wolff Axioms which play an important role in the remainder of the paper. The first and third consequences appear in \cite{WZ25} and the second in \cite{WZ24}. 

\begin{proposition}\label{prop:inheritance}
    Suppose $\calR$ covers $\calW$. 
    \begin{enumerate}
        \item If the cover is $C$-balanced and $C$-partitioning and $\calW$ satisfies the $t$-Frostman Convex Wolff Axiom with error $C'$, then $\calR$ satisfies the $t$-Frostman Convex Wolff Axiom with error $C^2C'$. 
        \item If $\calW[R] \neq \emptyset$ for each $R \in \calR$, $|\calR| = |\calW|$, and $\calW$ satisfies the $t$-Frostman Convex Wolff Axiom with error $C$, then $\calR$ satisfies the $t$-Frostman Convex Wolff Axiom with error $C$. In particular, if $\calR = \mathcal{D}_{r}(\calW)$ for some $r > 0$, then $\calR$ satisfies the $t$-Frostman Convex Wolff Axiom with error $C$.
        \item If $\calW$ satisfies the $t$-Katz Tao Convex Wolff Axiom with error $C'$, then for each $R \in \calR$, $\calW^R$ satisfies the $t$-Katz-Tao Convex Wolff Axiom with error $C'$. 
        \item If $\calW[R]$ satisfies the $t$-Frostman Convex Wolff Axiom in $R$ with error $C$, then $\calW^R$ satisfies the $t$-Frostman Convex Wolff Axiom with error $C$.
    \end{enumerate}
\end{proposition}
We use the following terminology from \cite{WZ25} to refer to Proposition \ref{prop:inheritance} (1) and (2). 

\begin{definition}\cite[Remark 4.3]{WZ25}
    If we apply Proposition \ref{prop:inheritance} (1) or (2), we say the $t$-Frostman Convex Wolff Axiom is ``inherited upwards''. 
\end{definition}

An important tool in our proof of Proposition \ref{prop:WZ256.3} is Lemma \ref{lem:WZ2546}, which appears in the $t = 1$ case in \cite[Lemma 4.6]{WZ25}. They only state the result when $t = 1$, but their proof applies equally well in the case that $t \in (0, 1]$ needed in this paper. 

\begin{lemma}[{\cite[Lemma 4.6]{WZ25}}]\label{lem:WZ2546}
    Suppose $t \in (0, 1]$. Let $\mathcal{U}$ be a finite set of congruent convex subsets of the unit ball in $\R^d$, each of which contains a ball of radius $\de$ such that $|\mathcal{U}| \le \de^{-100}$\footnote{ As long as $|\mathcal{U}| \le \de^{-K}$ for some value of $K$, the error terms in the rest of the proposition are $\lessapprox_{\de} 1$. We can assume any family of objects we apply Lemma \ref{lem:WZ2546} to satisfies that assumption.}. Then there is a set $\mathcal{W}$ of congruent convex subsets of $\R^3$ and a set $\mathcal{U}' \subset \mathcal{U}$ with the following properties.

    \begin{enumerate}[label = (\roman*)]
        \item $|\mathcal{U}'| \gtrapprox_{\de} |\mathcal{U}|$.
        \item $\mathcal{W}$ is a $\lessapprox_{\de} 1$-balanced, $\lessapprox_{\de} 1$-almost partitioning cover of $\mathcal{U}'$ and \begin{equation}\label{eqn:WZ2546}|\mathcal{U}'[W]| \gtrapprox_{\de} C_{t\text{-KT-CW}}(\mathcal{U}')(\vol(W)/\vol(U))^t.\end{equation}
        \item $\mathcal{W}$ factors $\mathcal{U}'$ from above with respect to the $t$-Katz-Tao Convex Wolff Axiom with error $\lessapprox_{\de} 1$.
        \item $\mathcal{W}$ factors $\mathcal{U}'$ from below with respect to the $t$-Frostman Convex Wolff Axiom with error $\lessapprox_{\de} 1$.
    \end{enumerate}
\end{lemma}

The $t$-Katz-Tao Convex Wolff Axiom, defined in \ref{def:KTCWA}, plays a small role in some of our applications of Lemma \ref{lem:WZ2546}. If $|\mathcal{U}| = \vol(U)^{-1}$, then the $t$-Katz-Tao Convex Wolff Axiom and $t$-Frostman Convex Wolff Axiom conditions are identical. In some cases, when we apply Lemma \ref{lem:WZ2546}, we make sure that $|\mathcal{U}| = \vol(U)^{-1}$ before doing so and conclude that if $\mathcal{U}$ does \emph{not} satisfy the $t$-Frostman Convex Wolff Axiom, then (ii) in Lemma \ref{lem:WZ2546} implies that the resulting collection of convex sets $\calW$ must be substantially sparser than $\mathcal{U}$. One way we can turn this qualitative statement into a quantitative bound is by rearranging (\ref{eqn:WZ2546}) and using that $|\mathcal{U}'[W]| \approx_{\de} \frac{|\mathcal{U}'|}{|\mathcal{W}|}$. We conclude that
\begin{equation}\label{eqn:WZ2546+} |\mathcal{W}| \lessapprox_{\de} C_{t\text{-KT-CW}}(\mathcal{U}')^{-1}|\mathcal{U'}|(\vol(W)/\vol(U))^{-t}.\end{equation}

For the interested reader, the Convex Wolff Axioms are discussed in more detail in \cite{WZ25}. This concludes our discussion about the non-concentration conditions for convex sets. We proceed to the proof of Proposition \ref{prop:WZ256.3}.

\subsection{The proof of Proposition \ref{prop:WZ256.3}}

Recall the statement of Proposition \ref{prop:WZ256.3}.

\begin{proposition*}[\ref{prop:WZ256.3}]
    For any $\beta \in [2t,4]$ and $\zeta > 0$, there exists $\alpha > 0$ such that the following holds. For any $\eta_1 > \eta_2 > 0$ with $\eta_1 < \zeta/2$ and any $\de'_0 > 0$, there exists $\eta, \de_0 > 0$ such that the following holds for all $\de \in (0, \de'_0)$. Let $\T$ be a set of $\de$-tubes satisfying the $t$-Frostman Convex Wolff Axiom with error $\le \de^{-\eta_1}$ and with $|\T| \in [\de^{-\beta+\alpha}, \de^{-\beta}]$. Then there exists $\T' \subset \T$ such that at least one of the following is true. 

    \begin{enumerate}[label = (\Alph*)]
        \item There exists a scale $\rho \in (\de, \min(\de'_0, \de^{\eta}))$ and some collection $\T'_{\rho} \subset \calD_{\rho}(\T)$ such that $|\T'_{\rho}| \ge \rho^{-\beta - \alpha}$ and $\T'_{\rho}$ satisfies the $t$-Frostman Convex Wolff Axiom with error $\rho^{-\eta_1}$.
        \item There exists a scale $\rho \in [\max(\de/\de'_0, \de^{1-\eta}), 1]$, a $\rho$-tube $T_{\rho}$, and a set of tubes $\T_{\ast} \subset \T'[T_{\rho}]$ such that $|\T_{\ast}| \ge \left(\frac{\de}{\rho}\right)^{-\beta-\alpha}$ and $\T_{\ast}^{T_{\rho}}$ satisfies the $t$-Frostman Convex Wolff Axiom with error $(\de/\rho)^{-\eta_1}$.
        \item There exist $\de \le \rho < \De \le 1$ with $\rho \le \de^{\eta_1/100}\De$ and $\rho < \de'_0$ and a set $\calW$ of $\rho \times \De \times 1$ prisms which factor $\T'$ from above and below with respect to the $t$-Frostman Convex Wolff Axiom with error $\le \rho^{-\eta_2}$. 
        \item $\T'$ satisfies the $t$-Frostman Convex Wolff Axiom at all scales with error $\de^{-\zeta}$.
    \end{enumerate}
\end{proposition*}

Our proof of Proposition \ref{prop:WZ256.3} combines the reasoning in \cite[Proposition 4.1]{WZ24} and \cite[Proposition 6.3]{WZ25} but is adapted to easily fit into our proof of Theorem \ref{thm:realmain}. To prove Proposition \ref{prop:WZ256.3}, we start with a set of $\de$ tubes $\T$ and break the range of scales $[\de, 1]$ into $\de^{\eta}$-separated scales $\de_N = \de, \dots, \de_0 = 1$. At each scale $\de_j$, we first confirm that $|\T_{\de_j}| \le \de_j^{-\beta - \alpha}$, as otherwise we would be in case (A), and then we confirm that $|\T^{T_{\de_j}}| \le \left(\frac{\de}{\de_j}\right)^{-\beta - \alpha}$, as otherwise we would be in case (B) or (C). Once we have checked both cases for each $\de_j$, we apply Lemma \ref{lem:WZ2546} to $\T^{T_{\de_j}}$ for each $j$. If the resulting convex sets are essentially the unit ball for each $j$, then we are in case (D). The convex sets cannot be tubes, as that would contradict our assumptions on the size of $\T^{T_{\de_j}}$ for some $j$. The final alternative is that the convex sets are flat prisms for some $j$, in which case we arrive in case (C). 

We now proceed with the proof.
\begin{proof}[Proof of Proposition \ref{prop:WZ256.3}]
    We begin by briefly explaining how parameters will be chosen. The statement calls for a value $\alpha > 0$. We choose this value depending only on $\zeta$ over the course of the proof. We choose $\eta \in \frac{1}{\N}$ and write $N = \frac{1}{\eta}$. We allow $\eta$ to depend on all parameters except $\de_0$ and $\de_0$ to depend on all parameters. 
    
    Now, we proceed with the multiscale decomposition. Let $\de_j = \de^{j \eta}$ for $j = 0, \dots, N$. We call these \emph{usable scales}. The separation between usable scales is not very important, all we need is that it is much smaller than $\zeta$, so $\eta$ suffices. Through a standard iterative pigeonholing argument (see for example \cite[Lemma 3.4]{KS19} or \cite[Lemma 3.6]{S19}), we can find a collection $\T' \subset \T$ with $|\T'| \gtrapprox_{\de} |\T|$ such that for each $j = 1, \dots, N$, there exists a set of tubes $\T_{\de_j}$ such that $\T_{\de_j}$ is a $\lessapprox_{\de} 1$ balanced, partitioning cover of $\T'$ and $\T_{\de_{N}} = \T'$

    Suppose that $|\T_{\de_k}| \ge \de_k^{-\beta - \alpha}$ for some $k \in \{1, \dots, N-1\}$. We know that $\T_{\de_k} \subset \calD_{\de_k}(\T')$. Since the $t$-Frostman Convex Wolff Axiom is inherited upwards, we know that $\T_{\de_k}$ satisfies the $t$-Frostman Convex Wolff Axiom with error $\lessapprox_{\de} \de^{-\eta}$ for each $k$. As long as $\eta < \frac{\eta_1}{2}$ and $\de_0$ is sufficiently small, we conclude that $\T_{\de_k}$ satisfies the $t$-Frostman Convex Wolff Axiom with error $\le \de^{-\eta_1} \le \de_k^{-\eta_1}$. We choose $\de_0$ sufficiently small relative to $\eta$ so that $\de^{\eta} < \de_0^{\eta} < \de'_0$. Since $k \ge 1$, $\de_{k} \le \de^{\eta} < \de'_0$ and hence $\rho := \de_k \le \min(\de^{\eta}, \de'_0)$. We arrive in outcome (A).

    Now, suppose that $|\T_{\de_k}| \le \de^{\eta}\left(\frac{\de}{\de_k}\right)^{\beta + 2\alpha}|\T'|$ for some $k \in \{1, \dots, N-1\}$. Since $\T_{\de_k}$ is a $\lessapprox_{\de} 1$ partitioning, balanced cover of $\T$, each $T_{\de_k} \in \T_{\de_k}$ satisfies $|\T^{T_{\de_k}}| \gtrapprox_{\de} \frac{|\T|}{|\T_{\de_k}|} \ge \de^{-\eta} \left(\frac{\de}{\de_k}\right)^{-\beta - 2\alpha}$. Taking $\de$ sufficiently small depending on $\eta$, we have $|\T^{T_{\de_k}}| \ge \left(\frac{\de}{\de_k}\right)^{-\beta - 2\alpha}$ for each $T_{\de_k} \in \T_{\de_k}$. Apply Lemma \ref{lem:WZ2546} to $\T^{T_{\de_k}}$ for each $T_{\de_k} \in \T_{\de_k}$. This returns some $\T_{\ast}^{T_{\de_k}}$ and dimensions $a_{T_{\de_k}} \times b_{T_{\de_k}} \times 1$ and a set $\calW_{T_{\de_k}}$ of $a_{T_{\de_k}} \times b_{T_{\de_k}} \times 1$ prisms which factor $T_{\de_k}$ from below with respect to the $t$-Frostman Convex Wolff Axiom with error $\lessapprox_{\de} 1$. After pigeonholing on $\T_{\de_k}$ to some $\T^*_{\de_k}$, we may assume that $a, b$ are constants. Since $|\T^{\ast}_{\de_k}| \gtrapprox_{\de} |\T_{\de_k}|$ and $|\T_{\ast}^{T_{\de_k}}| \gtrapprox_{\de} |\T_{\de_N}^{T_{\de_k}}|$, if $\T_{\ast} \subset \T_{\de_N}$ is the set of tubes contained in $\T^{\ast}_{\de_k}$ for some $T_{\de_k}$, then $|\T_{\ast}| \gtrapprox_{\de} |\T_{\de_N}|$.  

    We now have two cases to consider depending on the relative values of $a$ and $b$. First, suppose that $a < \de^{\eta_1/100}b$. Recall that $\T_{\ast}^{T_{\de_k}} = \phi_{T_{\de_k}}(\T_{\ast}[T_{\de_k}])$. Denote by $\tilde{\calW}[T_{\de_k}] = \phi_{T_{\de_k}}^{-1}(\calW_{T_{\de_k}})$. We denote $\tilde{a} = a\de_k, \tilde{b} = b\de_k$ and see that $\tilde{\calW}[T_{\de_k}]$ is a family of $\tilde{a} \times \tilde{b} \times 1$ prisms covering $\T_{\ast}[T_{\de_k}]$. We would like a family of prisms covering $\T_{\ast}$, so we define $\tilde{\calW} = \bigcup_{T_{\de_k} \in \T_{\de_k}} \tilde{\calW}[T_{\de_k}]$. Since $\tilde{\calW}[T_{\de_k}]$ is a $\lessapprox_{\de} 1$-balanced, partitioning cover of $\T_{\ast}[T_{\de_k}]$ and $\T_{\de_k}$ is a $\lessapprox_{\de} 1$ partitioning cover of $\T^{\ast}$, we know that $\tilde{\calW}$ is a partitioning cover of $\T_*$. After a harmless round of pigeonholing, we can assume as well that it is a $\sim 1$-balanced cover as well. 
    
    Each $W \in \tilde{\calW}[T_{\de_k}]$ only contains tubes from $\T_{\ast}[T_{\de_k}]$, so since $\left(\T_{\ast}^{T_{\de_k}}\right)^W$ satisfies the $t$-Frostman Convex Wolff Axiom with error $\lessapprox_{\de} 1$ and hence (so long as $\de_0$ is sufficiently small) $\le \tilde{a}^{-\eta_2}$, $\T_{\ast}^W$ does as well. Since $\T$ satisfies the $t$-Frostman Convex Wolff Axiom with error $\le \de^{-\eta}$ and $|\T_{\ast}| \gtrapprox_{\de} |\T|$, $\T_{\ast}$ satisfies the $t$-Frostman Convex Wolff Axiom with error $\lessapprox_{\de} \de^{-\eta}$. We know that $\tilde{a} < a < \de^{\eta_1/100}b < \de^{\eta_1/100}$, so $\tilde{a}^{-\eta_2} > \de^{-\eta_1\eta_2/100}$, so as long as $\eta$ is sufficiently small relative to $\eta_1, \eta_2$ and $\de_0$ is sufficiently small, $\T_{\ast}$ satisfies the $t$-Frostman Convex Wolff Axiom with error $\lessapprox_{\de} \tilde{a}^{-\eta_2/2}$. In the previous paragraph, we established that $\tilde{\calW}$ is a partitioning cover of $\T_*$, so since the $t$-Frostman Convex Wolff Axiom is inherited upwards, after again cancelling out a logarithmic error, $\tilde{\calW}$ satisfies the $t$-Frostman Convex Wolff Axiom with error $\le \tilde{a}^{-\eta_2}$. Finally, we need $\tilde{a} < \de'_0$, but since $\tilde{a} < \de^{\eta_1/100} < \de_0^{\eta_1/100}$, we can accomplish this by taking $\de_0$ very small. We arrive in case (C) by taking $\tilde{\calW}$ as our set of prisms, $\rho = \tilde{a}$ and $\De = \tilde{b}$. 

    This handles the $a < \de^{\eta_1/100}b$ case. Now, suppose $a > \de^{\eta_1/100}b$. We aim to reach (B). Let $\de_j$ be the smallest usable scale greater than $\de_kb$. Since $b \le 1$, we have $j \ge k$ (recall that smaller scales correspond to larger values of $j$) and since usable scales are $\de^{\eta}$ separated, we know that $\frac{\de_j}{\de_k} \le \de^{-\eta}b$. Fix some choice of $T_{\de_k}$ such that \begin{equation}\label{eqn:sheep1}\left(\frac{\de}{\de_k}\right)^{-\beta - 2\alpha} \le \frac{|\T'|}{|\T_{\de_k}|} \approx \frac{|\T_{\ast}|}{|\T_{\de_k}|} \le |\T_{\ast}[T_{\de_k}]|.\end{equation}We claim that there is a $\de_j$-tube $T_{\de_j} \subset T_{\de_k}$ such that $\T^{T_{\de_j}}$ satisfies the conclusions of (B). First, recall that $\calW^{T_{\de_k}}$ satisfies the $t$-Katz-Tao Convex Wolff Axiom with error $\lessapprox_{\de} 1$. It follows that \begin{equation}\label{eqn:sheep2}|\calW^{T_{\de_k}}| \lessapprox_{\de} (ab)^t \le \left(\frac{\de_j}{\de_k}\right)^{-2t}.\end{equation}We can therefore cover $\calW[T_{\de_k}]$ with a set of $\de_j$-tubes $\T_{\de_j}$ with $|\T_{\de_j}| \le \left(\frac{\de_j}{\de_k}\right)^{-2t}$ by taking the $\de_j$ neighborhood of each $W \in \calW[T_{\de_k}]$. By comparing (\ref{eqn:sheep1}) with (\ref{eqn:sheep2}), we see that at least one of the tubes in $\T_{\de_j}$ satisfies \[|\T_{\ast}[T_{\de_j}]| \ge \frac{|\T_{\ast}[T_{\de_k}]|}{|\T_{\de_j}|} \gtrapprox \left(\frac{\de}{\de_k}\right)^{-\beta - 2\alpha}\left(\frac{\de_j}{\de_k}\right)^{2t} \ge \left(\frac{\de}{\de_j}\right)^{-\beta-2\alpha}.\]If we take $\de_0$ sufficiently small, we conclude that $|\T_{\ast}[T_{\de_j}]| \ge \left(\frac{\de}{\de_j}\right)^{-\beta-\alpha}$, as desired.
    Since $\de_j \ge \de^{1-\eta}$, $\de_j > \max\left(\frac{\de}{\de'_0}, \de^{1-\eta}\right)$ so long as $\de_0$ is sufficiently small. To reach (B), it remains to prove that $\T_{\ast}^{T_{\de_j}}$ satisfies the $t$-Frostman Convex Wolff Axiom with error $\left(\frac{\de}{\de_j}\right)^{-\eta_1}$. Recall that each $T \in \T_{\ast}[T_{\de_j}]$ is contained in some $W \in \calW[T_{\de_j}]$. Since $\calW[T_{\de_k}]$ satisfies the $t$-Katz-Tao Convex Wolff Axiom with error $\lessapprox_{\de} 1$ in $T_{\de_k}$, we have \[|\calW[T_{\de_j}]| \le \left(\frac{\vol(T_{\de_j})}{\vol(W)}\right)^t \left(\frac{\de_j^2}{ab \de_k^2}\right)^t\]for each $T_{\de_j}$. We know that $\frac{\de_j}{b\de_k} \le \de^{-\eta}$ by construction, and since $a \ge \de^{\eta_1/100}b$, we also have $\frac{\de_j}{a\de_k} \le \de^{-\eta-\eta_1/100}$. Therefore, by taking $\eta$ small relative to $\eta_1$, we can bound \[|\calW[T_{\de_j}]| \le \left(\frac{\de_j^2}{ab \de_k^2}\right)^t \le \de^{-2t\eta -t\eta_1/100} \le \de^{-\eta_1/10}.\]For any convex set $V \subset B(0,1)$, \[|\T^{T_{\de_j}}[V]| \le \sum_{W \in \calW[T_{\de_j}]} \T_*[W \cap V] \lessapprox \de^{-\eta_1/10} |\T^*[W]|\vol(W \cap V)^t \le \left(\frac{\de}{\de_j}\right)^{-\eta_1}\vol(V)^t|\T_*[T_{\de_j}]|.\]We have the desired error in the $t$-Frostman Convex Wolff Axiom of $\T^{T_{\de_j}}$, so this completes the argument in the case that $|\T_{\de_k}| \le \de^{\eta}\left(\frac{\de}{\de_k}\right)^{\beta + 2\alpha}|\T'|$.

    Now we conclude that for each $j \in \{1, \dots, N\}$, we have \begin{equation}\label{eqn:ADish}|\T_{\de_j}| \ge \de^{\eta}\left(\frac{\de}{\de_k}\right)^{\beta + 2\alpha}|\T'|\quad\text{ and }\quad|\T_{\de_j}| \le \de_j^{-\beta - \alpha}.\end{equation}We now apply an iterative refinement process to this set of tubes, starting from $i = N$ and proceeding towards $1$. At some point, this iterative process may end and we arrive at one of case (B), (C), or (D). If we reach the $i$th step of the iterative process, we have for each $j \le i$ a collection of tubes $\tilde{\T}_j$ such that for each $j < i$, $\tilde{\T}_{j}$ is a $\lessapprox_{\de} 1$ partitioning cover of $\tilde{\T}$ which factors it from above with respect to the $t$-Frostman Convex Wolff Axiom with error $\le \de^{-\zeta}$ and $|\tilde{\T}_{j}| \gtrapprox_{\de} |\T_{j}|$. 

    Suppose we reach step $i$ of this iterative process. For each $T_{\de_{i-1}} \in \T_{\de_{i-1}}$, apply Lemma \ref{lem:WZ2546} to $\hat{\T} = \tilde{\T}^{T_{\de_{i-1}}}$. This returns a subcollection $\hat{\T}_1$ with $|\hat{\T}_1| \gtrapprox_{\de} |\hat{\T}|$ and a cover $\calW_{T_{\de_{i-1}}}$ which factors $\hat{\T}_1$ from above with respect to the $t$-Frostman Convex Wolff Axiom with error $\lessapprox_{\de} 1$. We denote the dimensions of $\calW_{T_{\de_{i-1}}}$ by $a_{T_{\de_{i-1}}} \times b_{T_{\de_{i-1}}} \times 1$. After refining $\T_{\de_{i-1}}$ to some $\tilde{\T}_{\de_{i-1}}$ with $|\tilde{\T}_{\de_{i-1}}| \gtrapprox_{\de} |\T_{\de_{i-1}}|$, we may assume that $a_{T_{\de_{i-1}}} \in [a_i/2, a_i]$ and $b_{T_{\de_{i-1}}} \in [b_i/2, b_i]$ for fixed values $a_i, b_i$. We expand the prisms in each family $\calW_{T_{\de_{i-1}}}$ to all have dimensions $a_i \times b_i \times 1$ at the cost of a $\lesssim 1$ factor increase in the error of the $t$-Frostman Convex Wolff Axiom. 

    We retained in $\tilde{\T}_{\de_{i-1}}$ a $\gtrapprox_{\de} 1$ fraction of the tubes from $\T_{\de_{i-1}}$. We refine $\tilde{\T}$ to consists of only tubes covered by elements of $\tilde{\T}_{\de_j}$, which similarly retains a $\gtrapprox_{\de} 1$ fraction of the number of tubes from $\tilde{\T}$. We also defined a subset $\hat{\T}_1 \subset \T^{T_{\de_{i-1}}}$ with $|\hat{\T}_1|\gtrapprox_{\de} |T^{T_{\de_{i-1}}}|$ for each $T_{\de_{i-1}} \in \tilde{\T}_{\de_i}$. These two refinements induce a refinement of $\tilde{\T}$ which retains a $\gtrapprox_{\de} 1$ share of the tubes. We continue to refer to the remaining set of tubes as $\tilde{\T}$ and this set of tubes still satisfies $|\tilde{\T}| \gtrapprox_{\de} |\T|$. Finally, for each $j \ge i-1$, $\tilde{\T}_{\de_j}$ is still a $\lessapprox_{\de} 1$ partitioning cover of $\tilde{\T}$.

    We have several cases to consider depending on the values of $a_i, b_i$ and $\beta$. First, suppose that $|a_ib_i| \ge \de^{\zeta/2}$. For each $T_{\de_{i-1}} \in \T_{\de_{i-1}}$, we have a $\lessapprox_{\de} 1$ partitioning cover by $a \times b\times 1$ prisms $\calW_{T_{\de_{i-1}}}$ which factors $\tilde{\T}^{T_{\de_{i-1}}}$ from above with respect to the $t$-Frostman Convex Wolff Axiom with error $\lessapprox_{\de} 1$. This implies that $\tilde{\T}^{T_{\de_{i-1}}}[W]$ satisfies the $t$-Frostman Convex Wolff Axiom with error $\lessapprox_{\de} |a_ib_i|^{-t} \le \de^{-\zeta/2}$. Since this holds for each $W \in \calW_{T_{\de_{i-1}}}$, which is a $\lessapprox_{\de} 1$ partitioning cover of $\tilde{\T}^{T_{\de_{i-1}}}$, we have the same fact (with a $\lessapprox_{\de} 1$ increase in error) for $\tilde{\T}^{T_{\de_{i-1}}}$. So long as $\de$ is small enough relative to the implicit constant, we have that $\tilde{\T}_{\de_{i-1}}$ factors $\tilde{\T}$ from above with respect to the $t$-Frostman Convex Wolff Axiom with error $\de^{-2\zeta/3}$. 

    Before considering other cases for the values of $a_i, b_i$, we note how to close the argument and reach case (D) in the case that $|a_ib_i| \ge \de^{\zeta/2}$ at each step. We have a some $\tilde{\T} \subset \T_{\de_N}$ with $|\tilde{\T}| \gtrapprox_{\de} |\T_{\de_N}|$ and for each $\rho > \de$, there exists some $\de_{j} \in [\rho, \de^{-\zeta}\rho]$ and a $\de^{-\zeta}$-partitioning cover $\tilde{\T}_{\de_j}$ of $\tilde{\T}$ which factors $\tilde{\T}$ from above with respect to the $t$-Frostman Convex Wolff Axiom with error $\de^{-2\zeta/3}$. It remains to confirm that the $\tilde{\T}_{\de_j}$ are balanced. This is not true, but by repeatedly refining each scale to be a $\le \de^{-\zeta}$-balanced cover of the previous scale, we can refine $\tilde{\T}$ by a $\lessapprox_{\de} 1$ factor and restore this property, at the cost of a $\lessapprox_{\de} 1$ increase in the error in the $t$-Frostman Convex Wolff Axiom. Choosing $\de_0$ sufficiently small depending on $\eta$, we can ensure that the errors in the $t$-Frostman Convex Wolff Axiom are always $\le \de^{-\zeta}$ and we reach case (D).

    Therefore, we suppose for the remainder of the proof that $|a_ib_i| < \de^{\zeta/2}$. We have two cases to consider. First, we suppose that $a_i < \de^{\eta_1/100}b_i$. We reach case (C) by the same reasoning as earlier in the proof when we were in this situation. 
    
    Now, suppose $a_i > \de^{\eta_1/100}b_i$ and $\beta > 2t$. Recall that $\calW_{T_{\de_{i-1}}}$ factors $\tilde{\T}^{T_{\de_{i-1}}}$ from below with respect to the $t$-Katz-Tao Convex Wolff Axiom with error $\lessapprox_{\de} 1$. It follows from the definition of the $t$-Katz-Tao Convex Wolff Axiom that \begin{equation}\label{eqn:crane3}|\calW_{T_{\de_{i-1}}}| \le (a_ib_i)^{-t} < \de^{\eta_1/100}a_i^{-2t}.\end{equation}We define a cover $\tilde{\calW} = \bigcup_{T_{\de_{i-1}} \in \T_{\de_{i-1}}} \phi_{T_{\de_{i-1}}}^{-1}(\calW_{T_{\de_{i-1}}})$ consisting of $a_i\de_{i-1} \times b_i\de_{i-1} \times 1$ prisms and note that \begin{equation}\label{eqn:crane4}|\tilde{\calW}| \le \sum_{T_{\de_{i-1}} \in \T_{\de_{i-1}}} |\calW_{T_{\de_{i-1}}}| \le |\T_{\de_{i-1}}| \de^{\eta_1/100}a_i^{-2t}.\end{equation}Let $\tilde{a}$ be the largest usable scale $< a_i\de_{i-1}$. By (\ref{eqn:ADish}), we see that \begin{equation}\label{eqn:crane1}|\T_{\tilde{a}}| \ge \de^{\eta}\left(\frac{\de}{\tilde{a}}\right)^{\beta+2\alpha}|\T'| \ge \de^{\eta + 3\alpha}\tilde{a}^{-\beta-2\alpha}.\end{equation}We chose $\tilde{a} \ge \de^{\eta}a_i\de_{i-1}$. On the other hand, $\T_{\tilde{a}}$ consists of essentially distinct tubes, and there are $\lesssim \frac{a_i^2b_i^2\de^4_{i-1}}{\tilde{a}^4} \le \de^{-4\eta-\eta_1/50}$ many elements of $\T_{\tilde{a}}$ contained in the $\tilde{a}$ neighborhood of any $W \in \tilde{\calW}$. This bound, (\ref{eqn:crane4}), and (\ref{eqn:ADish}) combine to show that \begin{equation}\label{eqn:crane2}|\T_{\tilde{a}}| \lesssim |\T_{\de_{i-1}}|\de^{-4\eta-\eta_1/100}a_i^{-2t}\le \de^{-4\eta-\eta_1/100}\tilde{a}^{-2t}\de^{2t}_{i-1}|\T_{\de_{i-1}}| \le \de^{-4\eta-\eta_1/100}\tilde{a}^{-2t}\de^{-\beta - \alpha + 2t}_{i-1}.\end{equation}Comparing (\ref{eqn:crane1}) with (\ref{eqn:crane2}), we conclude that \[\de^{\eta + 3\alpha}\tilde{a}^{-\beta-2\alpha + 2t} \lesssim \de^{-4\eta-\eta_1/100}\de^{-\beta - \alpha + 2t}_{i-1} \le \de^{-4\eta-\eta_1/100}\de^{-\beta - 2\alpha + 2t}_{i-1}.\]We have assumed that $\eta_1 < \alpha$, and we take $\eta < \eta_1/10$, so this rearranges to $\tilde{a} \gtrsim \de^{5\alpha/(\beta + 2\alpha - 2t)}\de_{i-1}$. So long as $\alpha$ is sufficiently small relative to $\zeta$ and $\beta - 2t$, we see that $\tilde{a} \ge \de^{\zeta/6}\de_{i-1}$. We conclude that (so long as $\eta > 0$ is sufficiently small), $a_{i} \ge \de^{\zeta/5}$ and the same for $b_i$. Multiplying, we have $a_ib_i > \de^{\zeta/2}$, contradicting our assumption that $a_ib_i < \de^{\zeta/2}$.

    It remains to treat the $\beta \le 2t$ case. We know that $\beta > 2t - \eta$ (that is, $|\T| \ge \de^{-2t+\eta}$) since $\T$ satisfies the $t$-Frostman Convex Wolff Axiom with error $\de^{-\eta}$. It follows that $|\tilde{\T}| \ge \de^{-2t + 2\eta}$, so long as $\de_0$ is sufficiently small. We follow essentially the same reasoning as in the previous paragraph, except we need a better bound in (\ref{eqn:crane3}). To find this bound, we need to recall how we defined the prisms $\calW_{T_{\de_{i-1}}}$ in the first place. We would like to prove that $\tilde{\T}^{T_{\de_{i-1}}}$ satisfies the $t$-Frostman Convex Wolff Axiom with error $\le \de^{-\zeta}$. If $|\tilde{\T}^{T_{\de_{i-1}}}| \ge \de^{\zeta/2}\left(\frac{\de}{\de_{i-1}}\right)^{-2t}$, then it suffices to prove that $\tilde{\T}^{T_{\de_{i-1}}}$ satisfies the $t$-Katz-Tao Convex Wolff Axiom with error $\le \de^{-\zeta/2}$. Since $|\T_{\de_{i-1}}| \le \de_{i-1}^{-2t - \alpha}$, we see that $|\tilde{\T}^{T_{\de_{i-1}}}| \approx \frac{|\tilde{\T}|}{|\T_{\de_{i-1}}|} \ge \de^{\zeta/2}\left(\frac{\de}{\de_{i-1}}\right)^{-2t}$, so long as $\alpha, \eta$ are sufficiently small relative to $\zeta$. We can therefore assume that $C_{t-\text{KT-CW}}(\tilde{\T}^{T_{\de_{i-1}}}) \ge \de^{-\zeta/2}$. Using (\ref{eqn:WZ2546+}), we conclude that \[|\calW_{T_{\de_{i-1}}}| \le \de^{\zeta/2}(a_ib_i)^{-t}(\de/\de_{i-1})^{2t}|\tilde{\T}^{T_{\de_{i-1}}}|.\]Since we have assumed that $|\T_{\de_{i-1}}| \ge \de^{\eta}\left(\frac{\de}{\de_{i-1}}\right)^{2t+2\alpha-\eta}|\tilde{\T}|$, we have that \[|\tilde{\T}^{T_{\de_{i-1}}}| \lesssim \de^{-\eta}\left(\frac{\de}{\de_{i-1}}\right)^{-2t-2\alpha+\eta}.\]Therefore, \[|\calW_{T_{\de_{i-1}}}| \le \de^{\zeta/2-\eta}(a_ib_i)^{-t}\left(\frac{\de}{\de_{i-1}}\right)^{-2\alpha} \le \de^{\zeta/2-\eta-2\alpha}(a_ib_i)^{-t}.\]Choosing $\alpha, \eta$ sufficiently small, we can ensure that this is $\le \de^{\zeta/3}(a_ib_i)^{-t}$. We use this bound instead in (\ref{eqn:crane3}), follow the same reasoning in that paragraph, and conclude that so long as all parameters are sufficiently small relative to $\zeta$, $\tilde{a}^{-2\alpha+\eta} \le \de^{\zeta/5}\de_{i-1}^{-2\alpha + \eta}$. Recalling that $\tilde{a} \le \de_{i-1}$, if $\eta < -2\alpha$, this is a contradiction. At last, this completes the argument.
\end{proof}

\section{Reducing to the case of tubes concentrating in flat prisms}\label{sec:proof}

In this section, we reduce Theorem \ref{thm:realmain} to Proposition \ref{prop:prismsetup}, an upper spectrum bound for unions of tubes which concentrate in prisms. The statement of Proposition \ref{prop:prismsetup} refers to a shading, so we define this first.

\begin{definition}
    For an $\de \times \rho \times 1$ convex set $W$, we say that $Y(W)$ is a \emph{shading} on $W$ if $Y(W)$ is a union of dyadic $\de$-cubes incident to $W$. We define the \emph{full shading} $Y_{\text{full}}$ on $W$ to be the set of all dyadic $\de$-cubes incident to $W$. Since $\de \le a$, $Y(W) \subset 2W$. We say that $Y(W)$ is $\ge C$ dense if $|Y(W)|_{\de} \ge C|W|_{\de}$. 

    We often write $(\calW, Y)$ to denote that $Y$ is a shading on $\calW$. 
\end{definition}
\begin{remark}
    The definition of $Y_{\text{full}}(W)$ ensures that $|Y_{\text{full}}(W)|_{\de} = |W|_{\de}$. It is harmless to assume that $Y(W)$ is contained in $W$, $Y_{\text{full}}(W) = W$, and any $W \in \calW[R]$ is contained in $R$. We do so for the remainder of this paper. Accounting for this change would occasionally introduce constant-factor errors in some bounds which never impact any of the final results.
\end{remark}

We are now ready to state Proposition \ref{prop:prismsetup}.
\begin{proposition}\label{prop:prismsetup}
    For any $\omega, \e > 0$, there exists $\eta, \de_0 > 0$ such that the following holds for all $\de \in (0, \de_0)$. Suppose $\T$ is a collection of $\de$-tubes satisfying the $t$-Frostman Convex Wolff Axiom with error $\de^{-\eta}$, $Y$ is a shading $\de^{\eta}$ dense on $\T$, $\calR$ is a collection of $\de \times \De \times 1$ prisms for some $\De \ge \de^{1-\e}$, and $\calR$ factors $\T$ from above and below with respect to the $t$-Frostman Convex Wolff Axiom with error $\de^{-\eta}$. 
    
    Then $\bigcup_{T \in \T} Y(T)$ has upper spectrum at $\eta$ and scale $\de$ at least $2t+1 - \omega$. 
\end{proposition}
For the purpose of the proof of Theorem \ref{thm:realmain}, we could only consider the shading $Y_{\text{full}}$ in Proposition \ref{prop:prismsetup}, but treating general cases may be useful in other applications and does not add much difficult in the proof. The proof of Proposition \ref{prop:prismsetup} will comprise the final four sections of the paper. 

For the proof of Theorem \ref{thm:realmain} itself, we follow a similar outline to that Wang and Zahl's argument in the $t=1$ case \cite[Section 5]{WZ24}. We assume that any set of tubes $\T$ satisfying the $t$-Frostman Convex Wolff Axiom at all scales has quasi-Assouad dimension at least $\sigma$. We let $\sigma^*$ denote the smallest posssible quasi-Assouad dimension of a set of tubes $\T$ satisfying the $t$-Frostman Convex Wolff Axiom. If $\sigma^* \ge \sigma$, are done, so we suppose otherwise. Then, we set $\beta$ to be the largest value such that there exists a set of $\de$-tubes $\T$ with quasi-Assouad dimension $\sigma^*$ and $|\T| = \de^{-\beta}$. We take $\T$ to be a set of tubes with quasi-Assouad dimension about $\sigma^*$ and $|\T| \approx \de^{-\beta}$ and apply Proposition \ref{prop:WZ256.3} to $\T$. 

We have four options to consider. Options (A) gives us a cover of $\T$ by $\rho$-tubes $\T_{\rho}$ satisfying the $t$-Frostman Convex Wolff Axiom and with $|\T_{\rho}| \gg \rho^{-\beta}$. Option (B) gives a $\rho$-tube $T_{\rho}$ such that $\T^{T_{\rho}}$ satisfies the $t$-Frostman Convex Wolff Axiom and with $|\T^{T_{\rho}}| \gg \left(\frac{\de}{\rho}\right)^{-\beta}$. We will see shortly that the quasi-dimension of $\T_{\rho}$ and of $\T^{T_{\rho}}$ must not exceed the quasi-Assouad dimension of $\T$, so either (A) or (B) would contradict the extremality of the cardinality of $\T$. 

Option (C) tells us that there is a set of flat prisms $\calR$ for which $\calR$, $\T$, and the full shading $Y_{\text{full}}$ satisfy the assumptions of Proposition \ref{prop:prismsetup}. The conclusion of Proposition \ref{prop:prismsetup} is that $\bigcup_{T \in \T} T$ has quasi-Assouad dimension $2t+1 \ge \sigma$. We contradict our assumption that $\sigma^* < \sigma$. Option (D) tells us that $\T$ is sticky, in which case it has quasi-Assouad dimension at least $\sigma$, similarly a contradiction. We conclude that $\sigma^* \ge \sigma$, which implies the conclusion of Theorem \ref{thm:realmain}.

Before we begin with a formal proof, we justify our claim that for a set of tubes $\T$, $\T_{\rho}$ and $\T^{T_{\rho}}$ both have quasi-Assouad dimension at most that of $\T$. We actually need a more quantitative statement than this, in terms of the upper spectrum of the various sets of tubes. Unfortunately, must accept a very small error for the upper spectrum $\T^{T_{\rho}}$, but this will be irrelevant in the proof of Theorem \ref{thm:realmain}. 

\begin{lemma}\label{lem:qa}
    For any set of $\de$-tubes $\T$, the following holds:

    \begin{enumerate}[label = (\Alph*)]
        \item For any $\rho > \de$, the upper spectrum of $\T$ at $\eta_{\de}$ and scale $\de$ is greater than the upper spectrum of $\calD_{\rho}(\T)$ at $\eta_{\rho}$ and scale $\rho$, where $\de^{\eta_{\de}} = \rho^{\eta_{\rho}}$. 
        \item There exists an absolute constant $C \ge 1$ such that the following holds. For any $\rho > \de$, let $\de^{\eta_{\de}} = \left(\frac{\de}{\rho}\right)^{\eta_{\de/\rho}}$. For any $\rho$-tube $T_{\rho}$, if the upper spectrum of $\T^{T_{\rho}}$ at $\eta_{\de/\rho}$ and scale $\frac{\de}{\rho}$ is at least $\sigma$, then the upper spectrum of $\T$ at $\eta_{\de}$ and scale $\de$ is at least $\sigma - \frac{\log_{1/\de}(C)}{\eta_{\de}}$.
    \end{enumerate}
\end{lemma}
\begin{remark}
    Item (B) is a clear point where the quasi-Assouad dimension is better suited for incidence problems than the Assouad dimension. If instead of assuming control on the quasi-Assouad dimension of sticky sets, we assumed control on the Assouad dimension of sticky sets in Theorem \ref{thm:main}, we would not have the scale separation $\eta_{\de} > 0$ and an analog of item (B) without some form of scale separation seems impossible. 
\end{remark}

Item (A) follows immediately from the definition of the upper spectrum, so we only prove Item (B).

\begin{proof}[Proof of (B)]
    Suppose the upper spectrum of $\T^{T_{\rho}}$ at $\eta_{\de/\rho}$ and scale $\frac{\de}{\rho}$ is at least $\sigma$. The definition of upper spectrum gives scales $r < \left(\frac{\de}{\rho}\right)^{\eta} R$ and an $R$-ball $B_R$ such that $|\T^{T_{\rho}} \cap B_R|_{r} \ge \left(\frac{R}{r}\right)^{\sigma}$. Let $E = \phi_{T_{\rho}}^{-1}\left(\T^{T_{\rho}} \cap B_R\right)$. This maps an $r$-cube to a $\rho r \times \rho r \times r$ prism. Each such prism incides with $\gtrsim \rho^{-1}$ dyadic $\rho r$-cubes, so $|E|_{\rho r} \gtrsim \rho^{-1}|\T^{T_{\rho}} \cap B_R|_{r}$. On the other hand, $B_R$ is mapped to an a $\rho R \times \rho R \times R$ prism, so by pigeonholing we can find a $\rho R$-ball such $B_{\rho R}$ such that \begin{equation}\label{eqn:rescale}|E \cap B_{\rho R}|_r \gtrsim \left(\frac{\rho R}{\rho r}\right)^{\sigma}.\end{equation}Since $R/r \ge \left(\frac{\de}{\rho}\right)^{-\eta_{\de/\rho}} = \de^{-\eta_{\de}}$, we have the desired spacing on the scales. Letting $1/C$ be the implicit constant in (\ref{eqn:rescale}), with a bit of arithmetic we arrive at the desired spectrum.
\end{proof}

Finally, we complete the proof of Theorem \ref{thm:realmain}, assuming Proposition \ref{prop:prismsetup}. Before doing so, we recall the statement of Theorem \ref{thm:realmain}.

\begin{theorem*}[\textbf{\ref{thm:realmain}}]
    Suppose that for any $\e > 0$, there exists $\eta, \de_0 > 0$ such that for any $\de \in (0, \de_0)$ and any set of $\de$-tubes $\T$ satisfying the $t$-Frostman Convex Wolff Axiom at all scales with error $\de^{-\eta}$, $\bigcup_{T \in \T} T$ has upper spectrum at $\eta$ and scale $\de$ at least $\sigma - \e$. 

    Then for any $\omega > 0$, there exists $\eta, \de_0 > 0$ such that for any $\de \in (0, \de_0)$ and any set of $\de$-tubes $\T$ satisfying the $t$-Frostman Convex Wolff Axiom with error $\de^{-\eta}$, $\bigcup_{T \in \T} T$ has upper spectrum at $\eta$ and scale $\de$ at least $\sigma - \omega$.
\end{theorem*}

\begin{proof}[Proof of Theorem \ref{thm:realmain}]

    We fix $\omega > 0$ to be the value of $\omega$ we aim for in Theorem \ref{thm:realmain}. All parameters are allowed to depend on $\omega$. 

    We assume that there exists values $\zeta, \De_0 > 0$ such that if $\de < \De_0$ and $\T$ is a collection of $\de$-tubes satisfying the $t$-Frostman Convex Wolff Axiom at all scales with error $\de^{-\zeta}$, then $\T$ has upper spectrum at $\zeta$ and scale $\de$ at least $\sigma - \omega/2$. 
    
    Now, we set up the extremization argument to prove the same result for general sets of tubes satisfying the $t$-Frostman Convex Wolff Axiom. For a set of $\de$-tubes $\T$, define \[\sigma^*(\T, \eta) = \inf\{\sigma^* : \T \text{ has upper spectrum at } \eta \text{ and scale } \de \text{ at least } \sigma^*.\}\] Define $\sigma^{*}(\beta, \eta, \de_0) = \inf \sigma^*(\T, \eta),$where the infimum is taken over sets of $\de$-tubes $\T$ such that the following three conditions hold:

    \begin{enumerate}[label = (\roman*)]
        \item $\de \le \de_0$.
        \item $\T$ satisfies the $t$-Frostman Convex Wolff Axiom with error $\de^{-\eta}$.
        \item $|\T| \ge \de^{-\beta+2\eta}$
    \end{enumerate}

    We apply Proposition \ref{prop:WZ256.3} to close the argument. As the statement of Proposition \ref{prop:WZ256.3} involves several layers of quantifiers, to clarify our approach we briefly summarize how we apply this proposition. We assume that sets of tubes satisfying the $t$-Frostman Convex Wolff Axiom with arbitrarily small errors can have upper spectrum $< \sigma$. To apply Proposition \ref{prop:WZ256.3}, we need parameters $\beta$ and $\zeta$. We have already defined $\zeta$. We choose $\beta$ so that $\de^{-\beta}$ is the largest possible cardinality of such a set of tubes satisfying the $t$-Frostman Convex Wolff Axiom with small error and with the smallest possible upper spectrum among all tubes satisfying the $t$-Frostman Convex Wolff Axiom. Proposition \ref{prop:WZ256.3} tell us that for these paramters, there exists a value $\alpha > 0$ which works in the remainder of Proposition \ref{prop:WZ256.3}. For a given set of tubes $\T$ with small errors (more on this later) and cardinality about $\de^{-\beta}$, the first two outcomes of Proposition \ref{prop:WZ256.3} tell us that we can find a set of $\rho$-tubes $\T_1$ with upper spectrum at most that of $\T$ and $|\T_1| \gg \rho^{-\beta}$. Assuming that $\T_1$ satisfies the $t$-Frostman Convex Wolff Axiom with small error, we contradict our assumption about the minimality of $\beta$. We determine $\eta_1> 0$ sufficiently small to apply our assumption about the minimality of $\beta$ for any such $\T_1$; $\eta_2$ small enough to apply Proposition \ref{prop:prismsetup} with $\e = \eta_1$; and $\de'_0$ small enough to work in both. These three parameters gives a choice of parameters $\eta, \de_0 > 0$ which determine the maximum scale and error in the $t$-Frostman Convex Wolff Axiom for $\T$. 
    
    Now, we take a set of $\de$-tubes $\T$ which satisfies the $t$-Frostman Convex Wolff Axiom with error at most $\de^{-\eta}$, $\de \le \de_0$, $|\T| \approx \de^{-\beta}$, and upper spectrum $< \sigma$. We apply Proposition \ref{prop:WZ256.3} to this set of tubes. The previous paragraph describes what we do in cases (A) or (B); in either case, we reach a contradiction. As long as $\de'_0$ is sufficiently small, we can also use the good bounds from Proposition \ref{prop:prismsetup} for tubes concentrating in prisms in case (C) to prove that the upper spectrum of the set of tubes is about $2t+1 \ge \sigma$. As we have assumed $\T$ has upper spectrum $< \sigma$, this is impossible. And if (D) occurs, we use our assumption about the smallest possible upper spectrum of sets of tubes satisfying the $t$-Frostman Convex Wolff Axiom at all scales with error $\de^{-\zeta}$ to conclude that $\T$ has upper spectrum $\ge \sigma$, again a contradiction. As those were the only four possibilities, our assumption that sets of tubes satisfying the $t$-Frostman Convex Wolff Axiom with arbitrarily small error can have upper spectrum $< \sigma$ must be false, and hence the conclusion of Theorem \ref{thm:realmain} is true. This concludes our summary, we now proceed to fill in the details.
    
    Let $\sigma^*(\beta) = \inf_{\eta, \de_0 > 0} \sigma^*(\beta, \eta, \de_0)$ and $\sigma^* = \sigma^*(2t)$. If $\sigma^* \ge \sigma - 3\omega/4$, we have the desired result, so suppose otherwise. Let $\beta^* = \sup \{\beta \in [2t,4] :  \sigma^*(\beta) = \sigma^*\}$. Proposition \ref{prop:WZ256.3} tells us that for this value of $\beta$, there exists a value of $\alpha > 0$ satisfying the conclusions of Proposition \ref{prop:WZ256.3}. We know that $\sigma^*(\beta^* + \alpha/2) > \sigma^*$, so for some $\eta_1 < \alpha/2$ and $\hat{\de}_0 > 0$, $\sigma^*(\beta^* + \alpha/2, \eta_1, \hat{\de}_0) > \sigma^*$. Apply Proposition \ref{prop:prismsetup} with $\e = \eta_1/200$. We see that there are values $\eta_2 > 0$, $\overline{\de}_0 > 0$ such that the conclusions of the proposition hold. Let $\de'_0 = \min(\hat{\de}_0, \overline{\de}_0/2)$. Let $\eta, \de_0$ be the values given by Proposition \ref{prop:WZ256.3} for our choice of $\eta_1, \eta_2$ and $\de'_0$. By construction, we have that for any $\beta \in (\beta^* - \alpha, \beta^*)$, $\sigma^*(\beta) = \sigma^*$, so $\sigma^*(\beta, \eta, \de_0) \le \sigma^*$. Take such a value of $\beta$ and write $\e = \sigma^*(\beta^* + \alpha/2, \eta_1, \hat{\de}_0) - \sigma^*(\beta, \eta\eta_1, \de_0)$. For some $\de < \de_0$, can find a set of $\de$-tubes $\T$ satisfying the $t$-Frostman Convex Wolff Axiom with error $\de^{-\eta}$ so that $|\T| \ge \de^{-\beta}$, and $\T$ has does not have upper spectrum at $\eta$ and scale $\de$ at least $\sigma^* + \min(\e/2, \omega/4)$, since if we cannot find such a set of tubes, then $\sigma^*(\beta, \eta\eta_1, \de_0) > \sigma^*$, a contradiction.
    
    We finally apply Proposition \ref{prop:WZ256.3} to $\T$. Let $\T'$ denote the resulting set of tubes. We have four potential conclusions which we treat in turn. The first is conclusion (A): there exists a value of $\rho \le \hat{\de}_0$ and a collection of $\rho$-tubes $\T_{\rho}$ covering $\T'$ which satisfies the $t$-Frostman Convex Wolff Axiom with error $\rho^{-\eta_1}$ and with $|\T_{\rho}| \ge \rho^{-\beta - \alpha}$. Since $\sigma^*(\beta^* + \alpha/2, \eta_1, \de'_0) \ge \sigma^* + \e$, we conclude that $\T_{\rho}$ has upper spectrum at $\eta_1$ and scale $\rho$ at least $\sigma^* + \e$. Since $\rho < \de^{\eta}$, $\rho^{\eta_1} < \de^{\eta \eta_1}$, and hence $\T'$ has upper spectrum at $\eta\eta_1$ and scale $\de$ at least $\sigma^* + \e$ as well, by Lemma \ref{lem:qa}. The upper spectrum is monotonic, so the same holds for $\T$, contradicting our assumption otherwise. Hence, (A) cannot occur.

    The second is conclusion (B): there exists a value of $\rho > \frac{\de}{\de'_0}$ and a $\rho$-tube $T_{\rho}$ such that $|\T'^{T_{\rho}}| \ge \left(\frac{\de}{\rho}\right)^{-\beta - \alpha}$ and $\T'^{T_{\rho}}$ satisfies the $t$-Frostman Convex Wolff Axiom with error $\left(\frac{\de}{\rho}\right)^{-\eta_1}$. Since $\frac{\de}{\rho} < \de_0$ we conclude as in the previous paragraph that ${\T'}^{T_{\rho}}$ has upper spectrum at $\eta_1$ and scale $\frac{\de}{\rho}$ at least $\sigma^* + \e$, which as in the previous paragraph implies that $\T'$ has upper spectrum at $\eta\eta_1$ and scale $\de$ at least $\sigma^* + \e - \log_{1/\de}(C)/\eta$ again by Lemma \ref{lem:qa}. By taking $\de_0$ sufficiently small (note that this only requires knowing the absolute constant $C$ and $\eta$), we ensure that $\T'$ has upper spectrum at $\eta\eta_1$ and scale $\de$ at least $\sigma^* + 3\e/4$. As in the previous paragraph, this is a contradiction, so (B) cannot occur.

    The third conclusion is (C). Denote the resulting set of prisms by $\calR$. Let $\T_{\rho} = \calD_{\rho}(\T')$. If we set $\T_{\rho}(R) = \{T_{\rho} \in \calD_{\rho}(\T') : T_{\rho} = N_{\rho}(T), T \in \T[R]\}$, then since the $t$-Frostman Convex Wolff Axiom is inherited upwards, we have that $\T_{\rho}(R)$ satisfies the $t$-Frostman Convex Wolff Axiom with error $\rho^{-\eta_2}$ for each $R \in \calR$ and $\calR$ satisfies the $t$-Frostman Convex Wolff Axiom with error $\rho^{-\eta_2}$ as well. Finally, elements $R \in \calR$ have dimensions $\rho \times \De \times 1$. Since $\rho \le \de^{\eta_1/100}\De \le \de_0^{\eta/100}$ , so long as $\de_0$ is sufficiently small, we can ensure that $\rho \le \overline{\de}_0$. Then we can apply Proposition \ref{prop:prismsetup} with $Y_{\text{full}}(T)$ the full shading for each $T \in \T'$ and conclude that $\T'$ has upper spectrum at $\eta_1$ and scale $\de$ at least $ 2t+1 - \omega/2 \ge \sigma - \omega/2$. It follows that $\sigma^* \ge \sigma - 3\omega/4$, a contradiction. Hence (C) cannot occur.

    Finally, we have conclusion (D). Recall $\T$ satisfies the $t$-Frostman Convex Wolff Axiom at all scales with error $\de^{-\zeta}$ and $\de \le \De_0$. Then $\T$ has upper spectrum at $\zeta$ and scale $\de$ at least $\sigma - \omega/2$. By the same reasoning as in the previous paragraph, we reach a contradiction.

    In all cases, we contradict our assumption that $\sigma^* < \sigma - 3\omega/4$. It follows that $\sigma^*(2t) \ge \sigma - 3\omega/4$. We conclude that for $\eta, \de_0 > 0$ sufficiently small, $\sigma^*(2t, \eta, \de_0) \ge \sigma - \omega$. Then for any $\de < \de_0$ and any set of $\de$-tubes $\T$ satisfying the $t$-Frostman Convex Wolff Axiom with error $\le \eta$, we know that $|\T| \gtrsim \de^{-2t+\eta} \ge \de^{-2t + 2\eta}$. It follows that $\sigma^*(\T, \eta) \ge \sigma - \omega$, that is, $\bigcup_{T \in \T} T$ has upper spectrum $\ge \sigma - \omega$ at $\eta$ and scale $\de$. This is the conclusion of Theorem \ref{thm:realmain}, completing the proof.
\end{proof}

\section{Proving Proposition \ref{prop:prismsetup}: Induction on scales}\label{sec:prismsetup}

In this section, we prove Proposition \ref{prop:prismsetup}, whose statement we recall now.

\begin{proposition*}[\textbf{\ref{prop:prismsetup}}]
    For any $\omega, \e > 0$, there exists $\eta, \de_0 > 0$ such that the following holds for all $\de \in (0, \de_0)$. Suppose $\T$ is a collection of $\de$-tubes satisfying the $t$-Frostman Convex Wolff Axiom with error $\de^{-\eta}$, $Y$ is a shading $\de^{\eta}$ dense on $\T$, $\calR$ is a collection of $\de \times \De \times 1$ prisms for some $\De \ge \de^{1-\e}$, and $\calR$ factors $\T$ from above and below with respect to the $t$-Frostman Convex Wolff Axiom with error $\de^{-\eta}$. 
    
    Then $\bigcup_{T \in \T} Y(T)$ has upper spectrum at $\eta$ and scale $\de$ at least $2t+1 - \omega$. 
\end{proposition*}

The proof of Proposition \ref{prop:prismsetup} follows from an induction on scales argument, which requires two further propositions. We state these propositions now.

\begin{proposition}\label{prop:largeprism}
    For any $\omega, \e, \tau > 0$, there exist $\eta, \de_0 > 0$ such that the following holds for all $\overline{\de} \in (0, \de_0)$. Suppose there exist values $\de, \De > 0$ with $\overline{\de} \le \de \le \De$ and $\de \le \overline{\de}^{\e}\De$ and a collection $\calR$ of $\de \times \De \times 1$ prisms satisfying the $t$-Frostman Convex Wolff Axiom with error $\de^{-\eta}$ and with a $\de^{\eta}$-dense shading $Y$ on $\calR$.

    Then one of the following conclusions must occur.

    \begin{enumerate}[label = \Alph*.]
        \item There exists $r > \overline{\de}^{-\eta}\de$ and an $r$-ball $B_r$ such that \begin{equation*}\left|\bigcup_{R \in \calR} Y(R) \cap B_r\right|_{\de} \ge \left(\frac{r}{\de}\right)^{2t + 1 - \omega};\end{equation*}or 
        \item There exist $\de' \le \De'$ with $\de' \in [\overline{\de}^{-\e\omega/6}\de, \overline{\de}^{\e/2}]$ and $\De' = \min\left(1, \frac{\De\de'}{\de}\right)$ and a collection $\calR'$ of $\de' \times \De' \times 1$ prisms satisfying the $t$-Frostman Convex Wolff Axiom with error $\de'^{-\tau}$ and a $\de'^{\tau}$-dense shading $Y'$ on $\calR'$ such that for all $R \in \calR'$ \begin{equation*}Y'(R) \subset N_{\de'}\left(\bigcup_{R \in \calR[R']} Y(R)\right).\end{equation*}
    \end{enumerate}
\end{proposition}
We prove Proposition \ref{prop:largeprism} in Section \ref{sec:largeprism}. 
\begin{proposition}\label{prop:smallprism}
    For any $\omega, \e, \tau > 0$, there exists $\eta, \de_0 > 0$ such that the following holds for all $ \de \in (0, \de_0)$ and $\De \in (\de^{1-\e}, 1]$. Suppose $\calR$ is a collection of $\de \times \De \times 1$ prisms satisfying the $t$-Frostman Convex Wolff Axiom with error $\de^{-\eta}$, $\T$ is a collection of $\de$-tubes such that for each $R \in \calR$, there exists a family $\T(R)$ contained in $R$, satisfying the $t$-Frostman Convex Wolff Axiom with error $\de^{-\eta}$ and with $|\T(R)| \le \de^{-\eta} \left(\frac{\De}{\de}\right)^t$. Let $Y(T)$ be a $\de^{\eta}$-dense shading on $\T$. Then one of the following must hold. 

    \begin{enumerate}[label = \Alph*.]
        \item There exists $r > \de^{1-\eta}$ and an $r$-ball $B_r$ such that \[\left|\bigcup_{T \in \T} Y(T) \cap B_r\right|_{\de} \ge \left(\frac{r}{\de}\right)^{2t + 1 - \omega};\]or
        \item There exists some $\overline{\de}, \overline{\De}$ with $\frac{\overline{\De}}{\overline{\de}} = \left(\frac{\De}{\de}\right)^{\omega/(10t + 5)}$ and a collection $\overline{\calR}$ of $\overline{\de}\times \overline{\De}$ prisms satisfying the $t$-Frostman Convex Wolff Axiom with error $\overline{\de}^{-\tau}$, equipped with a $\overline{\de}^{\tau}$-dense shading $\overline{Y}$ such that $\overline{Y}(\overline{R}) \subseteq N_{\overline{\de}}\left(\bigcup_{R \in \calR} Y(R)\right)$ for each $\overline{R} \in \overline{\calR}$.
    \end{enumerate}
\end{proposition}

We prove Proposition \ref{prop:smallprism} in Section \ref{sec:smallprism}. While we actually apply Proposition \ref{prop:smallprism} before Proposition \ref{prop:largeprism} in the proof of Proposition \ref{prop:prismsetup}, the proof of Proposition \ref{prop:smallprism} is more technically challenging than the proof of Proposition \ref{prop:largeprism}, so the author suggests reading them in the order presented here.

We also require three $L^2$-type incidence bounds, which we state below. These three incidence bounds have a common set-up: we have error terms $C_1, C_2, C_3,$ and $C_4 \ge 1$, a choice of $t \in (0, 1]$, and some scale $\de > 0$.

\begin{lemma}\label{lem:RFPA}
    Fix $\De \in (\de, 1]$. Suppose $(\T, Y)$ has $|Y(T)|_{\de} \ge C_1^{-1}\de^{-1}$. Also, suppose for some $\de \times \De \times 1$ prism $R$ and all $T \in \T$, $T \subset R$  and $\T^R$ satisfies the $t$-Frostman Convex Wolff Axiom with error $C_2$. Then \[\left|\bigcup_{T \in \T} Y(T)\right|_{\de}\gtrsim |\log \de|^{-1}C_1^{-1}C_2^{-2}(\de/\De)^{-t}\de^{-1}.\]
        In particular, if $C_1, C_2 \le \de^{-\eta}$ and $\de$ is sufficiently small, then \[\left|\bigcup_{T \in \T} Y(T)\right|_{\de} \ge \de^{4\eta}(\de/\De)^{-t}\de^{-1}.\]
\end{lemma}
\begin{lemma}\label{lem:RFPB}
    Fix $\De \in (\de, 1]$ and $\nu \in (\De, 1]$. Suppose $\calP$ is a collection of $\de \times \De \times \De$ prisms satisfying the $t$-Frostman Convex Wolff Axiom with error $C_1$ inside some $G$ a $\nu \times \De \times \De$ prism. Suppose furthermore that for each $P \in \calP$, we have a collection of tubes $\T(P)$ satisfying the $t$-Frostman Convex Wolff Axiom with error $C_2$ inside $P$ and $|\T(P)| \le C_3(\frac{\De}{\de})^t$. Finally, $\T(P)$ is equipped with a shading $Y$ such that $\left|\bigcup_{T \in \T(P)} Y(T)\right|_{\de} \ge C^{-1}_4 (\frac{\De}{\de})^{t+1}$.Then \[\left|\bigcup_{T \in \T} Y(T)\right|_{\de} \gtrsim |\log \de|^{-2}C_{4}^{2}C_3^{-1}C_2^{-1}C_1^{-1} \left(\frac{\nu}{\de}\right)^t \left(\frac{\De}{\de}\right)^{t+1}.\]
        In particular, if $C_i \le \de^{\eta}$ for $i = 1, 2, 3, 4$, then taking $\de$ sufficiently small, we have that \[\left|\bigcup_{T \in \T} Y(T)\right|_{\de} \gtrsim \de^{6\eta} \left(\frac{\nu}{\de}\right)^t \left(\frac{\De}{\de}\right)^{t+1}.\]
\end{lemma}

\begin{lemma}\label{lem:RFPC}
    Fix $\De \in (\de, 1]$ and $\nu \in (\De, 1]$. Suppose that $\calP$ is a collection of $\de \times \De \times \De$ prisms satisfying the $t$-Frostman Convex Wolff Axiom with error $C_1$ inside some $G$ a $\nu \times \De \times \De$ prism. Suppose $\calP$ is equipped with a $C^{-1}_2$ dense shading $Y$. Then \[\left|\bigcup_{P \in \calP} Y(P)\right|_{\de} \gtrsim |\log \de|^{-1} C_1^{-1}C_2^{-2} \left(\frac{\nu}{\de}\right)^{t}\left(\frac{\De}{\de}\right)^2.\]
\end{lemma}

The bounds in these lemmas all follow from standard Cordoba-style arguments or small variations thereof. Their statements, in particular Lemma \ref{lem:RFPB}, are adapted to their use in the proofs of Proposition \ref{prop:prismsetup}, Proposition \ref{prop:largeprism}, and Proposition \ref{prop:smallprism}. We defer their proofs to Section \ref{subsec:RFP}. 

We also prove that we can reduce a set of tubes $\T$ satisfying the $t$-Frostman Convex Wolff Axiom in a slab $W$ with small error to a set $\T' \subset \T$ with $\approx \de^{-t}$ many elements satisfying the $t$-Frostman Convex Wolff Axiom in $W$. This is a special case of the general fact that sets of $\de$-tubes $\T$ satisfying the $t$-Frostman Convex Wolff Axiom contain sets of $\de$-tubes $\T_0$ satisfying the $t$-Frostman Convex Wolff Axiom with $|\T_0| \approx \de^{-t}$. This general fact is well-known but the author is not aware of a proof in the literature for the specific formulation we need, so we prove it here.

\begin{lemma}\label{lem:folklore}
    For any $\eta > 0$, there exists $\de_0$ sufficiently small such that for any $\de \times 1 \times 1$ slab $W$ and any set of tubes $\T \subset 3W$ satisfying the $t$-Frostman Convex Wolff Axiom with error $\de^{-\eta}$ in $W$, there exists a set $\T' \subset \T$ such that $|\T'| \ge \de^{-t-10\eta}$ which satisfies the $t$-Frostman Convex Wolff Axiom with error $\de^{-10\eta}$ in $W$.
\end{lemma}

\begin{proof}
    To construct $\T'$, we randomly include tubes from $\T$ with probability $p = \frac{\de^{-t}}{\T}$. The size of $\T$ is a binomial random variable with probability $p$ and $n = |\T|$ trials. In particular, $\mu := \mathbb{E}[|\T|] = \de^{-t}$. Define the following events: \begin{itemize}
        \item $E_{-} = \{|\T'| \le \de^{-t-10\eta}\}$.
        \item $E_+ = \left\{|\T'| \ge \frac{1}{2}\de^{-t}\right\}$. 
        \item For each $j = 1, \dots, \left\lceil\frac{1}{\eta}\right\rceil$ and each $\de_j = \de^{j\eta}$, let $\T_{\de_j}$ be a maximal set of essentially distinct $\de_j$ tubes. For $T_{\de_j} \in \T_{\de_j}$, let \[E_{T_{\de_j}} = \left\{|\T'[T_{\de_j}]| \ge \de^{-5\eta}\left(\frac{\de_j}{\de}\right)^t\right\}.\]
    \end{itemize}
    Suppose \begin{equation}\label{eqn:probs}P(E_-) + P(E_+) + \sum_{j = 1}^J \sum_{T_{\de_j} \in \T_{\de_j}} P(E_{T_{\de_j}}) < 1.\end{equation}Then there exists a choice of $\T'$ with \begin{equation}\label{eqn:numbs}|\T'| \in \left[\frac{1}{2}\de^{-t}, \de^{-t-10\eta}\right]\end{equation} and for all $T_{\de_j}$ \begin{equation}\label{eqn:bumbs}|\T'[T_{\de_j}]| \le \de^{-5\eta}\left(\frac{\de_j}{\de}\right)^t.\end{equation}Suppose $\T'$ fails the $t$-Frostman Convex Wolff Axiom with error $\le \de^{-10\eta}$ in $W$. Then there exists an $a \times b \times 1$ prism $R$ with \begin{equation}\label{eqn:dialectic}|\T'[R]| \ge \de^{-10\eta}\frac{\text{Vol}(R)^t}{\vol(W)^t}|\T'|.\end{equation}We might as well assume that $a = 3\de$, since each tube is in $\T$ is contained in a common $3\de \times 1 \times 1$ slab $3W$. Let $\de_{j'}$ be the smallest choice of $\de_{j'}$ larger than $2b$. In particular, $b \gtrsim \de^{\eta}\de_{j'}$. Together with (\ref{eqn:numbs}), these facts imply that \[\frac{\text{Vol}(R)^t}{\vol(W)^t}|\T'| \gtrsim \de^{t\eta}\left(\frac{\de_{j'}}{\de}\right)^{t}.\]

    By the maximality of $\T_{\de_{j' - 1}}$, we can find a $\de_{j'-1}$ tube $T_{\de_{j'-1}}$ containing $R$, so by (\ref{eqn:dialectic}), we see that \[|\T'[T_{\de_{j'-1}}]| \ge |\T'[T_R]| \ge \de^{-10\eta}\frac{\text{Vol}(R)^t}{\vol(W)^t}|\T'| \gtrsim \de^{-9\eta}\left(\frac{\de_{j'}}{\de}\right)^{-t} \gtrsim \de^{-8\eta} \left(\frac{\de_{j'-1}}{\de}\right)^{t}.\]Comparing with (\ref{eqn:bumbs}), we reach a contradiction. Therefore, if (\ref{eqn:probs}) holds, then we can find a set $\T' \subset \T$ satisfying the conclusions of the lemma. 

    It remains to prove (\ref{eqn:probs}). For each $T_{\de_j}$, $|\T'[T_{\de_j}]|$ is binomially distributed with $p$ as defined above and $n_{T_{\de_j}} = |\T[T_{\de_j}]|$. Since $\T$ satisfies the $t$-Frostman Convex Wolff Axiom, $|\T[T_{\de_j}]| \le \de^{-\eta}\de_j^t |\T|$, $pn_{T_{\de_j}} \le \de^{-\eta} \left(\frac{\de_j}{\de}\right)^t$. Let $\nu = \frac{\de^{-5\eta}\left(\frac{\de_j}{\de}\right)^t - pn_{T_{\de_j}}}{pn_{T_{\de_j}}}$. Our bound for $pn_{T_{\de_j}}$ implies that $pn_{T_{\de_j}} \nu \ge \de^{-4\eta}$and $\nu \ge \frac{\de^{-4\eta}}{2}$, assuming $\de$ is sufficiently small \footnote{ The first bound is certainly very rough, but suffices for the proof.}. By the multiplicative Chernov bound, we see that \[P(E_{T_{\de_j}}) \le P(|\T'[T_{\de_j}]| \ge (1 + \nu)pn_{T_{\de_j}}) \le e^{-\nu^2 p n_{T_{\de_j}}/(2 + \nu)}.\]Since $\nu > 2$, $\frac{\nu}{2+\nu} > \frac{1}{2}$, so $e^{-\nu^2 p n_{T_{\de_j}}/(2 + \nu)} < e^{-\nu p n_{T_{\de_j}}/2} < e^{-\de^{-4\eta}/4}$. We know that $|\T_{\de_j}| \le \de^{-4}$, so we conclude that \[\sum_{j = 1}^J \sum_{T_{\de_j} \in \T_{\de_j}} P(E_{T_{\de_j}}) \le J \de^{-4}e^{-\de^{-4\eta}/4}.\]So long as $\de$ is sufficiently small, we conclude that $\sum_{j = 1}^J \sum_{T_{\de_j} \in \T_{\de_j}} P(E_{T_{\de_j}}) < \frac{1}{3}$.

    Applying the same argument with $B(0,1)$ taking the place of $T_{\de_j}$, we see that we can ensure $P(E_-) < \frac{1}{3}$. To prove that $P(E_+) < \frac{1}{3}$, we use a different variation of the multiplicative Chernov bound. We know \[P(E_+) = P\left(|\T'| > \frac{\mu}{2}\right) < e^{-\mu/8} = \de^{-\de^{-t}/8}.\]So long as $\de$ is sufficiently small, this is $<\frac{1}{3}$, as desired. Since each of the three summands in (\ref{eqn:probs}) is $< 1/3$, we know (\ref{eqn:probs}) holds, completing the proof.
\end{proof}

Now, we outline the proof of Proposition \ref{prop:prismsetup}. We start with a collection of $\de$-tubes $\T$ supporting a shading $Y$ and factored from above and below with respect to the $t$-Frostman Convex Wolff Axiom by some set of $\de \times \De \times 1$ prism $\calR$. We know that by applying Lemma \ref{lem:RFPB}, we have a good upper spectrum for $\T$ if $\De = 1$. So we suppose otherwise. Then, we apply Proposition \ref{prop:smallprism} with our choice of $Y, \T,$ and $\calR$. Proposition \ref{prop:smallprism} A implies good upper spectrum bounds for $\T$. Proposition \ref{prop:smallprism} B implies that a small neighborhood of $\T$ arranges into $\de_N \times \De_N \times 1$ prisms $\calR_N$ which satisfy the hypotheses of Proposition \ref{prop:largeprism}. We then apply Proposition \ref{prop:largeprism} to $\calR_N$. If outcome A occurs, we have good upper spectrum for $\calR_N$ and hence for $\T$ as well. If outcome B occurs, the prisms in $\calR_N$ arrange into a set of $\de_{N-1} \times \De_{N-1} \times 1$ prisms $\calR_{N-1}$ with $\De_{N-1} \ge \de^{-C}\De_N$, for some small fixed value $C > 0$ determined in the course of the proof. We repeat this argument at most $N = \left\lceil \frac{1}{C}\right\rceil$ many times, at which point we are sure to have reach $\De_i = 1$ for some $i$, at which point we use Lemma \ref{lem:RFPC} to complete the argument.

This ends the outline of the proof of Proposition \ref{prop:prismsetup}. We proceed with the proof itself.

\begin{proof}
    Let $\e' = \frac{\e\omega}{10t+5}$; this is a lower bound on the scale separation of the side lengths of the prisms returned by Proposition \ref{prop:smallprism} B. We claim that there exists some $\eta_0 < \frac{\omega}{100}$ satisfying the following claim.
    
    If $\calR_0$ is a collection of $\de_{(0)} \times 1 \times 1$ prisms with $\de_{(0)} \le \de^{\e'/4}$ and satisfying the Frostman Convex Wolff Axiom with error $\le \de_{(0)}^{-\eta_0}$ and with a $\ge \de_{(0)}^{-\eta_0}$-dense shading $Y_0$, then \[\left|\bigcup_{R \in \calR_0} Y_0(R)\right|_{\de_{(0)}} \ge \de_{(0)}^{-2t-1+\omega}.\]
    
    For the rest of this proof, we call this claim the \emph{prism claim}. To prove the prism claim, we apply Lemma \ref{lem:RFPC} to $\calR_0$ and $Y_0$, with $\nu = \De = 1, C_1 = \de_{(0)}^{-\eta_0}$, and $C_2 =\de_{(0)}^{-\eta_0}$. We conclude that, so long as $\eta_0$ and $\de_{(0)}$ are sufficiently small (note that we can make $\de_{(0)}$ smaller by choosing $\de_0$ smaller, since $\de_{(0)} \le \de^{\e'/4} \le \de_0^{\e'}$), we have \begin{equation}\label{eqn:whatsonemorenumber}\left|\bigcup_{R \in \calR_0} Y_0(R)\right|_{\de_{(0)}} \gtrapprox_{\de} \de_{(0)}^{3\eta_0}\de_{(0)}^{-2t-1} \ge \de_{(0)}^{-2t-1+\omega}.\end{equation}

    Let $N = \left\lceil\frac{4}{\e'\omega}\right\rceil$. Now that we have defined $\eta_0$, we inductively define sequences $\eta_i, \de_i$ for $i = 1, \dots, N$ by applying Proposition \ref{prop:largeprism} with $\e := \e'$ and $\tau = \eta_{i-1}$ and let $\tilde{\eta}_i, \de_i$ respectively be the values of $\eta, \de_0$ given by the proposition. Let $\eta_i = \min(\eta_{0}, \tilde{\eta}_i)$. We arrive after $N$ steps of this procedure at a choice of $\eta_N > 0$. We choose $\eta', \tilde{\de}_0$ sufficiently small to apply Proposition \ref{prop:smallprism} with $\tau = \eta_N$ and $\omega, \e$ as given in the statement of Proposition \ref{prop:prismsetup}. We choose $\eta = \frac{\e\eta'}{10}$, this is the final choice of $\eta$ in the proposition. We set our final choice of $\de_0$ to be $\le \min_{i=1, \dots, N} \de_i$, $\le \tilde{\de}_0$, and small enough to satisfy the implicit constraints from (\ref{eqn:whatsonemorenumber}). 

    Now, suppose $\calR, \T$, and $Y$ are as defined in the statement of this proposition with the our choices of $\eta'$ and $\de \le \de_0$. We first apply Proposition \ref{prop:smallprism} with $\tau = \eta_N$. We know $\calR$ satisfies the $t$-Frostman Convex Wolff Axiom with error $\de^{-\eta'}$ and $Y(T)$ is $\de^{\eta'}$-dense, but $\T[R]$ may be too large, so we must first define $\T(R)$ to satisfy the $t$-Frostman Convex Wolff Axiom with error $\de^{-\eta'}$ and with $|\T(R)| \le \de^{-\eta'}\left(\frac{\De}{\de}\right)^t$. To do so, we let $T_{\De}$ be the $\De$-tube containing $R$, set $\tilde{\T} = \T(R)^{T_{\De}}$, a set of $\tilde{\de} = \frac{\de}{\De} \le \de_0^{\e}$-tubes satisfying the $t$-Frostman Convex Wolff Axiom with error $\le \de^{-\e\eta'/10} \le \tilde{\de}^{-\eta'/10}$. We apply Lemma \ref{lem:folklore} to $\tilde{\T}_0$, taking $\de_0$ small enough to ensure Lemma \ref{lem:folklore} applies. Denote by $\tilde{\T}'_0$ the resulting set of tubes, with \[|\tilde{\T}'_0| \le \tilde{\de}^{-10\eta}\left(\frac{\De}{\de}\right)^t \le \de^{-\eta'}\left(\frac{\De}{\de}\right)^t\]and $\tilde{\T}'_0$ satisfies the $t$-Frostman Convex Wolff Axiom with error $\le \tilde{\de}^{-\eta'} \le \de^{-\eta'}$. Undoing the rescaling, we have a set of tubes $\T(R) \subset \T[R]$ with $|\T(R)| \le \de^{-\eta'}\left(\frac{\De}{\de}\right)^t$ and satisfying the $t$-Frostman Convex Wolff Axiom with error $\le \de^{-\eta'}$, as desired. We now apply Proposition \ref{prop:smallprism}.
    
    If Conclusion A of that proposition occurs, then $\bigcup_{T \in \T} Y(T)$ has upper spectrum at $\eta$ and scale $\de$ at least $2t+1 - \omega$. Otherwise, Conclusion B tells us that we have variables $\de_N := \overline{\de}, \De_N := \overline{\De}$ and a collection $\calR_N$ of $\de_N \times \De_N \times 1$ prisms satisfying the $t$-Frostman Convex Wolff Axiom with error $\le \de_N^{-\eta_N}$ and a $\le \de_N^{\eta_N}$ dense shading $Y_N$. Moreover, $Y_N(\calR_N) \subset Y(\T)$ and $\frac{\De_N}{\de_N} \ge \left(\frac{\De}{\de}\right)^{\omega/(10t + 5)} \ge \de^{-\e\omega/(10t+5)} = \de^{-\e'}$. 

    We now are in a position to inductively apply Proposition \ref{prop:largeprism}. If we ever reach Conclusion A, then we have the desired upper spectrum and so we stop the induction and end the proof. If $\De_i = 1$ after any step, then we show that $\de_i \le \de^{\e'/4}$, then apply the prism claim to conclude the argument. If neither happen, we continue with induction. Finally, we prove that if we continue with the induction until $i = 1$, we will have $\De_1 = 1$ and hence can prove the upper spectrum is sufficiently large.

    We maintain the following properties throughout the induction argument:
    \begin{enumerate}[label = (\alph*)]
        \item $(\calR_i, Y_i)$ is a $\de^{-\eta_i}$ dense shading on a collection $\calR_i$ of $\de_i \times \De_i \times 1$ prisms satisfying the $t$-Frostman Convex Wolff Axiom with error $\de^{-\eta_i}$;
        \item $\De_i \in (\de^{-\e'}\de_i, 1)$;
        \item $Y_i(\calR_i) \subset N_{\de_i}(Y(\T))$; and
        \item $\de_i > \de^{(i-N) \e'\omega/6}\de$.
    \end{enumerate}
    We have already confirmed each property of the four properties when $i = N$ except for the claim in (b) that $\De_N < 1$. Suppose otherwise. Since $\eta_N \le \eta_0$ and $\de_N \le \de^{\e'} \le \de^{\e'/2}$, we know by the prism claim that \[\left|\bigcup_{R \in \calR_N} Y_N(R)\right|_{\de_{N}} \ge \de_{N}^{-2t-1+\omega}.\]We know that $Y_i(\calR_i) \subset N_{\de_i}(Y(\T))$. It follows that \[\left|\bigcup_{T \in \T} Y(T)\right|_{\de_{N}} \ge \left|\bigcup_{R \in \calR_N} Y_N(R)\right|_{\de_{N}} \ge\de_{N}^{-2t-1+\omega}.\]If $\De = 1$, this completes the proof of the proposition, so we going forward we suppose $\De_N < 1$. 
    
    We claim that if properties (a) - (d) hold for $(\calR_i, Y_i)$ at step $i$, then one of three possibilities must hold: there exists a pair $(\calR_{i-1}, Y_{i-1})$ such that properties (a) - (d) also hold at step $i-1$; there exists a $r \ge \de^{-\eta}\de_{i}$ and an $r$-ball $B_r$ such that \begin{equation}\label{eqn:goal1}\left|\bigcup_{R \in \calR_i} Y_i(R) \cap B_r \right|_{\de_{i}} \ge \left(\frac{r}{\de_{i}}\right)^{2t+1-\omega};\end{equation}or \begin{equation}\label{eqn:goal2}\left|\bigcup_{R \in \calR_{i-1}} Y_{i-1}(R)\right|_{\de_{i-1}} \ge \de_{i-1}^{-2t-1+\omega}.\end{equation}We call this trichotomy the \emph{induction claim}. 
    
    To prove the induction claim, assume (a) - (d) hold for $(\calR_i, Y_i)$. Since $(\calR_i, Y_i)$ satisfies (a) and from how we chose $\eta_i$, we can apply Proposition \ref{prop:largeprism} with $\tau = \eta_{i-1}$, $(\calR_i, Y_i)$ the set of prisms, and $\overline{\de} = \de$. By (b), we have the needed scale separation. 
    
    First, suppose that Conclusion B occurs when we apply Proposition \ref{prop:largeprism}. Let $\de_{i-1}, \De_{i-1}$ be the scales returned in that case and $(\calR_{i-1}, Y_{i-1})$ the prisms and shading. Conclusion B gives us (a) immediately and (c) is a restatement of (\ref{eqn:lpcover}). Applying (d) from the $i$th step, we see that \[\de_{i-1} > \de^{-\e \omega/6}\de_i \ge \de^{(i-1-N)\e\omega/6}\de,\]and hence we have (d) on the $i-1$st step. If $\De_{i-1} = 1$, then since $\eta_{i-1} \le \eta_0$ and $\de_{i-1} \le \de^{\e'/4}$, we apply the prism claim to prove that (\ref{eqn:goal2}) holds. Otherwise, we know that $\De_{i-1} < 1$ and $\frac{\De_{i-1}}{\de_{i-1}} = \frac{\De_i}{\de_i} \ge \de^{-\e'}$, so (b) holds. Now, suppose Conclusion A occurs when we apply Proposition \ref{prop:largeprism}. After relabelling variables, (\ref{eqn:lpdense}) is (\ref{eqn:goal1}). This completes the proof of the induction claim.

    Suppose properties (a) - (d) hold for $(\calR_i, Y_i)$ step $i$. We apply the induction claim to $(\calR_i, Y_i)$. If there exists a pair $(\calR_{i-1}, Y_{i-1})$ such that properties (a) - (d) hold at step $i-1$, then we continue with the induction process. If (\ref{eqn:goal1}) or (\ref{eqn:goal2}) occur, then the induction process halts. Using property (c), we see that (\ref{eqn:goal1}) or (\ref{eqn:goal2}) imply that $Y(\T)$ has the desired upper spectrum at $\eta$ and scale $\de$. Since $\de_i < \De_i < 1$ at each step and $N \ge \frac{6}{\e'\omega}$, (d) implies that the induction procedure must halt by the time $i = 0$. This completes the proof.
\end{proof}

\section{Proving Proposition \ref{prop:largeprism}: Prism expansion versus larger upper spectrum}\label{sec:largeprism}

In this section, we prove Proposition \ref{prop:largeprism}. We begin this section by discussing some lemmas we need in the proof. Then we recall the statement of Proposition \ref{prop:largeprism} and give the proof. 

Our lemmas start with a few lemmas concerning a non-concentration property called broadness. We initially state this for an arbitrary finite set in a metric space, where it coincides with the definition of a $\beta$-Frostman measure on the metric space. We then apply it to the directions of the set of tubes containing a point.

\begin{definition}\label{def:broad}
    \begin{enumerate}
        \item If $(X, d)$ is a finite metric space and $A \subset X$ is a $\de$-separated set of points, we say $A$ is \emph{$\beta$-broad in $X$ with error $K$} if for any $y \in X$ and any $r > \de$, \[|\{x \in A: d(x, y) < r\}| \le K \left(\frac{r}{\text{diam}(X)}\right)^{\beta}|A|.\]
        \item If $\T$ is a collection of $\de$-tubes with a shading $Y(T)$, we say that $(\T, Y)$ is \emph{$\beta$-broad with error $K$} if for any $\de$-cube $Q \in \bigcup_{T \in \T} Y(T)$, $\{\dir(T) : Q \in Y(T)\}$ is $\beta$-broad.
        \item We say that $\T_{\rho}$ is a \emph{$\beta$-broad cover with error $K$} of $(\T, Y)$ if $\T_{\rho}$ is a balanced paritioning cover of $\T$ and for each $T_{\rho} \in \T_{\rho}$ and $(\T^{T_{\rho}}, Y^{T_{\rho}})$ is $\beta$-broad with error $K$.
        \item We say that $\T_{\rho}$ is a \emph{$\beta$-broad cover of $(\T, Y)$ with error $K$ broadness multiplicity $M$} if $\T_{\rho}$ is a $\beta$-broad cover with error $K$ of $\T$ and for each $B \in \bigcup_{T \in \T} Y(T)$, \[\left|\left\{T_{\rho} \in \T_{\rho} : B \in \bigcup_{T \in \T[T_{\rho}]} Y(T)\right\}\right| \le M.\] 
    \end{enumerate}
\end{definition}
\begin{remarkc}\label{rmk2}
    One useful consequence of this definition is that if $A$ is $\beta$-broad in $B(x,r)$ with error $K$, then there exists $r' \gtrsim K^{-1/\beta}r$ such that $A \cap (B(x,r) \setminus B(y, r')) \neq \emptyset$ for any $y \in B(x,r)$. In particular, this implies that for any $v \in A$, there exists $v' \in A$ with $d(v, v') \gtrsim K^{-1/\beta}r$. 

    Another useful consequence of this definition is that if $W \subset V$, $V$ is $\beta$-broad in some ambient set $X$ with error $C$, and $|W| \ge \kappa |V|$, then $W$ is $\beta$-broad in $X$ with error $\kappa^{-1}C$. By the same token, if $(\T, Y)$ is broad in some $\T_{\rho}$ with error $K$ and $(\T, Y')$ is a subshading of $Y$ so that for each cube $Q \in \bigcup_{T \in \T} Y'(T)$, $|{T \in \T : Q \in Y'(T)|}| \ge \kappa |{T \in \T : Q \in Y(T)|}|$, then $(\T, Y')$ is broad in $\T_{\rho}$ with error $\kappa^{-1}K$.  

    Combining these two consequences, we conclude that if $A$ is $\beta$-broad in $B(x,r)$ with error $K$ and $A' \subset A$ with $|A'| \ge \kappa |A|$, we have that for any $v \in A'$, there exists $v' \in A'$ with $d(v, v') \gtrsim \kappa^{1/\beta}K^{-1/\beta}r$.
\end{remarkc}

If (iii) or (iv) occur, we say that $\T$ is broad in $\T_{\rho}$. In practice, we investigate the broadness of $\calD_{\De}(\calR)$, where $\calR$ is a collection of $\de \times \De \times 1$ prisms. If $\calD_{\De}(\calR)$ is broad in a collection of $u$-tubes $\T_u$, we say that the long directions of $\calR$ are broad in $\T_u$. 

While these definitions focus on broadness of smaller tubes in larger tubes, we also need to apply broadness for the normal directions of sets of hyperplanes. We apply the following lemma to do so. The lemma is an easy consequence of \cite[Lemma 7.8]{WZ25}. 

\begin{lemma}[{\cite[Lemma 7.8]{WZ25}}]\label{lem:localbroad}
    Let $\de, \beta > 0$ and suppose $V \subset S^2$ is a $\de$-separated set. Then there exists some $\rho \in [\de, 1]$, a $\rho$-ball $B \subset S^2$, and some $V' \subset V \cap B$ such that 
    \begin{enumerate}[label = (\roman{*})]
        \item $|V'| \gtrsim \rho^{\beta}|V|$.
        \item $V'$ is $\beta$-broad in $B$ at scales $\ge \de$ with error $\lesssim 1$.
    \end{enumerate}
\end{lemma}

To derive Lemma \ref{lem:localbroad} from \cite[Lemma 7.8]{WZ25}, we choose some $B \in \mathcal{B}$ and take $V' = \mathcal{V}_B$. With this definition (i) and (ii) are restatements of (i) and (ii) from \cite[Lemma 7.8]{WZ25}. 

When we apply Lemma \ref{lem:localbroad} for the normal directions of sets of hyperplanes, we often will not have $\de$-separation (indeed, the ``sets of normals'' may in fact be multisets). We use a Vitali-style refinement followed by dyadic pigeonholing to find sets $V \subset \overline{V} \subset \tilde{V}$ so that $V$ is $\de$-separated, $\overline{V} \subset N_{\de}(V)$, $|\overline{V}| \approx_{\de} |\tilde{V}|$, and for each $\de$-ball in $B \in N_{\de}(V)$, $|\overline{V} \cap B|$ is constant. We then apply Lemma \ref{lem:localbroad} to $V$ to find some $V' \subset V$ and some $\rho$-ball $B$ which satisfies (i) and (ii). We set $\overline{V}' = \overline{V} \cap N_{\de}(V')$ and easily see that $|\overline{V'}| \gtrsim \rho^{\beta}|\overline{V}|$. While $\overline{V'}$ is not broad by our definition, it satisfies the two conclusions of Remark \ref{rmk2}. Instead of repeating this line of reasoning in the proofs of Propositions \ref{prop:largeprism} and \ref{prop:smallprism}, we state the result as a corollary here, which we later apply in both propositions.

\begin{corollary}\label{cor:localbroad}
    Let $\de, \beta > 0$ and suppose $V \subset S^2$ is a finite multiset. Then there exists some $\rho \in [\de, 1]$, a $\rho$-ball $B \subset S^2$, and some $V' \subset V \cap B$ such that 
    \begin{enumerate}[label = (\roman{*})]
        \item $|V'| \gtrapprox_{\de} \rho^{\beta}|V|$.
        \item For any $W \subset V'$ with $|W| \ge \kappa |V'|$, we have that for any $v \in W$, there exists $v' \in W$ with $d(v, v') \gtrsim \kappa^{1/\beta}\rho$
    \end{enumerate}
\end{corollary}

\begin{definition}\label{def:broadlyspread}
    If a finite multiset $V'$ has property (ii) from Corollary \ref{cor:localbroad}, we say that $V'$ is $\beta$-broadly spread. 
\end{definition}

To be precise in this definition, we would need to keep track of the implicit error in $d(v, v') \gtrsim \kappa^{1/\beta}\rho$. We only apply this definition when the implicit error is $\sim 1$, either because we constructed $V'$ to be broad with error $\sim 1$ itself, or $V' \subset V$ with $|V'| \sim |V|$ and $V$ broad with error $\sim 1$. The result is that we can always take the implicit constant to be $\le C^{1/\beta}$ for some large fixed value $C$, and while this is not strictly a $\sim 1$ term, $C^{1/\beta}\lessapprox_{\de} 1$ for $\de$ sufficiently small, which is always be acceptabe for our purposes. So to simplify the exposition, we do not keep track of the implicit error terms in the definition of broadly spread. 

This concludes our discussion of Lemma \ref{lem:localbroad}. As Definition \ref{def:broad} above suggests, we also need to consider broadness in tubes. We have two variants of this proposition, one where we retain essentially all the mass of the shading at the cost of having broadness multiplicity substantially larger than $1$ and the other where sacrifice a non-trivial share of the mass of the shading while having broadness multiplicity one. Part (A) of this lemma is a rephrasing of \cite[Corollary 7.10]{WZ25}. Part (B) of the lemma is an easy consequence of Part (A).

We introduce a 

\begin{lemma}[{\cite[Corollary 7.10]{WZ25}}]\label{lem:tubebroad}
    For any $\beta > 0$, suppose $(\T, Y)_{\de}$ is a shading on a set of $\de$-tubes. Then 
    
    \begin{enumerate}[label=(\Alph{*})]
        \item There exists a scale $\rho \in [\de, 1]$ and a pair $(\T', Y')$ with $\T' \subset \T$, $Y'(T) \subset Y(T)$, and $\sum_{T \in \T'} |Y'(T)|_{\de} \gtrapprox_{\de} \sum_{T \in\T} |Y(T)|_{\de}$. The tubes $\T'$ have a balanced, partitioning cover $\T_{\rho}$ such that $\T_{\rho}$ is a $\beta$-broad cover with error $\lessapprox_{\de} 1$ of $(\T',Y')$ and has broadness multiplicity $\lessapprox_{\de} \rho^{-\beta}$
        \item There exists a scale $\rho \in [\de, 1]$ and a pair $(\T', Y')$ with $\T' \subset \T$, $Y'(T) \subset Y(T)$, and $\sum_{T \in \T'} |Y'(T)|_{\de} \gtrapprox_{\de} \rho^{-\beta}\sum_{T \in\T} |Y(T)|_{\de}$. The tubes $\T'$ have a balanced, partitioning cover $\T_{\rho}$ such that $\T_{\rho}$ is a $\beta$-broad cover with error $\lessapprox_{\de} 1$ of $(\T', Y')$ with broadness multiplicity $1$. 
    \end{enumerate}
\end{lemma}
\begin{proof}[Proof of part (ii)]
    First, apply Lemma \ref{lem:tubebroad} (A) to $(\T, Y)$ with the given value of $\beta$. Let $(\T', Y')$ denote the $\gtrapprox_{\de} 1$ dense subsets and $\T_{\rho}$ the cover. At each cube $Q$, we have $\lessapprox_{\de} \rho^{-\beta}$ different tubes $T_{\rho}$ so that $Q \in \bigcup_{T \in \T'[T_{\rho}]} Y'(T)$, so we can find an element of $T_{\rho}$ such that \[|\{T \in \T'[T_{\rho}] : Q \in Y'(T) \}| \gtrapprox_{\de} \rho^{\beta} |\{T \in \T' : Q \in Y'(T)\}|.\]We set $Q \in Y''(T)$ if and only if $Q \in Y'(T)$ and $T \subset T_{\rho}$. Carry this procedure out at each cube $Q \in Y'(\T')$. since at each cube of $(\T', Y')$, we retained at least a $\gtrapprox_{\de} \rho^{\beta}$ share of the tubes with incident shading, so \[\sum_{T \in \T'} |Y''(T)|_{\de }\gtrapprox_{\de} \rho^{\beta} \sum_{T \in \T'} |Y'(T)|_{\de} \gtrapprox_{\de} \rho^{\beta} \sum_{T \in \T} |Y(T)|_{\de}.\]For each cube $Q \in Y''(\T')$, we have exactly one $T_{\rho} \in \T_{\rho}$ such that $Q \in Y''(\T'[T_{\rho}])$, but we keep $Q \in Y''(T)$ if $Q \in Y'(T)$ and $T \subset \T_{\rho}$. It follows that $\{T \in \T'[T_{\rho}] : Q \in Y''(T)\} = \{T \in \T'[T_{\rho}] : Q \in Y'(T)\}$, so since the latter is broad in the $\rho$-ball of directions determined by $T_{\rho}$, the former is as well. We conclude that $\T_{\rho}$ remains a $\beta$-broad cover with error $\lessapprox_{\de} 1$ of $(\T', Y'')$.
\end{proof}

Next, we introduce a combinatorial lemma which helps turn incidence properties that hold ``on average'' to incidence properties that hold ``uniformly''. This lemma proven in Wang and Zahl's resolution of the sticky Kakeya problem in $\R^3$ \cite[Lemma 8.2]{WZ22}. Wang and Zahl attribute the lemma to a generalization of a result by Dvir and Gopi \cite{DG15}. Before stating the lemma, we give a pair of relevant definitions.

\begin{definition}
    For finite sets $A_1, \dots, A_k$, let $G \subset A_1 \times \dots \times A_k$ be a $k$-partite hypergraph. The vertices of this hypergraph are the elements of $A_1 \sqcup \dots \sqcup A_k$ and the edges are $k$-tuples in $G$.

    The density of a $k$-partite hypergraph $G$ is defined to be $\frac{|G|}{\prod_{i=1}^k |A_k|}$. We say a $G$ is uniformly $c$-dense if for any partition $I \sqcup J$ of $\{1, \dots, k\}$ and any choice of $(a_i)_{i \in I}$, we have \[|\{(b_i)_{i=1}^k : a_i = b_i \text{ for all }i \in I\}| \ge c \prod_{j \in J} |A_j|.\]
\end{definition}

\begin{lemma}[{\cite[Lemma 8.2]{WZ22}}]\label{lem:WZ2547}
    Let $A_1, \dots, A_k$ be finite sets and $G \subset A_1 \times \dots \times A_k$ a $k$-uniform hypergraph (i.e. each hyper edge connects $k$ elements). Let $d = \frac{|G|}{\prod_{i=1}^k |A_k|}$ denote the density of $G$. There exists a uniformly $\sim_k d$-dense subgraph $G'$ of $G$ with $|G'| \ge \frac{1}{2}|G|$. 
\end{lemma}

\begin{remarkc}\label{rmk}
    Mostly, we use this in the case $k = 2$. In that case, we can express the uniformity in a much simpler way: for each $v \in A_1$, $\deg(v) \gtrsim \frac{|G|}{|A_2|}$ and for each $v \in A_2$, $\deg(v) \gtrsim \frac{|G|}{|A_1|}$. This special case appears in \cite[Lemma 4.7]{WZ25}.
\end{remarkc}
Finally, we define the multiplicity of a cube in a shading on a collection of convex sets. 

\begin{definition}
    If $Y$ is a shading on a collection of convex sets $\calR$, then the multiplicity of $Y$ at a $\de$-cube $Q$ is $\mu(Q) = \{R \in \calR : Q \in Y(R)\}$. 

    To clarify which multiplicities are associated with which shadings, we use the same subscripts and superscripts for the shadings as the multiplicities. For example, $\mu_3^{T_u}$ will denote the multiplicity for the shading $Y_3^{T_u}$. 
\end{definition}

We are now ready to proceed with the proof of Proposition \ref{prop:largeprism}. We recall the statement here.

\begin{proposition*}[\textbf{\ref{prop:largeprism}}]
    For any $\omega, \e, \tau > 0$, there exist $\eta, \de_0 > 0$ such that the following holds for all $\overline{\de} \in (0, \de_0)$. Suppose there exist values $\de, \De > 0$ with $\overline{\de} \le \de \le \De$ and $\de \le \overline{\de}^{\e}\De$ and a collection $\calR$ of $\de \times \De \times 1$ prisms satisfying the $t$-Frostman Convex Wolff Axiom with error $\de^{-\eta}$ and with a $\de^{\eta}$-dense shading $Y$ on $\calR$.

    Then one of the following conclusions must occur.

    \begin{enumerate}[label = \Alph*.]
        \item There exists $r > \overline{\de}^{-\eta}\de$ and an $r$-ball $B_r$ such that \begin{equation}\label{eqn:lpdense}\left|\bigcup_{R \in \calR} Y(R) \cap B_r\right|_{\de} \ge \left(\frac{r}{\de}\right)^{2t + 1 - \omega};\end{equation}or 
        \item There exist $\de' \le \De'$ with $\de' \in [\overline{\de}^{-\e\omega/6}\de, \overline{\de}^{\e/2}]$ and $\De' = \min\left(1, \frac{\De\de'}{\de}\right)$ and a collection $\calR'$ of $\de' \times \De' \times 1$ prisms satisfying the $t$-Frostman Convex Wolff Axiom with error $\de'^{-\tau}$ and a $\de'^{\tau}$-dense shading $Y'$ on $\calR'$ such that for all $R \in \calR'$ \begin{equation}\label{eqn:lpcover}Y'(R) \subset N_{\de'}\left(\bigcup_{R \in \calR[R']} Y(R)\right).\end{equation}
    \end{enumerate}
\end{proposition*}

We are in largely the same situation as the induction step \cite[Lemma 4.2]{WZ24} and our proof follows similar lines as their proof, with some of the technical steps moved to lemmas. We split our proof into three pieces: the setup; the slightly transverse case; and the slightly broad case. 

In the setup piece of the proof, we ensure the multiplicity of the set of prisms is fairly large, then use Lemma \ref{lem:tubebroad} to find a broadness parameter $u$ which describes the spread of the long direction of the prisms in $\calR$ and a collection of $u$-tubes $\T_u$ so that the long directions of $\calR$ are broad in $\T_u$. We then apply Lemma \ref{lem:localbroad} to find a transversality parameter $\theta$ which describes the spread of the normal direction to the prisms in $\calR^{T_u}$ for each $T_u$. To count the number of $\de$-cubes in our shading after rescaling through $u$-tubes, we need to pigeonhole to cover the shading with a disjoint collection of $\de \times \de \times \frac{\de}{u}$-tubelets all incident to $\sim M$ many $\de$-cubes. We define $u$ so that $u \ge \De$ and $\theta$ so that $\theta \ge \frac{\de}{\De}$. We apply the $t$-Frostman Convex Wolff Axiom to conclude that either $\theta \gg \frac{\de}{\De}$ or $u \gg \De$. As an aid to the reader we pause and recap the properties we established in the setup piece before proceeding with the remainder of the proof.

After the recap, we begin with the slightly transverse case, which is when $\theta \gg \frac{\de}{\De}$. In this case, for some $T_u$, we break the collection $\calR^{T_u}$ into a collection $\calP$ of $\frac{\de}{u} \times \frac{\De}{u} \times \frac{\De}{u}$ prisms coming from subsets of individual $\frac{\de}{u} \times \frac{\De}{u} \times 1$ prisms $R^{T_u} \in \calR^{T_u}$. We refine $\calP$ so that each of the prisms in $\calP$ has large intersection with the shading $Y$, then use a brush argument to reach Conclusion A. This completes the slightly transverse case.

After the slightly transverse case, we begin the slightly broad case. We define larger dimensions $\de', \De'$ satisfying the assumptions of Proposition \ref{prop:largeprism}, Conclusion B. We then use a brush method with each $R^{T_u} \in \calR^{T_u}$ to prove that $\calD_{\de'/u}(\calR^{T_u})$ can be arranged to largely fill out a $\frac{\de'}{u} \times \frac{\De'}{u} \times 1$ prism concentric with $R^{T_u}$. Undoing the rescaling through $T_u$ and then taking the union over all $T_u$, we arrive at a collection of $\de' \times \De' \times 1$ prisms $\calR'$ satisfying the requirements of Conclusion B. This completes the proof.

\begin{proof}~
    \subsection*{Reduction I}\vspace{-7pt}In this reduction, we refine the shading to have constant multiplicity $\mu_1$ and prove that if $\mu_1$ is too small, then we have Conclusion A.

    Let $(\calR, Y)$ be as described in the statement of the proposition. We begin with a constant multiplicity refinement. To accomplish this, we dyadically pigeonhole on the value of $\mu(Q) = |\{R \in \calR:Q \in Y(R)\}|$ over dyadic $\de$-cubes $Q \in \bigcup_{R \in \calR} Y(R)$. Denote the resulting multiplicity by $\mu_1$ and the resulting shading by $Y_1$. This shading satisfies \begin{equation}\label{eqn:prop5shading1}\sum_{R \in \calR} |Y_1(R)|_{\de} \gtrapprox_{\de} \sum_{R\in \calR} |Y(R)|_{\de} \ge |\calR|\de^{\eta}\frac{\De}{\de^2} \ge \de^{\eta-1} \left(\frac{\De}{\de}\right)^t|\calR|.\end{equation}Suppose $\mu_1 \le \overline{\de}^{-\e\omega/2} |\calR|(\de \De)^t$. Then \[\left|\bigcup_{R \in \calR} Y_1(R)\right|_{\de} \ge \frac{\sum_{R \in \calR} |Y_1(R)|_{\de}}{\mu_1} \gtrapprox_{\de} \de^{\eta-1}\overline{\de}^{\e\omega/2}(\de \De)^{-t}\left(\frac{\De}{\de}\right)^t.\]
    Since $\de \le \overline{\de}^{\e}\De \le \overline{\de}^{\e}$, we know that $\overline{\de}^{\e\omega/2} \ge \de^{\omega/2}$. We conclude that $\left|\bigcup_{R \in \calR} Y_1(R)\right|_{\de} \gtrapprox_{\de} \de^{-2t-1+\eta + \omega/2}$. So long as $\eta < \omega/4$ and $\de$ is sufficiently small, we conclude that $\left|\bigcup_{R \in \calR} Y_1(R)\right|_{\de} \ge \de^{-2t-1 + \omega}$. We arrive at Conclusion A, completing the proof. 

    Going forward, we assume that \begin{equation}\label{eqn:m1bound}\mu_1 \ge \overline{\de}^{-\e\omega/2}|\calR|(\de \De)^t.\end{equation}

    \subsection*{Reduction II}In this reduction, we find a cover of $\calR$ by $u$-tubes $\T_u$ so that the long direction of $\calR$ is broad in the $u$-tubes.

    Recall that for a collection $\mathcal{A}$ of subsets of $\R^3$, we defined $\calD_{\rho}(\mathcal{A}) = \{N_{\rho}(A) : A \in \mathcal{A}\}$. Set $\T_{\De} := \calD_{\De}(\calR)$, a set of $\De$-tubes. For $T_{\De} \in \T_{\De}$, we equip $T_{\De} = N_{\De}(R)$ with the shading $Y_1(T_{\De}) = Y_1(R)$. Now, we apply Lemma \ref{lem:tubebroad} (B) to $\T_{\De}$ with $\beta_1 = \beta_1(\omega, \e)$ a small parameter to be determined in the course of the proof. We arrive at parameter $u \ge \De$ and some $(\T'_{\De}, Y_2)$ with \begin{equation}\label{eqn:temp}\sum_{T \in \T'_{\De}} |Y_2(T)|_{\de} \gtrapprox_{\de} u^{\beta_1}\sum_{T \in \T_{\De}} |Y_1(T)|_{\de}.\end{equation}We set $\calR_1 = \{R \in \calR : N_{\De}(R) \in \T'_{\De}\}$ and $Y_2(R) := Y_2(N_{\De}(R))$. Our construction of $(\T_{\De}, Y_1)$ and (\ref{eqn:temp}) implies that \begin{equation}\label{eqn:21shadingdensity}\sum_{R \in \calR_1} |Y_2(R)|_{\de} \gtrapprox_{\de} \de^{\beta_1}\sum_{R \in \calR} |Y_1(R)|_{\de}.\end{equation}This collection of prisms $\calR_1$ is covered by $\T_u$ a set of $u$-tubes so that for any $Q \in Y_2(\calR_1)$ and any $T_u \in \T_u$, $\{\dir(R) : R \in \calR_1[T_u], Q \in Y_2(R)\}$ is $\beta_1$-broadly spread\footnote{Definition \ref{def:broadlyspread} and the discussion that precedes it applies equally well to the long direction of these prisms} with error $\sim 1$ inside the $u$-cube centered at $\dir(T_u) \subset S^2$.

    We conclude with another constant multiplicity refinement. In a small abuse of notation, we refer to the refined shading as $Y_2$ and denote the multiplicity by $\mu_2$. The change in mass of the shading is negligible and we either retain all incidences with a cube or remove them all, so we still have the same broadness properties.

    \subsection*{Reduction III} In this reduction, we break the shading in each $T_u$ up into short tubelets containing about the same number of $\de$-cubes, which allows us to control the mass of the shading after rescaling. 

    Recall that $\calR_1^{T_u}$ is a collection of approximately $\frac{\de}{u} \times \frac{\De}{u} \times 1$ prisms and $Y_{\text{full}}(\calR_1^{T_u})$ is the collection of dyadic $\frac{\de}{u}$-cubes covering each element of $\calR_1^{T_u}$. Then $\phi^{-1}_{T_u}(Y_{\text{full}}(\calR_1^{T_u}))$ is a collection of disjoint $\de \times \de \times \frac{\de}{u}$-tubelets $\mathcal{T}_{T_u}$ covering each element $\calR_1[T_u]$. We write \[E = \{ (\lambda, R) \in \mathcal{T}_{T_u} \times \calR_1 : \lambda \cap Y_2(R) \neq \emptyset\}.\]For each pair $(\lambda, R) \in E$, $\lambda$ intersects with some number of $\de$-cubes from $Y_2(R)$. We denote this number by $M_{\lambda, R}$. After pigeonholing on the value of $M_{\lambda, R}$ over $E$, we arrive at a subset $E'$ and a value $M_{T_u}$ with $M_{\lambda, R} = M_{T_u}$ for each $(\lambda, R) \in E'$ and a shading \[Y'_3(R) = \{Q \in Y_2(R) : Q \cap \lambda \neq \emptyset \text{ for some }\lambda \text{ with }(\lambda, R) \in E')\}.\]We ensure that \begin{equation}\label{eqn:we2}\sum_{R \in \calR_1[T_u]} |Y'_3(R)|_{\de} \gtrapprox_{\de} \sum_{R \in \calR_1[T_u]} |Y_2(R)|_{\de}.\end{equation}We use an argument in the style of the Vitali covering theorem to reduce $Y'_3(R)$ to $Y_3(R)$ so that any $\de$-cube $Q \in Y_3(R)$ intersects at most one tubelet from $\mathcal{T}_{T_u}$ and \[\sum_{R \in \calR_1[T_u]} |Y_3(R)|_{\de} \gtrsim \sum_{R \in \calR_1[T_u]} |Y'_3(R)|_{\de}.\]Together with (\ref{eqn:we2}), we conclude that \begin{equation}\label{eqn:pre23shadingdensity} \sum_{R \in \calR_1} |Y_3(R)|_{\de} \gtrapprox_{\de} \sum_{R \in \calR_1[T_u]} |Y_2(R)|_{\de}.\end{equation}

    We claim that we can find a value $M$ and a collection $\T'_u \subset \T_u$ such that $M_{T_u} \sim M$ such that each $T_u \in \T'_u$, if $\calR_2 = \{R \in \calR_1: R \subset T_u \text{ for some }T_u \in \T'_u\}$, then $|\calR_2| \gtrapprox_{\de} \de^{\beta_1 + \eta} |\calR|$, and $\sum_{R\in\calR_2} |Y_3(R)|_{\de} \gtrapprox_{\de} \sum_{R \in \calR_1} |Y_2(R)|_{\de}$. To see this, define for dyadic values of $M$ $\T_{u,M} = \{T \in \T_u : M_T \in [M, 2M)\}$ and $\calR_M =\{R \in \calR_1[T_u] : T_u \in \T_{u,M}\}$. Note that $\calR_1 = \bigsqcup_{M \text{ dyadic}} \calR_M$ and $\sum_{R \in \calR_1} |Y_3(R)|_{\de} = \sum_{M \text{ dyadic}}\sum_{R \in \calR_M} |Y_3(R)|_{\de}$. Dyadically pigeonholing (note that $M \in [0, \sim \de^{-1}] \cap \Z$), we can find a value of $M$ so that $\sum_{R \in \calR_M} |Y_3(R)|_{\de} \gtrapprox_{\de} \sum_{R \in \calR_1} |Y_3(R)|_{\de}$. Write $\calR_2 := \calR_M$ and $\T_u' := \T_{u, M}$. We claim that $|\calR_2| \gtrapprox_{\de} \de^{\beta_1 + \eta} |\calR|$. To see this, note that by (\ref{eqn:prop5shading1}), (\ref{eqn:21shadingdensity}), and (\ref{eqn:pre23shadingdensity}), as well as our initial assumption that $|Y(R)|_{\de} \ge \de^{-2+\eta}\De$, \begin{equation}\label{eqn:we3}\sum_{R \in \calR_2} |Y_3(R)|_{\de} \gtrapprox_{\de} \sum_{R \in \calR_1} |Y_2(R)|_{\de} \gtrapprox_{\de} \de^{\beta_1} \sum_{R \in \calR} |Y_1(R)|_{\de} \gtrapprox_{\de} \de^{\beta_1}\sum_{R \in \calR} |Y(R)|_{\de} \gtrapprox_{\de} \de^{-2+ \beta_1 + \eta}\De|\calR|.\end{equation}On the other hand, $\sum_{R \in \calR_1} |Y_3(R)|_{\de} \le |\calR_1|\de^{-2}\De$. As promised, we see that \begin{equation}\label{eqn:R2bound}|\calR_2| \gtrapprox_{\de} \de^{\beta_1 + \eta} |\calR|.\end{equation}By our definition of $\calR_2$ and (\ref{eqn:pre23shadingdensity}), we also have \begin{equation}\label{eqn:babycapybara}\sum_{R\in\calR_2} |Y_3(R)|_{\de} \gtrapprox_{\de} \sum_{R \in \calR_1} |Y_2(R)|_{\de}.\end{equation}
    
    Next, we bound $M$ from below. From (\ref{eqn:we3}), we know that \[\sum_{R \in \calR_2} |Y_3(R)|_{\de} \gtrapprox_{\de} \de^{\beta_1 + \eta -2}\De|\calR|.\]Together with (\ref{eqn:babycapybara}) and the obvious bound $|\calR_2| \le |\calR|$, we conclude that $\sum_{R\in\calR_2} |Y_3(R)|_{\de} \gtrapprox_{\de} \de^{\beta_1 + \eta } \frac{\De}{\de^2}|\calR_2|$. In particular, there must exist some $R \in \calR_2$ with $|Y_3(R)|_{\de} \gtrapprox_{\de} \de^{\beta_1 + \eta}\frac{\De}{\de^2}$. This choice of $R$ intersects $\lesssim \frac{\De u}{\de^2}$ many tubelets from $\mathcal{T}_{T_u}$ and $Y_3(R)$ is supported on those tubelets, so some tubelet must intersect $\gtrapprox_{\de} \de^{\beta_1 + \eta }\frac{1}{u}$ many $\de$-cubes from $Y_3(R)$. Since each tubelet intersects $\sim M$ $\de$-cubes from $Y_3(R)$ and each $\de$-cube in $Y_3(R)$ intersects one tubelet, we conclude that \begin{equation}\label{eqn:Mbound2}M \gtrapprox_{\de} \de^{\beta_1 + \eta }\frac{1}{u}.\end{equation}

    Finally, define $Y^{T_u}(R^{T_u}) = \bigcup_{\lambda: (\lambda, R) \in E'} \phi_{T_u}(\lambda)$. This is a collection of dyadic $\frac{\de}{u}$-cubes intersecting $R^{T_u}$ and hence is a well-defined shading. We know that $|Y^{T_u}(R^{T_u})|_{\de/u}M \sim |Y_3(R)|_{\de}$, so since $M \lesssim u^{-1}$, we conclude that $|Y^{T_u}(R^{T_u})|_{\de/u} \gtrsim u|Y_3(R)|_{\de}$. Using (\ref{eqn:we3}), we see that \begin{equation}\label{eqn:rescaledmultbound}\sum_{T_u \in \T'_u} \sum_{R \in \calR_2[T_u]} |Y^{T_u}(R^{T_u})|_{\de/u}\gtrsim u\sum_{R \in \calR_2} |Y_3(R)|_{\de} \gtrapprox_{\de} u\de^{-2 + \beta_1 + \eta}\De|\calR|.\end{equation}

    \subsection*{Reduction IV}In this reduction, we refine the shading $Y_3$ to some $Y_4$ so that after rescaling through $u$-tubes, for each $\de$-cube $Q \in Y_3(\calR_2)$, the set of prisms incident to $Q$ is transverse with parameter $\theta$. 

    To accomplish this, we apply a pointwise tangential broadness reduction. At each $\de/u$-cube $Q \in Y^{T_u}(\calR_2)$, we have a collection $\calR_Q= \{R \in \calR_2 : Q \in Y_3(R)\}$ of incident prisms, all contained in some $T_u \in \T_u$. Elements $R^{T_u} \in \calR_Q^{T_u}$ are $\frac{\de}{u} \times \frac{\De}{u} \times 1$ prisms and hence have a center plane $\Pi$ with normal direction defined up to error $\frac{\de}{\De}$. This center plane has a normal vector $n$, which we denote $n(R^{T_u}) \in S^2$. Let $\tilde{V}_Q = \{n(R^{T_u}) : R \in \calR_Q\}$. Apply Corollary \ref{cor:localbroad} to $\tilde{V}_Q$ with parameter $\beta_2 > 0$, a small parameter to be determined in the course of the proof. Let $\theta_Q$ denote the resulting broadness parameter and $V_Q \subset \tilde{V}_Q$ the $\beta_2$-broadly spread set satisfying $|V_Q| \gtrapprox_{\de} \de^{\beta_2}|V_Q|$. For dyadic choices of $\theta$, set \[Y^{T_u}_{1, \theta}(R^{T_u}) = \{Q \in Y^{T_u}(R^{T_u}) : \theta_Q \in [\theta, 2\theta], n(R) \in N_{\de}(V'_Q)\}.\]Set $\tilde{Y}^{T_u}(R^{T_u}) = \bigcup_{\theta \text{ dyadic}} Y_{1, \theta}^{T_u}(R^{T_u})$. We retained a $\gtrapprox_{\de} \de^{\beta_2}$ fraction of the incident prisms to each $\frac{\de}{u}$-cube, so $\sum_{R^{T_u} \in \calR^{T_u}} |\tilde{Y}^{T_u}(R^{T_u})|_{\de/u}\gtrapprox_{\de} \de^{\beta_2}\sum_{R^{T_u} \in \calR^{T_u}} |Y^{T_u}(R^{T_u})|_{\de/u}$. Dyadically pigeonholing, we can find a single value $\theta$ and a single multiplicity $\mu^{T_u}_1$ \emph{not depending on $T_u$}\footnote{Most of the we use depend on the $T_u$ in the super-script. This multiplicity is related to $Y^{T_u}_1$ so to be notationally consistent we denote it $\mu^{T_u}_1$, but it does not depend on $T_u$.} such that $Y^{T_u}_{1,\theta}$ is a $\gtrapprox_{\de} 1$ dense subshading of $\tilde{Y}$ with constant multiplicity $\mu^{T_u}_1$. Set $Y_1^{T_u} := Y^{T_u}_{1,\theta}$. This shading satisfies \begin{equation}\label{eqn:10shadingbound}\sum_{T_u \in \T'_u} \sum_{R \in \calR_2[T_u]} |Y_1^{T_u}(R^{T_u})|_{\de/u} \gtrapprox_{\de} \de^{\beta_2}\sum_{T_u \in \T'_u} \sum_{R \in \calR_2[T_u]} |Y^{T_u}(R^{T_u})|_{\de/u}\end{equation} let \[Y_4(R) = \{Q \in Y_3(R) : Q \cap \lambda \neq \emptyset \text{ for some }\lambda \in \phi^{-1}_{T_u}(Y_1^{T_u}(R^{T_u}))\}.\]Since each tubelet intersects the same number of $\de$-cubes in $Y_3$ and in $Y_4$ and $Y_1^{T_u}$ is a $\gtrapprox_{\de} \de^{\beta_2}$ dense subshading of $Y^{T_u}$, we know that \begin{equation}\label{eqn:34shadingbound2}\sum_{T_u \in \T_u}\sum_{R \in \calR_2[T_u]} |Y_4(R)|_{\de} \gtrapprox_{\de} \de^{\beta_2}\sum_{T_u \in \T_u}\sum_{R \in \calR_2[T_u]} |Y_3(R)|_{\de}.\end{equation}

    \subsection*{Reduction V}In this reduction, we prove that either $\theta \ge \de^{-20\gamma/\omega}\frac{\de}{\De}$ or $u \ge \overline{\de}^{-\e\omega/6}\De $.

    We define in the course of the remainder of this proof a parameter $\gamma \sim \beta_1 + \beta_2 + \eta$, which we can choose to be very small relative to $\tau, \e,$ and $\omega$ by making $\beta_1, \beta_2$, and $\eta$ sufficiently small. Suppose that both $u \le \overline{\de}^{-\e\omega/6}\De$ and $\theta \le \de^{-20\gamma/\omega}\frac{\de}{\De}$. We aim to reach a contradiction. 
    
    Applying (\ref{eqn:34shadingbound2}), (\ref{eqn:babycapybara}), and (\ref{eqn:21shadingdensity}) we see that $\sum_{R \in \calR_2} |Y_4(R)|_{\de} \gtrapprox_{\de} \de^{\beta_1 + \beta_2 + \eta}\sum_{R \in \calR_1} |Y_1(R)|_{\de}$. As long as $\de$ is sufficiently small, we can ensure $\sum_{R \in \calR_2} |Y_4(R)|_{\de} \ge \de^{\gamma}\sum_{R \in \calR_2} |Y_1(R)|_{\de}$. It follows that there exists a $\de$-cube $Q$ such that $\mu_4(Q) \ge \de^{\gamma}\mu_1(Q)$. Applying our bound (\ref{eqn:m1bound}) for $\mu_1$, we conclude that \begin{equation}\label{eqn:m4bound}\mu_4(Q) \ge \de^{\gamma}\overline{\de}^{-\e\omega/2}|\calR|(\de \De)^t.\end{equation}Let $\calR_Q = \{R \in \calR_2 : Q \in Y_4(R)\}$. Each $R \in \calR_Q$ must be contained in a common $T_u \in \T'_u$ and $Y_1^{T_u}(R^{T_u})$ each must contain a dyadic $\frac{\de}{u}$-cube intersecting $\phi_{T_1}(Q)$. It follows that $\{n(R^{T_u}) : R \in \calR_{Q}\} \subset S^2$ is contained in a $\theta$-ball with center $n$, so all elements of $\calR^{T_u}_{Q}$ are contained in $N_{3\theta}\left(n^{\perp} + \phi_{T_u}(Q)\right) \cap B(0,1)$, a $3\theta \times 1 \times 1$ prism. This is mapped by $\phi^{-1}_{T_u}$ to a $3\theta u \times u \times 1$ prism containing each prism in $\calR_{Q}$. Since $\calR$ satisfies the $t$-Frostman Convex Wolff Axiom with error $\de^{-\eta}$, we know that \begin{equation}\label{eqn:lynx1}\mu_{4}(Q) = |\calR_{Q}| \lesssim \de^{-\eta}(u^2\theta)^t |\calR|\end{equation}
    Comparing (\ref{eqn:lynx1}) with (\ref{eqn:m4bound}), we conclude that $(u^2 \theta)^t\de^{-\eta} \gtrapprox_{\de} \de^{\gamma}\overline{\de}^{-\e\omega/2}(\de \De)^t$. Since we have assumed $u \le \overline{\de}^{-\e\omega/6}\De$ and $\theta \le \de^{-20\gamma/\omega}\frac{\de}{\De}$, we see that \[(u^2 \theta)^t \de^{-\eta}\le \de^{-\eta}\left(\overline{\de}^{-\e\omega/3}\de^{-20\gamma/\omega}\de \De\right)^t\]and hence \[\de^{\gamma}\overline{\de}^{-\e\omega/2} \lessapprox_{\de}\de^{-\eta} \overline{\de}^{-t\e\omega/3}\de^{-20t\gamma/\omega}.\]Rearranging and recalling that $t \le 1$ and $\overline{\de} \le \de$, we conclude that \[\de^{-\e\omega/6} \le \overline{\de}^{-\e\omega/6} \lessapprox_{\de} \de^{-20t\gamma/\omega - \gamma -\eta}.\]So long as $\beta_1, \beta_2, \eta, \gamma, \de_0$ are sufficiently small depending on $\e$, $\omega$, and $t$, we arrive at a contradiction. We conclude that either $\theta \ge \de^{-20\gamma/\omega}\frac{\de}{\De}$ or $u \ge \de^{-\e\omega/4}\De$. 

    \subsection*{Recap of the proof so far.} At this point, we have found values $\theta, u$, a refinement $\calR_2$ of $\calR$, a set of $u$-tubes $\T'_u$, and $Y_4$ a subshading of $Y_2$ a subshading of $Y$. We recap the properties established for these objects roughly in the order in which we use them in the remainder of the proof. If we repeat a numbered expression from earlier in the proof, we give it the same number.

    \begin{enumerate}[label = (\alph*)]
        \item There exists a value $\mu_1^{T_u}$ \emph{not depending on $T_u$} such that for each $T_u \in \T'_u$, $Y_1^{T_u}$ is a shading on $\calR_2^{T_u}$ with multiplicity $\mu_1^{T_u}$. Combining (\ref{eqn:10shadingbound}) and (\ref{eqn:rescaledmultbound}), we see that $Y_1^{T_u}$ satisfies \begin{equation}\label{eqn:1rescaledmultbound}\sum_{T_u \in \T'_u} \sum_{R \in \calR_2[T_u]} |Y_1^{T_u}(R^{T_u})|_{\de/u}\gtrsim \de^{\beta_2} u\sum_{R \in \calR_2} |Y_3(R)|_{\de} \gtrapprox_{\de} u\de^{-2 + \beta_1 + \beta_2 + \eta}\De|\calR|.\end{equation} For each $Q \in Y_1^{T_u}(\calR_2^{T_u})$, $\{n(R^{T_u}) : Q \in Y_1^{T_u}(\calR_2^{T_u})\}$ is $\beta_2$-broadly spread in a $\theta$ ball. These properties were established in Reduction V.
        \item For each $\frac{\de}{u}$-cube $Q \in Y_1^{T_u}$, $\phi_{T_u}^{-1}(Q)$ is incident to $\sim M$ $\de$-cubes in $Y_3(\calR_2)$ and each $\de$-cube in $Y_3(\calR_2)$ is incident to $\phi_{T_u}^{-1}(Q)$ for $1$ $\frac{\de}{u}$-cube $Q$. Moreover \begin{equation*}
            M \gtrapprox_{\de} \de^{\beta_1+ \eta} \frac{1}{u}. \tag{\ref{eqn:Mbound2}}
        \end{equation*}
        \item By (\ref{eqn:34shadingbound2}) and (\ref{eqn:babycapybara}), we see that \begin{equation}\label{eqn:24shadingbound2}
            \sum_{R \in \calR_2} |Y_4(R)|_{\de} \gtrapprox_{\de} \de^{\beta_2}\sum_{R \in \calR_1} |Y_2(R)|_{\de}.
        \end{equation}Also, $Y_2$ has constant multiplicity $\mu_2$. 
        \item By (\ref{eqn:24shadingbound2}), (\ref{eqn:21shadingdensity}), and (\ref{eqn:prop5shading1}), we see that, so long as $\de_0$ is sufficiently small, \begin{equation}\label{eqn:40shadingbound}
            \sum_{R \in \calR_2 }|Y_4(R)|_{\de} \gtrapprox_{\de} \de^{\beta_1 + \beta_2 + \eta}|\calR_2|\frac{\De}{\de^2} \ge \de^{\gamma}|\calR_2|\frac{\De}{\de^2}
        \end{equation}
        \item From Reduction III, we have \begin{equation*}
            |\calR_2| \gtrapprox_{\de} \de^{\beta_1 + \eta} |\calR| \tag{\ref{eqn:R2bound}}
        \end{equation*}
    \end{enumerate}
    This concludes the recap. We now finish the proof in the dichotomous cases $\theta \ge \de^{-20\gamma/\omega}\frac{\de}{\De}$ and the case $u \ge \overline{\de}^{-\e\omega/6}\De$ and $\theta < \de^{-20\gamma/\omega}\frac{\de}{\De}$.
    \subsection*{The slightly transverse case} We assume $\theta \ge \de^{-20\gamma/\omega}\frac{\de}{\De}$ and use a brush argument for prisms to reach Conclusion A.
    
    For each $T_u \in \T_u$, $\calR_2^{T_u}$ is a set of approximately $\frac{\de}{u} \times \frac{\De}{u} \times \frac{\De}{u}$ prisms. We divide each $R^{T_u} \in \calR_2^{T_u}$ into a collection of $\frac{u}{\De}$ many $\frac{\de}{u} \times \frac{\De}{u} \times \frac{\De}{u}$ prisms, which we denote $\calP_R^{T_u}$. Set $\calP^{T_u} = \bigsqcup_{R \in \calR_2[T_u]} P_R^{T_u}$, $A = \bigsqcup_{T_u \in \T'_u} Y_1^{T_u}(\calR_2^{T_u})$ and $B = \bigsqcup_{R \in \calR_2}\calP_R$. Define the bipartite graph \[E = \{(Q, P^{T_u}) \in A \times B : Q \in Y_1^{T_u}(R^{T_u}), Q \cap P^{T_u} \neq \emptyset, P \in \calP_R\}.\]This graph has vertex set $A \sqcup B$ and edge set $E$. For each $R$, if $Q \in Y_1^{T_u}(R^{T_u})$, then $Q \cap P^{T_u} \neq \emptyset$ for $\sim 1$ different choices of $P^{T_u} \in \calP_R$, so \begin{equation}\label{eqn:lpE}|E| \sim \sum_{T_u \in \T'_u} \sum_{R \in \calR_2[T_u]} |Y_1^{T_u}(R)|_{\de/u} = \mu_1^{T_u}\left(\sum_{T_u \in \T'_u} |Y_1^{T_u}(P^{T_u})|_{\de/u}\right).\end{equation}Apply Lemma \ref{lem:WZ2547} to this graph and let $E'$ denote the resulting edge set. Define a shading on $B$ by $Y_2^{T_u}(P^{T_u}) = \{Q : (Q, P) \in E'\}$. By construction, $|A| = \sum_{T_u \in \T'_u} |Y^{T_u}_1(\calR_2(T_u))|_{\de/u}$, so by Lemma \ref{lem:WZ2547} and Remark \ref{rmk}, we know that for each $Q$, \begin{equation}\label{eqn:lpmubound}\mu_2^{T_u}(Q) \gtrsim \frac{|E'|}{|A|} \gtrsim \mu_1^{T_u}\end{equation}

    Since $|\calP_R| \lesssim \frac{u}{\De}$ for each $R \in \calR_2$, we know that $|B| \le |\calR_2| \frac{u}{\De}$. Let $\calP^{T_u}_1 = \{P \in \calP^{T_u} : Y_2^{T_u}(P) \neq \emptyset\}$.  Applying (\ref{eqn:1rescaledmultbound}), we see that \[|E| \sim \sum_{T_u \in \T'_u} \sum_{R \in \calR_2[T_u]} |Y_1^{T_u}(R)|_{\de/u} \gtrapprox_{\de} u\de^{-2 + \beta_1 + \beta_2+ \eta}\De|\calR|,\]so using Remark \ref{rmk} again, we know that for each $P^{T_u} \in \calP^{T_u}_1$, \begin{equation}\label{eqn:lp2bound}|Y_2^{T_u}(P^{T_u})|_{\de/u} \ge \frac{|E'|}{|B|} \gtrapprox_{\de} \de^{-2+\beta_1 + \beta_2 +  \eta}\De^2.\end{equation} 
    
    Recall that $\nu = \theta \De$. We claim that for a prism $P^{T_u} \in \calP_1^{T_u}$, if $G$ is the $3\frac{\nu}{u} \times 3\frac{\De}{u} \times 3\frac{\De}{u}$ prism concentric with $P^{T_u}$, then we can find some collection $\calP^{T_u}(G)$ satisfying the $1$-Frostman Convex Wolff Axiom with error $\le \de^{-5\gamma}$ in $G$. Since $\{n(R^{T_u}) : Q \in Y^{T_u}_1(R)\}$ is $\beta_2$-broadly-spread in a $\theta$-ball, by applying the definition $\beta_2$-broadly spread to $\{n(R^{T_u}) : Q \in Y^{T_u}_2(R)\} \subset \{n(R^{T_u}) : Q \in Y^{T_u}_1(R)\}$ and recalling (\ref{eqn:lpmubound}), we see that $\{n(R^{T_u}) : Q \in Y^{T_u}_2(R)\} $ is $\beta_2$-broadly spread is as well, so we conclude that there must exist a prism $P_Q^{T_u}$ with $\angle(n(P_Q^{T_u}), n(P^{T_u})) \sim \theta$ and $Q \in Y_2^{T_u}(P^{T_u})$. It follows that $P^{T_u}_Q \cap P^{T_u}$ is contained in a $\frac{\de}{u} \times \frac{\de}{\theta u} \times \De$ prism with center line $\ell_Q$. After pigeonholing on $Q$ to some collection $\tilde{Y}(P^{T_u}) \subset Y_2^{T_u}(P^{T_u})$ with $|\tilde{Y}(P^{T_u})|_{\de/u} \sim |Y_2^{T_u}(P^{T_u})|_{\de/u}$, we may assume for all choices of $Q$, $\ell_Q$ is parallel up to angle $\frac{1}{100}$ with a fixed direction $\ell_{||}$. Let $\ell_{\perp}$ denote a line segment in $P^{T_u}$ orthogonal to $\ell_{||}$ such that $\vol_{\R^1}(\ell_{\perp} \cap \tilde{Y}(P^{T_u})) \ge \de^{\gamma}\frac{\De}{u}$. This is possible by Fubini's theorem. 

    We know that $\vol_{\R^1}(\ell_{\perp} \cap \tilde{Y}(P^{T_u})) \ge \de^{\gamma}\frac{\De}{u}$ and that for each $Q \in \tilde{Y}(P^{T_u})$, $\vol_{\R^1}(\ell_{\perp} \cap \tilde{Y}(P_Q^{T_u})) \lesssim \frac{\de}{\theta u}.$ It follows that we can find $\frac{\de}{u}$-cubes $Q_1, \dots, Q_N$ and distinct prisms $P_{Q_1}^{T_u}, \dots, P_{Q_N}^{T_u} \subset G$ so that $P_{Q_i}^{T_u} \cap P_{Q_j}^{T_u} \cap \ell_{\perp} = \emptyset$, $N \gtrsim \de^{\gamma}\frac{\De \theta}{\de} = \de^{\gamma}\frac{\nu}{\de}$, and $N \le \frac{\nu}{\de}$. Set $\calP^{T_u} = \{P_{Q_i} : i = 1, \dots, N\}$. If this set of prisms satisfies the $1$-Katz-Tao Convex Wolff Axiom with error $\le \de^{-4\gamma}$ in $G$, then it satisfies the $1$-Frostman Convex Wolff Axiom with error $\le \de^{-5\gamma}$ in $G$. So suppose $C_{1\text{-KT-CW}}(\calP^{T_u}) \ge \de^{-4\gamma}$. We apply Lemma \ref{lem:WZ2546} to this set of prisms. We arrive at a collection of $\alpha \times \frac{\De}{u}  \times \frac{\De}{u}$ prisms $\mathcal{W}$ covering some $\calP' \subset \calP^{T_u}$ with $|\calP| \gtrapprox_{\de} |\calP^{T_u}|$, $\mathcal{W}$ a $\lessapprox_{\de} 1$ partitioning set, and $|\calP'[W]| \gtrapprox \de^{-4\gamma} \frac{\alpha u}{\de}$. Since $\mathcal{W}$ is a $\lessapprox_{\de} 1$ partitioning set, we conclude that $|\calW| \lessapprox_{\de} \frac{|\calP^{T_u}|}{|\calP'[W]|} \le \de^{4\gamma}\frac{\nu}{\alpha u}$. Since each $W \in \mathcal{W}$ contains some $P \in \calP^{T_u}$, we know elements of $W$ intersect $\ell_{\perp}$ with length $\sim \frac{\alpha}{\theta}$. It follows that \begin{equation}\label{eqn:mink1}\vol_{\R^1}\left(\bigcup_{W \in \calW} W \cap \ell_{\perp}\right) \lessapprox_{\de} \de^{4\gamma}\frac{\nu}{\alpha u}\frac{\alpha}{\theta} = \de^{4\gamma}\frac{\de}{u}.\end{equation}On the other hand, we know that \begin{equation}\label{eqn:mink2}\vol_{\R^1}\left(\bigcup_{P \in \calP'} P \cap \ell_{\perp}\right) = |\calP'| \frac{\de}{u\theta} \gtrapprox_{\de} \de^{\gamma} \frac{\De}{\de}.\end{equation}
    So long as $\de$ is sufficiently small depending on $\gamma$ and the implicit log factors, comparing (\ref{eqn:mink1}) with (\ref{eqn:mink2}), we arrive at a contradiction. It follws $\calP^{T_u}$ satisfies the $1$-Frostman Convex Wolff Axiom in $G$ with error $\le \de^{-5\gamma}$. 

    We then apply Lemma \ref{lem:RFPC} to $\calP^{T_u}(G)$ with the shading $Y_2^{T_u}$; with $\frac{\de}{u}$ taking the place of $\de$ and $\frac{\De}{u}$ taking the place of $\De$ in the statement of Lemma \ref{lem:RFPC}; $t=1$; and $\nu = \theta \De$. Since $\calP^{T_u}(G)$ satisfies the $1$-Frostman Convex Wolff Axiom in $G$ with error $\le \de^{-5\gamma}$, we take $C_1 = \de^{-5\gamma}$. Recalling (\ref{eqn:lp2bound}), the shading is $\ge \de^{\beta_1 + \beta_2 +  \eta}$ dense, but for convenience we take $C_2 = \de^{-\gamma}$. The conclusion of Lemma \ref{lem:RFPC} is that $|Y^{T_u}_2(\calP^{T_u})|_{\de/u} \gtrsim \de^{7\gamma} \frac{\nu \De^2}{\de^3}$, which is contained in a $\frac{3\nu}{u} \times \frac{3\De}{u} \times \frac{3 \De}{u}$ prism. Since $Y^{T_u}_2(\calP^{T_u})$ is contained in a $\frac{\nu}{u} \times \frac{\De}{u} \times \frac{\De}{u}$ prism $G$, we can find a $\frac{\nu}{u} \times \frac{\nu}{u} \times \nu$ prism $\tilde{G}$ with $|Y^{T_u}_2(\calP^{T_u}) \cap \tilde{G}|_{\de/u} \gtrsim \de^{7\gamma} \frac{\nu^3u}{\de^3}$ such that $\phi_{T_u}^{-1}(\tilde{G})$ is a dyadic $\nu$-cube $B_{\nu}$. When we undo the rescaling under $\phi_{T_u}$, we map a $\frac{\de}{u}$-cube in $Y^{T_u}_2(\calP^{T_u})$ to $\sim M$ $\de$-cubes in $Y_3(\calR_2)$. Different $\frac{\de}{u}$-cubes from $Y^{T_u}_2$ map to different $\de$-cubes from $Y_3(\calR_2)$, so $|Y_3(\calR_2) \cap B_{\nu}|_{\de} \gtrsim M\de^{7\gamma}\frac{\nu^3u}{\de^3}$. Recall from (\ref{eqn:Mbound2}) that $M \gtrapprox_{\de} \de^{\beta_1 + \eta}\frac{1}{u}$. As long as $\de$ is sufficiently small, we conclude that $M \ge \de^{\gamma}\frac{1}{u}$, so therefore \[|Y_3(\calR_2) \cap B_{\nu}|_{\de} \gtrapprox_{\de} M\de^{7\gamma}\frac{\nu^3u}{\de^3} \ge \de^{8\gamma}\frac{\nu^3}{\de^3}.\]
    Since we assumed that $\theta \ge \de^{-20\gamma/\omega}\frac{\de}{\De}$, we have that $\left(\frac{\nu}{\de}\right)^{-\omega} \le \de^{20\gamma}$ and hence \[|Y_3(\calR) \cap B_{\nu}|_{\de} \ge \left(\frac{\nu}{\de}\right)^{3-\omega} \ge \left(\frac{\nu}{\de}\right)^{2t+ 1 -\omega}.\]We also know that $\nu = \De \theta \ge \de^{-20\gamma/\omega}\de$. Taking $\eta$ small enough to ensure $\eta < 20\gamma/\omega$, we arrive at Conclusion A.

    \subsection*{The slightly broad case.}This case happens when $u \ge \overline{\de}^{-\e\omega/6}\De$ and $\theta \le \de^{-20\gamma/\omega}\frac{\de}{\De}$.

    Take $\theta_0 = \de^{-20\gamma/\omega}\frac{\de}{\De}$, $\de' = 2\theta_0 u$ and $\De' = \min(2\de^{-20\gamma/\omega}u, 1) = \min\left(\frac{\De\de'}{\de},1\right)$. We immediately have $\frac{\De'}{\de'} = \frac{\De}{\de}$, and $\de' \ge \overline{\de}^{-\e\omega/6}\de$ from our lower bound on $u$ and the definition of $\theta_0$. 
    
    There are two subcases two consider, depending on if $\de' < \De$ and if $\de' > \De$. The former is slightly more complicated, we treat it first.

    For each $R \in \calR_2$, we have a shading $Y_4(R)$. We have a cover of $R$ by by $\de \times \de' \times 1$ planks $\calW_R$. Define $\calW = \bigsqcup_{R \in \calR_1} \calW_R$. Note that $|\calW| = |\calR_1|\frac{\De}{\de'}$. We apply Lemma \ref{lem:WZ2547} to the hypergraph $E \subset Y_4(\calR) \times \calW \times \calR_2$ given by $(Q, W, R) \in E$ if $Q \in W$ and $W \subset R$. Since each $Q \in Y_4(R)$ can be in $\sim 1$ different choices of $P \in \calW_R$, applying (\ref{eqn:40shadingbound}), we know that \[|E| \sim \sum_{R \in \calR_2} |Y_4(R)|_{\de} \gtrsim \de^{\gamma}|\calR|\frac{\De}{\de^2}.\]

    Denote by $E'$ the resulting refinement from Lemma \ref{lem:WZ2547}. Define a shading $\tilde{Y}$ on $\calR_2$ by $Q \in \tilde{Y}(R)$ if there exists $W \in \calW_R$ such that $(Q, W, R) \in E'$. Define $\tilde{\calW}_R = \{P \in \calP_R : \exists Q \text{ s.t. }(Q, W, R) \in E'\}$ and $\tilde{\calR}_2 = \{R \in \calR_2 : \tilde{Y}(R) \neq \emptyset\}$. By Remark \ref{rmk} know that if $\tilde{\mu}$ is the multiplicity for $\tilde{Y}$, then $\tilde{\mu} \sim \mu_{4,\text{avg}}$, where $\mu_{4,\text{avg}}$ is the average multiplicity for $Y_4$. Since $Y_2(\calR_2)$ has constant multiplicity $\mu_2$, applying (\ref{eqn:24shadingbound2}) we see that $\tilde{\mu} \gtrapprox_{\de} \de^{\beta_2} \mu_2$. 
    
    For $W \in \calW$, \[|\tilde{Y}(W)|_{\de} \gtrsim \frac{|E|}{|\calW|} \ge \frac{\de^{\gamma}|\calR_2|\frac{\De}{\de^2}}{|\calW|} \ge \frac{\de^{\gamma}\de'}{\de^2},\]while we always have $|\tilde{Y}(W)|_{\de} \le \frac{\de'}{\de^2}$. We also know that |$\tilde{Y}(R)|_{\de} \gtrsim \frac{\de^{\gamma}|\calR|\frac{\De}{\de^2}}{|\calR|} \gtrsim \de^{\gamma} \frac{\De}{\de^2}$. It follows that for each $R \in \tilde{\calR}$, $|\tilde{\calW}_R| \ge \de^{\gamma} \frac{\De}{\de'}$. Applying similar reasoning for $\tilde{\calR}$ relative to $G'$ and recalling (\ref{eqn:R2bound}), we conclude that \begin{equation}\label{eqn:calrbound}|\tilde{\calR}| \ge \de^{\gamma}|\calR_2| \gtrapprox_{\de} \de^{\beta_1 + \gamma + \eta} |\calR|.\end{equation}

    Define $\calR'$ to be the multiset of $\de' \times \De' \times 1$ prisms concentric with elements of $\tilde{\calR}$. For $R' \in \calR'$, define $Y'(R') = N_{\de'}\left(\bigcup_{R \in \calR[R']} Y(R)\right)$. We prove that this shading is $\de^{-\tau}$ dense on $R'$. Fix a choice of $R'$. We know that $R'$ is the $\de' \times \De' \times 1$ prism concentric with some $R \in \tilde{\calR}$. Since $|\tilde{Y}(R)|_{\de} \ge \de^{\gamma} \frac{\De}{\de^2}$, we can pigeonhole to find a unit line $\ell$ intersecting $\gtrsim \de^{\gamma}\frac{1}{\de}$ many $\de$-cubes from $\tilde{Y}(R)$. Take a collection of $N \ge \de^{\gamma + 2\beta_2/\beta_1} \frac{u}{\De}$ many $\de^{-2\beta_2/\beta_1}\frac{\De}{u}$-separated cubes $Q_1, \dots, Q_N$ intersecting $\ell \cap \tilde{Y}(R)$. The long directions of the prisms through each point of $Y_2$ is $\beta_1$-broad with error $\lessapprox_{\de} 1$ and $\tilde{Y}$ has multiplicity $\tilde{\mu}(Q) \gtrapprox_{\de} \de^{\beta_2} \mu_2(Q)$, we conclude that $\tilde{Y}$ is $\beta_1$-broad with error $\lessapprox_{\de} \de^{\beta_2}$. It follows that, so long as $\de$ is sufficiently small, for each $Q_j$ there is a distinct prism $R_j \in \tilde{\calR}$ with $\angle(\dir(R_j), \dir(R)) \ge \de^{2\beta_2/\beta_1} u$. By the same reasoning we applied in proving (\ref{eqn:lynx1}), we know that $R_j$ is contained in the $u\theta \times u \times 1$ prism concentric with $R$, so since $\theta \le \theta_0$, $R_j \subseteq R'$. Now for each $R_j$, we have a collection $\tilde{\calW}_{R_j}$ of $\de \times \de' \times 1$ prisms with a $\ge \de^{\gamma}$-dense shading $\tilde{Y}$. It follows that $\calD_{\de'}(\tilde{\calP}_{R_j})$ is a set of $\de'$-tubes $\T_j$ with a $\ge \de^{\gamma}$-dense shading $Y'(T) = \calD_{\de'}(\tilde{Y}) \cap T$. 

    We claim that $\T_{R'} = \bigsqcup_{j = 1}^N \T_j$ satisfies the $1$-Frostman Convex Wolff Axiom in $R'$ with error $\le \de^{-\tau/10}$ so long as $\gamma, \beta_2/\beta_1, \de_0$ sufficiently small relative to $\tau, \omega$ (we choose $\beta_1$ before $\beta_2$ to make sure this is possible). If we can prove this claim, then applying Lemma \ref{lem:RFPA} with $t= 1$ and noting that if $\gamma, \beta_2/\beta_1, \de_0$ sufficiently small relative to $\tau, \omega$, then $Y'(T)$ is $\de^{-\tau/10}$-dense, we conclude that \begin{equation}\label{eqn:largerscalebound}|Y'(R')|_{\de'} \ge \de^{-\tau} \frac{\De'}{\de'^2}.\end{equation}Since $\tilde{\calR} \subset \calR$ and by (\ref{eqn:calrbound}) $|\tilde{\calR}| \gtrapprox_{\de} \de^{\gamma + \beta_1 + \eta} |\calR|$ and $\calR$ satisfies the $t$-Frostman Convex Wolff Axiom with error $\de^{-\eta}$, we conclude that $\tilde{\calR}$ satisfies the $t$-Frostman Convex Wolff Axiom with error $\lessapprox_{\de}\de^{-\gamma -\beta_1-2\eta}$. So long as $\de_0, \gamma, \beta_1$ and $\eta$ are sufficiently small, $\tilde{\calR}$ satisfies the $t$-Frostman Convex Wolff Axiom withe error $\le \de^{-\tau} \le {\de'}^{-\tau}$. The $t$-Frostman Convex Wolff Axiom is inherited upwards, so the same holds for $\calR'$. Together with (\ref{eqn:largerscalebound}), $(\calR', Y')$ satisfies Conclusion B.

    We now prove our claim. Since $N \ge \de^{\gamma + 2\beta_2/\beta_1} \frac{u}{\De}$ and each $|\T_j| \ge\de^{\gamma} \frac{\De}{\de'}$ for each $j = 1, \dots, N$, we know $|\T_{R'}| \ge \de^{2\gamma + 2\beta_2/\beta_1}\frac{u}{\de'} \ge \de^{22\gamma/\omega + 2\beta_2/\beta_1}\frac{\De'}{\de'}$. Therefore, if $\T_{R'}$ satisfies the $1$-Katz-Tao Convex Wolff Axiom in $R'$ with error $\de^{-\tau/20}$ and $\gamma, \omega, \beta_2,\beta_1$ are sufficiently small, we conclude $\T_{R'}$ satisfies the $1$-Frostman Convex Wolff Axiom in $R'$ with error $\de^{-\tau/10}$. So $\T_{R'}$ does not satisfy the $1$-Katz-Tao Convex Wolff Axiom in $R'$. Applying Lemma \ref{lem:WZ2546}, we can find some $\T' \subset \T_{R'}$ with $|\T'| \gtrapprox_{\de} |\T_{R'}|$ and a collection $\calU$ of $\de' \times \alpha \times 1$ prisms covering $\T'$ so that $|\calU| \lessapprox_{\de} \de^{\tau/25}\frac{\De'}{\alpha}$. Moreover, each $U \in \calU$ makes angle $\ge \de^{\beta_2/\beta_1}u$ with $\ell$ and hence $\vol_{\R^1}(U \cap \ell) \lessapprox_{\de} \de^{-\beta_2/\beta_1}\frac{\alpha}{u}$. Combining our bounds for $\vol_{\R^1}(U \cap \ell)$ and $|\calU|$ and recalling that $\De' = \de^{-20\gamma/\omega}u$, we see that \[\text{Vol}_{\R^1}\left(\bigcup_{U \in \calU} U \cap \ell\right) \lessapprox_{\de} \de^{\tau/25-\beta_2/\beta_1} \frac{\De'}{u} \lessapprox_{\de} \de^{\tau/25-\beta_2/\beta_1-20\gamma/\omega}.\]On the other hand, $\vol_{\R_1}\left(\tilde{Y}(\T') \cap \ell\right) \ge \de^{\gamma}$. Taking $\de_0, \beta_2/\beta_1$ and $\gamma$ smal, we conclude that \[\text{Vol}_{\R^1}\left(\bigcup_{U \in \calU} U \cap \ell\right) \le \vol_{\R_1}\left(\tilde{Y}(\T') \cap \ell\right),\]a contradiction since $\bigcup_{U \in \calU} U \cap \ell \supset \tilde{Y}(\T') \cap \ell$. This proves our claim.

    Finally, we consider the $\de' > \De$ subcase. We apply several arguments already done in this proof, so we omit the tedious details in those arguments. Using Lemma \ref{lem:WZ2547}, we can refine $(\calR_2, Y_4)$ to some $\tilde{\calR}, \tilde{Y}$ so that $\tilde{Y}(R)$ is $\gtrsim \de^{\gamma}$ dense, has multiplicity $\tilde{\mu} \sim \mu_4$, and $|\tilde{\calR}| \gtrsim \de^{\gamma}|\calR_2|$. We take $\calR'$ to be the set of $u\theta \times u \times 1$ prisms concentric with elements of $\tilde{\calR}$. For each $R' \in \calR'$, let $R \in \tilde{\calR}$ be the central prism of $R'$. Take a unit line segment $\ell$ intersecting $\gtrsim \de^{\gamma-1}$ many $\de$-balls from $\tilde{Y}$. Pick $\gtrsim \de^{\gamma + 2\beta_2/\beta_1} \frac{u}{\de'}$ many $\de^{-2\beta_2/\beta_1}\frac{\de'}{u}$-separated points on $\ell$. The $\de'$-neighborhood of each prism is a $\de'$-tube contained in $R'$ and makes angle $\ge \de^{\beta_2/\beta_1}u$ with $\ell$. We apply the same argument as in the $\de' < \De$ case to prove this set of tubes satisfies the $1$-Frostman Convex Wolff Axiom with error $\le \de^{-\tau/10}$. We apply Lemma \ref{lem:RFPC} to prove we have the desired density. We arrive again in Conclusion B.

\end{proof}

\section{Proving Proposition \ref{prop:smallprism}: Upgrading the dimensionality of tubes in prisms}\label{sec:smallprism}

The proof of Proposition \ref{prop:smallprism} requires one more intermediate result. We state and prove it here.

\begin{lemma}\label{lem:grain}
    For any $\zeta > 0$, there exists $\eta', \de_0 > 0$ such that the following holds for all $\de \in (0, \de_0)$, any $\de < \De \le u$. 
    
    Suppose we have a collection $\calR$ of $\de \times \De \times 1$ prisms and a collection of tubes $\T$ with $|\T(R)| \le \de^{-\eta'}\left(\frac{\de}{\De}\right)^{-t}$ for each $R \in \calR$ and $\T(R)$ satisfying the $t$-Frostman Convex Wolff Axiom with error $\le \de^{-\eta'}$. Suppose we have a shading $Y$ with $\left|\bigcup_{T \in \T(R)} Y(T)\right|_{\de} \ge \de^{-1-t+\eta'}\De^t$ for each $R \in \calR$. Then there exists $\eta_1 \sim \eta'$ such that the following holds.

    For each $R \in \calR$, there exists $\T'(R) \subset \T(R)$ satisfying $|\T'(R)| \ge \de^{\eta_1}|\T(R)|$. Moreover, for each $R \in \calR$, $Y'(\T'(R))$ is contained in a union $\calP_{R}$ of disjoint $\de \times \De \times \frac{\De}{u}$ prisms with $|\calP_R| \ge \de^{\eta_1} \frac{u}{\De}$ and for each $P \in \calP_R$, $\T'_{P,R} = \{T \cap P: T \in \T'(R)\}$ satisfies the $t$-Frostman Convex Wolff Axiom with error $\de^{-\zeta}$ and $Y'$ a shading on $\T'$ with $|Y'(T_P)|_{\de} \ge \de^{\eta_1}\frac{\De}{\de u}$ for each $T_P \in \T'_{P,R}$.
\end{lemma}

First, we ensure the shading is uniformly spread through the prisms $\calP_R$. To accomplish this, we apply Lemma \ref{lem:WZ2547} to the $3$-partite hypergraph induced by the incidences between $\de$-balls in $Y(T)$, $\de$-tubes in $\T(R)$ and $\de \times \frac{\De}{u} \times 1$ prisms in $\calP_R$. After completing this step, we include a brief labelled recap describing the results. 

We then proceed to refine $\T'_{P,R}$ to satisfy the $t$-Frostman Convex Wolff Axiom with small error. To prove this, we prove that for any set of tubes $\T(R)$ for which $\T_{P,R} = \{T \cap P: T \in \T(R)\}$ fails the $t$-Frostman Convex Wolff Axiom for many choices of $P \in \calP_R$ must overconcentrate in a set of $\de \times 2^{-i} \times 1$ prisms $\calW_P$ for choices $P$ in a large subset $\calP_i \subset \calP_{R}$. Going forward, we write $2^{-i} = \alpha$. We prove that since $\T_{P,R}$ overconcentrates in $\calW_P$, we have \begin{equation}\label{eqn:weasel1}\left|\bigcup_{P \in \calP_i} \calW_P\right|_{\alpha} \ll \left(\frac{\De}{\alpha}\right)^t\alpha^{-1}.\end{equation}We define $\tilde{\T} = \calD_{\alpha}(\T)$ and define a shading $\tilde{Y}(\tilde{T}) = \tilde{T} \cap \left(\bigcup_{P \in \calP_i} P\right)$. By construction, $\bigcup_{\tilde{T} \in \tilde{T}}\tilde{Y}(\tilde{T}) \subset \bigcup_{P \in \calP_i} \calW_P$, but Lemma \ref{lem:RFPA} proves that \[\left|\bigcup_{\tilde{T} \in \tilde{T}} Y(\tilde{T})\right|_{\alpha} \approx \left(\frac{\De}{\alpha}\right)^t\alpha^{-1},\]contradicting (\ref{eqn:weasel1}). This completes our outline.

\begin{proof}
    It suffices to prove Lemma \ref{lem:grain} one prism at a time. Take $R \in \calR$. We being with a constant multiplicity refinement of $(\T(R), Y)$ to some $(\T(R), Y_1)$ with constant multiplicity $\mu_1$, so that $|Y_1(\T(R))|_{\de} \gtrapprox_{\de} |Y(\T(R))|_{\de}$ and for each $\de$-cube $Q \in Y(\T_1(R))$, $|\{T \in \T(R) : Q \in Y_1(T)\}| = \mu_1$. 
    
    Cover $R$ with a collection of $\de \times \De \times \frac{\De}{u}$ prisms, which we denote $\calP_R$. Define the following $3$-partite hypergraph $G_R$ on $Y_1(\T(R)) \times \T(R) \times \calP_R$. We say that $(Q, T, P)$ is an edge in the graph if $Q \in Y_1(T) \cap P$. This graph has $\sim 1$ hyperedge for each pair $(Q, T)$ of a $\de$-cube $Q \in Y_1(T)$, since each $\de$-cube can lie in $\sim 1$ different prisms. By Lemma \ref{lem:RFPA}, $|Y_1(\T(R))|_{\de} \gtrapprox_{\de} |Y(\T(R))|_{\de} \gtrsim \de^{4\eta'}\de^{-1}\left(\frac{\De}{\de}\right)^t$. Each cube $Q \in Y_1(T)$ incides with at least $1$ tube, so $|G_R| \gtrapprox_{\de} \de^{4\eta'}\de^{-1}\left(\frac{\De}{\de}\right)^t$ \footnote{Here, we bound $|G_R| \ge |Y_1(\T(R))|_{\de}$. We can and later will sharpen this to $|G_R| \sim \mu_1 |Y_1(\T(R))|_{\de}$, but at this point in the proof, doing so would only serve to complicate (\ref{eqn:density}).}. We know that $|\T(R)| \le \de^{-\eta'}\left(\frac{\de}{\De}\right)^{-t}$, $|Y_1(\T(R))|_{\de} \lesssim \de^{-\eta'}\de^{-1}\left(\frac{\De}{\de}\right)^t$, and $|\calP_R| = \frac{u}{\De}$. We conclude that $G_R$ has density \begin{equation}\label{eqn:density}d \sim \frac{|E|}{|Y_1(\T(R))|_{\de}}\gtrsim \frac{\de^{4\eta'}\de^{-1}\left(\frac{\De}{\de}\right)^t}{\de^{-\eta'}\frac{u}{\de \De}\left(\frac{\De}{\de}\right)^{2t}} = \de^{5\eta'}\frac{\De}{u}\left(\frac{\De}{\de}\right)^t.\end{equation}
    By Lemma \ref{lem:WZ2547}, we can find a subgraph $G^*_R \subset G_R$ with $|G^*_R| \sim |G_R|$ which is uniformly $\sim d$ dense. This induces a new shading $Y^*$ on $\T(R)$ with \[Y^*(T) = \{ Q : \text{there exists } P \in \calP_R  \text{ such that }(Q, T, P) \in G^*\},\]a new set of tubes $\T^*(R)= \{T \in \T(R) : Y^*(T) \neq \emptyset\}$, and a new set of prisms \[\calP^*_R = \{P \in \calP_R : \text{ there exists } Q \in Y^*(\T^*(R)), T \in \T^*(R) \text{ such that  } (Q, T, P) \in G^*\}.\]
    Each $\de$-cube in $Y^*(\T^*(R))$ is incident to $\sim 1$ prisms $P \in \calP^*_R$, so we have that \begin{equation}\label{eqn:raccoon}|G^*_R| = \sum_{Q \in Y^*(\T^*(R))} |\{T \in \T^*(R) : Q \in Y^*(T)\}| \le |Y^*(\T^*(R))|_{\de} \mu_1 \end{equation}On the other hand, $|G_R| \sim \mu_1|Y_1(\T(R))|_{\de}$. Since $|G^*_R| \gtrsim |G_R|$, we conclude that \begin{equation}\label{eqn:raccoon4}|Y^*(\T^*(R))|_{\de} \gtrsim |Y_1(\T(R))|_{\de} \gtrapprox_{\de} \de^{4\eta'}\de^{-1}\left(\frac{\De}{\de}\right)^t.\end{equation} Since $G^*_R$ is uniformly $\sim d$ dense, recalling (\ref{eqn:density}), for each tube $T \in \T^*(R)$ and each $P \in \calP^*$, \begin{equation}\label{eqn:raccoon6}|Y^*(T) \cap P|_{\de} \gtrsim d |Y^*(\T^*(R))|_{\de} \gtrapprox_{\de} \de^{9\eta'}\de^{-1}\frac{\De}{u}.\end{equation}Again using the fact that each $\de$-cube in $Y^*(\T^*(R))$ incides with $\sim 1$ many $P \in \calP^*_R$, we see that \begin{equation}\label{eqn:raccoon2}\sum_{P \in \calP^*_R} |\{ (Q, T, P) \in G^*_R\}| \sim |G^*_R|.\end{equation}For each fixed $P$, \begin{equation}\label{eqn:raccoon3}|\{ (Q, T, P) \in G^*_R\}| \le \mu_1|Y_1(\T(R)) \cap P|_{\de} \le \mu_1 \de^{-\eta'}\frac{\De}{u\de}\left(\frac{\De}{\de}\right)^t.\end{equation}Meanwhile, $|G^*_R| \sim \mu_1|Y_1(\T(R))|_{\de} = \mu_1 \de^{2\eta'}\de^{-1}\left(\frac{\De}{\de}\right)^t$, so recalling (\ref{eqn:raccoon2}) and (\ref{eqn:raccoon3}), we conclude that \begin{equation}\label{eqn:raccoon7}|\calP^*_R| \gtrsim \de^{3\eta'}\frac{u}{\De}.\end{equation}

    On the other hand, $|Y^*(T)|_{\de} \lesssim \de^{-1}$ for each $T \in \T^*(R)$, so since $|Y^*(\T^*(R))|_{\de} \gtrapprox_{\de} \de^{4\eta'}\de^{-1}\left(\frac{\De}{\de}\right)^t$, we conclude that $|\T^*(R)| \gtrapprox_{\de} \de^{4\eta'}\left(\frac{\De}{\de}\right)^t$. We assumed in the statement of this lemma that $\T(R)$ satisfies the $t$-Frostman Convex Wolff Axiom with error $\de^{-\eta'}$ in $R$ and has $|\T(R)| \le \de^{-\eta'}\left(\frac{\De}{\de}\right)^t$. Then $\frac{|\T(R)|}{|\T^*(R)|} \lessapprox_{\de} \de^{-5\eta'}$, so $\T^*(R)$ satisfies the $t$-Frostman Convex Wolff Axiom with error $\lessapprox_{\de}\de^{-6\eta'}$.

    \subsection*{Recap of the proof so far:} At this point in the proof of Lemma \ref{lem:grain}, we have established several useful uniformity properties for the tubes in $\T^*(R)$ with the shading in $Y^*$ between different prisms in $\calP^*_R$. 
    
    \begin{enumerate}[label = (\alph*)] 
        \item The shading $Y^*(R)$ is supported on a union of $\de \times \De \times \frac{\De}{u}$-prisms $\calP^*_R$ with \begin{equation*}
            |\calP^*_R| \gtrsim \de^{3\eta'}\frac{u}{\De}.
        \end{equation*}
        \item For each $P \in \calP^*_R$, \begin{equation*}|Y^*(T) \cap P| \gtrapprox_{\de} \de^{9\eta'}\de^{-1}\frac{\De}{\de}\tag{\ref{eqn:raccoon6}}.\end{equation*}
        \item The refined set of tubes $\T^*(R)$ satisfies the $t$-Frostman Convex Wolff Axiom with error $\lessapprox_{\de} \de^{-6\eta'}$.
    \end{enumerate}
    We now continue with the proof.
    
    For each $P \in \calP^*_R$, we have a set of $\de \times \de \times \frac{\De}{u}$ tubes $\T^*_P = \{T \cap P: T \in \T^*_R\}$. Set $\calP_{\text{good}} := \{P \in \calP_R^* : C_{t-\text{KT-CW}}(\T^*_P) \ge \de^{-\zeta/2}\}$. So long as $\eta'$ is small enough relative to $\zeta$, we ensure that if $P \in \calP_{\text{good}}$, then $\T^*_P$ satisfies the $t$-Frostman Convex Wolff Axiom with error $\de^{-\zeta}$. We aim to prove that $|\calP_{\text{good}}| \ge \frac{|\calP_R^*|}{2}$. Suppose otherwise. Set $\calP_{\text{bad}} = \calP_R^* \setminus \calP_{\text{good}}$. For each $P \in \calP_{\text{bad}}$, we apply Lemma \ref{lem:WZ2546} to arrive at some $\T^{**}_P \subset \T^*_P$ with $|\T^{**}_P| \gtrapprox_{\de} |\T^*_P|$ and a $\lessapprox_{\de} 1$ partitioning cover of $\T^{**}_P$ by $\de \times \alpha_P \times \frac{\De}{u}$ prisms with $\alpha_P \in [\de, \De]$. We refer to the family of prisms as $\mathcal{W}_P$. We know for $W \in \mathcal{W}_P$ that $|\T^{**}_P[W]| \gtrapprox_{\de} \de^{\zeta/2}(\alpha_P/\de)^t$. Since $|\T^{**}_{P}| \gtrapprox_{\de} |\T^*_R| \gtrapprox_{\de} \de^{4\eta'} \left(\frac{\De}{\de}\right)^{t}$, we conclude that $|\mathcal{W}_P| \le \de^{\zeta/4}\left(\frac{\De}{\alpha_P}\right)^t$. Since for each $\bigcup_{\T^{**}_P} T \subset \bigcup_{W \in \mathcal{W}_P} W$, it follows that \begin{equation}\label{eqn:seal1}\left|\bigcup_{T_P \in \T^{**}_P} T_P\right|_{\alpha_P} \le |\calW_P| \frac{\De}{\alpha_P u} \le \de^{\zeta/4}\left(\frac{\De}{\alpha_P}\right)^t\frac{\De}{\alpha_P u}.\end{equation}
    
    We dyadically decompose $\calP_{\text{bad}}$ into $\bigsqcup_i \calP_i$, where $\calP_i = \{P \in \calP_{\text{bad}} : \alpha_P \sim 2^{-i}\}$. Since $|\calP_{\text{bad}}| \gtrsim |\calP^*_R|$, for some $i$, $|\calP_i| \gtrapprox_{\de} |\calP^*_R|$. Write $\alpha = 2^{-i}$. For each tube $T \in \T^*_R$, define $N(T) = |\{P \in \calP_i: T \in \T^{**}_P\}|$. By construction, $|\T^{**}_P| \gtrapprox_{\de} |\T^*_P| = |\T^*_R|$, so \begin{equation}\label{eqn:capybara1}\sum_{T \in \T^*_R} N(T) = \sum_{P \in \calP_i} |\T^{**}_P| \gtrapprox_{\de} |\calP_i||\T^*(R)|.\end{equation}We know that $N(T) \le |\calP_i|$ for each $T \in \T^*_R$, so we can find some $\T^{**}_R \subset \T^*_R$ with $|\T^{**}_R| \gtrapprox_{\de} |\T^*_R|$ and $N(T) \gtrapprox_{\de} |\calP_i|$ for each $T \in \T^{**}_R$, since otherwise would contradict (\ref{eqn:capybara1}).

    Let $\tilde{\T} = \calD_{\alpha}(\T^{**}_R)$. Since $\T^{*}_R$ satisfies the $t$-Frostman Convex Wolff Axiom with error $\lessapprox_{\de} \de^{-6\eta'}$ and $|\T^{**}_R| \gtrapprox_{\de} |\T^*_R|$, we conclude that $\T^{**}_R$ satisfies the $t$-Frostman Convex Wolff Axiom with error $\lessapprox_{\de} \de^{-6\eta'}$ as well. For each $\tilde{T} = N_{\alpha}(T)$, define a shading \[\tilde{Y}(\tilde{T}) = \bigcup_{P \in \calP_i} \bigcup_{\substack{T_P \in \T^{**}_P \\ T_P \subset \tilde{T}}} N_{\alpha}(T_P).\]Each $\tilde{T} \in \tilde{\T}$ contains some $T \in \T^{**}_R$ with $N(T) \gtrapprox_{\de} |\calP_i|$. If $T \cap P \in \T^{**}_P$, then $\tilde{Y}(\tilde{T}) \cap P \supset Y^*(T) \cap P$, so \begin{equation}\label{eqn:capybara2}|\tilde{Y}(\tilde{T})|_{\alpha} \sim \alpha^{-1}\frac{\De}{u}.\end{equation}Using (\ref{eqn:raccoon7}), we see that there are $N(T) \gtrapprox_{\de} |\calP_i| \gtrsim \de^{3\eta'}\frac{u}{\De}$ many choices of $P$ with $T \cap P \in \T^{**}_P$, so together with (\ref{eqn:capybara2}), we conclude that $|\tilde{Y}(\tilde{T})|_{\alpha} \gtrapprox_{\de} \de^{3\eta'}\alpha^{-1}$. 
    
    Our construction of $\tilde{Y}(\tilde{T})$ implies that $\bigcup_{T \in \tilde{\T}} \tilde{Y}(\tilde{T}) \subset N_{\alpha}\left(\bigcup_{P \in \calP_i}\bigcup_{T \in \T^*_P} Y^*(T)\right)$. Summing (\ref{eqn:seal1}) over the $\le \frac{u}{\De}$ elements of $\calP_i$, we conclude that \begin{equation}\label{eqn:pika1}
        \left|\bigcup_{T \in \tilde{\T}} \tilde{Y}(\tilde{T})\right|_{\alpha} \le \left|\bigcup_{P \in \calP_i}\bigcup_{T \in \T^{**}_P} Y^*(T)\right|_{\alpha} \le \de^{\zeta/4}\left(\frac{\De}{\alpha}\right)^t\alpha^{-1}
    \end{equation}
    
    On the other hand, we have that $\tilde{\T}$ satisfies the $t$-Frostman Convex Wolff Axiom with error $\lessapprox_{\de}\de^{-6\eta'}$ and has a $\de^{3\eta'}$ dense shading $\tilde{Y}$, we conclude using Lemma \ref{lem:RFPA} that \begin{equation}\label{eqn:pika2}\left|\bigcup_{\tilde{T} \in \tilde{\T}} \tilde{Y}(\tilde{T})\right|_{\alpha} \gtrapprox_{\de} \de^{15\eta'}\left(\frac{\De}{\alpha}\right)^t \frac{1}{\alpha}.\end{equation}As long as $\eta'$ is sufficiently small relative to $\zeta$ and $\de_0$ is sufficiently small, comparing (\ref{eqn:pika1}) and (\ref{eqn:pika2}), we reach a contradiction and finally conclude that $|\calP_{\text{good}}| \ge \frac{1}{2}|\calP|$. This was all carried out in the context of a fixed prism $R$, so we relabel $\calP_{\text{good}}$ to $\calP_R$ to reflect the prism where this applies. 

    We define a shading $Y'(T)$ by $Y'(T) \cap P = Y^*(T) \cap P$ if $P \in \calP_{\text{good}}$ and $Y'(T) \cap P= \emptyset$ otherwise. Take $\eta_1 =10\eta'$. We know $\T^*(R)$ itself satisfies the $t$-Frostman Convex Wolff Axiom with error $\de^{-6\eta'} \le \de^{-\eta_1}$ and multiplicity $|\T^*(R)| \le \de^{-\eta_1}|\T(R)|$ in $R$. We also know $|\calP_{\text{good}}| \gtrsim \de^{3\eta'}\frac{u}{\De} \ge \de^{\eta_1}\frac{u}{\De}$. For each prism $P \in \calP_{\text{good}}$, we have the collection $\T'_P$ satisfying the $t$-Frostman Convex Wolff Axiom with error $\de^{-\zeta}$ in $P$. The shading $Y'(T) \cap P$ is $\de^{\eta_1}$ dense within this set as well. We follow this procedure for each $R \in \calR$ and label the resulting set of prisms with $\calP_{R}$ and the tube set for each prism by $\T'_{P, R}$. 
\end{proof}

Our final lemma is a brush argument for unions of prisms with shading supported on collections of tubes. The hypotheses and conclusion of this lemma are carefully adapted to its use Proposition \ref{prop:smallprism}.

\begin{lemma}\label{lem:brush}
    For any $\tau > 0$ there exists $\eta' > 0$ such that for any $\beta_2 > 0$ there exists $\de_0> 0$ such that the following holds for all $\de \in (0, \de_0)$, any $\De \in (\de, 1)$ and any $\theta \ge \frac{\de}{\De}$. . 

    Let $\calP$ be a set of $\de \times \De \times \De$ prisms, $\T$ a set of $\de \times \de \times \De$ tube segments, and $Y$ a shading on $\T$. Suppose that $\T = \bigcup_{P \in \calP} \T(P)$, where $\T(P)$ satisfies the $t$-Frostman Convex Wolff Axiom in $P$ with error $\de^{-\eta'}$ and multiplicity $\le \de^{-\eta'}\left(\frac{\De}{\de}\right)^t$ and $|Y(\T(P))|_{\de} \ge \de^{\eta'}\left(\frac{\De}{\de}\right)^{t+1}$. Suppose furthermore that for each $\de$-cube $Q \in Y(\T)$, $\{n(P) : Q \in Y(\T(P))\} \subset S^2$ is $\beta_2$-broadly spread in a $\theta$-ball. 

    Then for any $P_0 \in \calP$, if $G(P_0)$ is the concentric $3\theta\De \times 3\De \times 3\De$ prism, there exists some $\mathscr{S}_{P_0} \subset \calP$ so that each $P' \in \mathscr{S}_{P_0}$ is contained in $G(P_0)$, $\mathscr{S}_{P_0}$ satisfies the $t$-Frostman Convex Wolff Axiom in $G(P_0)$ with error $\de^{-\tau}$, and $Y(\T(P')) \cap Y(\T(P_0)) \neq \emptyset$ for each $P' \in \mathscr{S}_{P_0}$.
\end{lemma}

We defer the proof of this lemma to Section \ref{subsec:brush}. 

We now prove Proposition \ref{prop:smallprism}. This is the first step we take in the induction on scales argument in Proposition \ref{prop:prismsetup}. Roughly, it says that if a collection of $\de \times \De \times 1$ prisms $\calR$ factors a collection of $\de$-tubes $\T$ from above and below with respect to the $t$-Frostman Convex Wolff Axiom, then $\bigcup_{T \in \T} T$ has upper spectrum about $2t+1$ or there exist scales $\overline{\de}, \overline{\De}$ and a collection $\overline{\calR}$ of $\overline{\de} \times \overline{\De} \times 1$ prisms such that $\overline{\calR}$ factors $\calD_{\overline{\de}}(\T)$ from above with respect to the $t$-Frostman Convex Wolff Axiom but from below with respect to the $1$-Frostman Convex Wolff Axiom. In other words, $\overline{\calR}, \calD_{\overline{\de}}(\T)$ and $\calD_{\overline{\de}}(Y)$ satisfies the hypotheses for Proposition \ref{prop:largeprism}. 

\begin{proposition*}[\textbf{\ref{prop:smallprism}}]
    For any $\omega, \e, \tau > 0$, there exists $\eta, \de_0 > 0$ such that the following holds for all $ \de \in (0, \de_0)$ and $\De \in (\de^{1-\e}, 1]$. Suppose $\calR$ is a collection of $\de \times \De \times 1$ prisms satisfying the $t$-Frostman Convex Wolff Axiom with error $\de^{-\eta}$, $\T$ is a collection of $\de$-tubes such that for each $R \in \calR$, there exists a family $\T(R)$ contained in $R$, satisfying the $t$-Frostman Convex Wolff Axiom with error $\de^{-\eta}$ and with $|\T(R)| \le \de^{-\eta} \left(\frac{\De}{\de}\right)^t$. Let $Y(T)$ be a $\de^{\eta}$-dense shading on $\T$. Then one of the following must hold. 

    \begin{enumerate}[label = \Alph*.]
        \item There exists $r > \de^{1-\eta}$ and an $r$-ball $B_r$ such that \[\left|\bigcup_{T \in \T} Y(T) \cap B_r\right|_{\de} \ge \left(\frac{r}{\de}\right)^{2t + 1 - \omega};\]or
        \item There exists some $\overline{\de}, \overline{\De}$ with $\frac{\overline{\De}}{\overline{\de}} = \left(\frac{\De}{\de}\right)^{\omega/(10t + 5)}$ and a collection $\overline{\calR}$ of $\overline{\de}\times \overline{\De}$ prisms satisfying the $t$-Frostman Convex Wolff Axiom with error $\overline{\de}^{-\tau}$, equipped with a $\overline{\de}^{\tau}$-dense shading $\overline{Y}$ such that $\overline{Y}(\overline{R}) \subseteq N_{\overline{\de}}\left(\bigcup_{R \in \calR} Y(R)\right)$ for each $\overline{R} \in \overline{\calR}$.
    \end{enumerate}
\end{proposition*}

Proposition \ref{prop:smallprism} is the most difficult result in this paper. In Section \ref{subsec:bpop3}, we gave a rough outline of the proof of this result. As in the proof of Proposition \ref{prop:largeprism}, we break the proof into a sequence of setup reductions, which we recap before proceeding with the trichotomy described in Section \ref{subsec:bpop3}.

\begin{proof}~
    \subsection*{Reduction I}\vspace{-7pt}In this reduction, we use Lemma \ref{lem:tubebroad} to find a broad cover of $\calR$ by $u$-tubes.  
    
    Let $\T_{\De} = \calD_{\De}(\calR)$. Equip the tube $T_{\De} = N_{\De}(R)$ with the shading $Y(R)$, which we still refer to as $Y$. Apply Lemma \ref{lem:tubebroad} (A) to $(\T_{\De}, Y)$ with some $\beta_1 > 0$ to be determined in the course of the proof. Let $(\T'_{\De}, Y_1)$ denote the refined collection of tubes and the refined shading. This gives a set of tubes and a shading $(\T'_{\De}, Y')$ which induces a subset $(\calR, Y_1)$ which we denote by $(\calR_1, Y_1)$ satisfying \begin{equation}\label{eqn:Y10}\sum_{R \in \calR_1} |Y_1(R)|_{\de} \gtrapprox_{\de} \sum_{R \in \calR} |Y(R)|_{\de}.\end{equation}We record for future reference that \begin{equation}\label{eqn:R10}|\calR_1|\gtrapprox_{\de} |\calR|,\end{equation}since otherwise (\ref{eqn:Y10}) would be impossible. We also have a parameter $u \ge \De$ and a collection $\T_u$ of $u$-tubes which is a $1$-partitioning cover of $\calR_1$, so $(\calR_1, Y_1)$ is $\beta_1$-broad in $\T_u$ with error $\lessapprox_{\de} 1$ and broadness multiplicity $\lessapprox_{\de} u^{-\beta_1}$.

    We carry out two further refinements. First, we carry out a constant multiplicity refinement from $Y_1$ to $Y'_2$ so that \begin{equation}\label{eqn:dragon3}\sum_{R \in \calR_1} |Y'_2(R)|_{\de} \gtrapprox_{\de} \sum_{R \in \calR_1} |Y_1(R)|_{\de} \gtrapprox_{\de} \sum_{R \in \calR_1} |Y(R)|_{\de}.\end{equation}We call the resulting multiplicity $\mu_2$. For the second refinement, we claim we can find a subshading $Y_2$ of $Y'_2$ such $\sum_{R \in \calR_1} |Y_2(R)|_{\de} \gtrsim \sum_{R \in \calR_1} |Y'_2(R)|_{\de}$ and for each $\de$-cube $Q \in \bigcup_{R \in \calR_1} Y_2(R)$, \begin{equation}\label{eqn:dragon}|\{R \in \calR_1[T_u] : Q \in Y_2(R)\}| \gtrapprox_{\de} \de^{\beta_1}\mu_2.\end{equation}To prove this claim, we note that it suffices to prove that for each $Q \in \bigcup_{R \in \calR_1} Y_2(R)$, we can find some $\T'_u(Q) \subset \T_u$ such that for each $T_u \in \T'_u(Q)$, \begin{equation}\label{eqn:dragon2}|\{R \in \calR_1[T_u] : Q \in Y'_2(R)\}| \gtrapprox_{\de} \de^{\beta_1}\mu_2\end{equation}and $\sum_{T_u \in \T'_u(Q)} |\{R \in \calR_1[T_u] : Q \in Y'_2(R)\}| \ge \frac{\mu_2}{2}$. Since the broadness multiplicity is $\lessapprox_{\de} \de^{-\beta_1}$, choosing implicit constants appropriately, we can ensure that the set of tubes $\T''_u(Q)$ failing (\ref{eqn:dragon2}) satisfies $\sum_{T_u \in \T''_u(Q)} |\{R \in \calR_1[T_u] : Q \in Y'_2(R)\}| \le \frac{\mu_2}{2}$. We take $\T'_u(Q) = \T_u \setminus \T''_u(Q)$ and then set $Q \in Y_2(R)$ if $R \subset \T'_u(Q)$. This shading satisfies the claimed properties: $\sum_{R \in \calR_1} |Y_2(R)|_{\de} \gtrsim \sum_{R \in \calR_1} |Y'_2(R)|_{\de}$, so by (\ref{eqn:dragon3}), \begin{equation}\label{eqn:dragon4} \sum_{R \in \calR_1} |Y_2(R)|_{\de} \gtrapprox_{\de} \sum_{R \in \calR_1} |Y_1(R)|_{\de};\end{equation}and for each $Q \in \bigcup_{R \in \calR_1} Y_2(R)$, \begin{equation}\label{eqn:dragon5}|\{R \in \calR_1[T_u] : Q \in Y_2(R)\}| \gtrapprox_{\de} \de^{\beta_1}\mu_2.\end{equation}

    \subsection*{Reduction II}In this reduction, we refine the prisms and shading so that the number of $\de$-balls assigned to each prism is about the same and the number of prisms incident to each $\de$-ball is the same. 
    
    Specifically, we refine $\calR_1$ to some $\calR_3$ and $Y_2$ to some $Y_3$ so that each prism $R_3 \in \calR_3$ has about the same mass from $Y_3$ and each point in $Y_3$ still has multiplicity $\sim \mu_2$. Define $E = \{(Q, R) \in Y_2(\calR_1) \times \calR_1: Q \in Y_2(R)\}$. Note that $|E| = \sum_{R \in \calR_1} |Y_2(R)|_{\de}$. Apply Lemma \ref{lem:WZ2547} and Remark \ref{rmk} to the graph $(Y_2(\calR_1)  \sqcup  \calR_1, E)$. This induces a collection $\calR_3$ of prisms and a shading $Y_3(R)$ with $Q \in Y_3(R)$ if and only if $(Q, R) \in E'$. Since $|E'| \gtrsim |E|$ and $R \in \calR_3$ if and only if $Y_3(R) \neq \emptyset$, we know that \begin{equation}\label{eqn:bunny2}\sum_{R \in \calR_3} |Y_3(R)|_{\de} \gtrsim \sum_{R \in \calR_1} |Y_2(R)|_{\de}.\end{equation}For $Q \in Y_3(\calR)$, we know that $\mu_3(Q) \gtrsim \frac{|E|}{|Y_2(\calR_1)|} \sim \mu_2$. We also have that \begin{equation}\label{eqn:temp2}|Y_3(R)|_{\de} \gtrsim \frac{|E|}{|\calR_1|} = \frac{\sum_{R \in \calR_1} |Y_2(R)|_{\de}}{|\calR_1|}.\end{equation}By Lemma \ref{lem:RFPA}, for each $R$ \begin{equation}\label{eqn:L2}|Y(R)|_{\de} \gtrapprox_{\de} \de^{-4\eta}\de^{-1}\left(\frac{\De}{\de}\right)^t,\end{equation}which together with (\ref{eqn:dragon4}) and (\ref{eqn:Y10}) implies that \[\sum_{R \in \calR_1} |Y_2(R)|_{\de} \gtrapprox_{\de} \de^{2\eta}\De^t \de^{-t-1}|\calR_1|.\]Plugging this bound into (\ref{eqn:temp2}), we see that \begin{equation}\label{eqn:bunny3}|Y_3(R)|_{\de} \gtrapprox_{\de} \de^{2\eta}\De^t\de^{-t-1}.\end{equation}We know by a union bound that $|Y_3(R)|_{\de} \lesssim |\T(R)|\de^{-1} \lesssim \de^{-1-t-\eta}\De^t$, so \[|\calR_3|\de^{-1-t-\eta}\De^t \gtrsim \sum_{R \in \calR_3} |Y_3(R)|_{\de} = |E'| \gtrsim |E| = \sum_{R \in \calR_1} |Y_2(R)|_{\de} \gtrapprox_{\de} \de^{2\eta}\de^{-1-t}\De^t|\calR_1|.\]We conclude that \begin{equation}\label{eqn:R31}|\calR_3| \gtrapprox_{\de} \de^{2\eta}|\calR_1|.\end{equation}Finally, at each point the multiplicity of the shading is $\mu_3 \sim \mu_2$ ($\mu_3(Q)$ can vary by a $\sim 1$ factor over different cubes, so it is harmless to pretend that it is a constant).

    \subsection*{Reduction III}In this reduction, we use Lemma \ref{lem:grain} to refine $\T(R)$ to satisfy the $t$-Frostman Convex Wolff Axiom on many $\de \times \De \times \frac{\De}{u}$ subprisms of $\T(R)$, while maintaining most of the mass of the shading. 
    
    Let $\zeta = \zeta(\omega)$ be a parameter to be determined over the course of the proof. Apply Lemma \ref{lem:grain} with this value of $\zeta$ and $\eta' = 3\eta$. We assume the value of $\de_0$ in this proof of less than the value of $\de_0$ for Lemma \ref{lem:grain} and $\eta$ is sufficiently small depending on $\zeta$ to apply Lemma \ref{lem:grain}. Our choice of prisms is $\calR_3$. The choice of tubes is $\T(R)$, which certainly satisfy the $t$-Frostman Convex Wolff Axiom with error $\le \de^{\eta'}$ and has $|\T(R)| \le \de^{-\eta'}\left(\frac{\de}{\De}\right)^{-t}$. We take the shading to be $Y_3$. By (\ref{eqn:bunny3}), $\left|\bigcup_{T \in \T(R)} Y_3(T)\right|_{\de} \gtrapprox_{\de} \de^{-1-t+2\eta}\De^t$. So long as $\de_0$ is sufficiently small, $\left|\bigcup_{T \in \T(R)} Y_3(T)\right|_{\de} \ge \de^{-1-t+\eta'}\De^t$. 
    
    We have checked all the hypotheses for Lemma \ref{lem:grain}. We apply the lemma and conclude that there is a subshading $Y_4$ of $Y_3$, for each prism $R \in \calR_3$ a subset $\T_1(R)$ of $\T(R)$ with $|\T_1(R)| \ge \de^{\eta_1}|\T(R)|$, and a collection of $\de \times \De \times \frac{\De}{u}$-prisms $\calP_R$ with \begin{equation}\label{eqn:calp}|\calP_R| \ge \de^{\eta_1}\frac{u}{\De}.\end{equation}For each $P \in \calP_R$, there is a set of tubes $\T_{P,R} = \{T \cap P : T \in \T_1(R)\}$ satisfying the $t$-Frostman Convex Wolff Axiom in $P$ with error $\de^{-\zeta}$. Finally, $|Y_4(T_P)|_{\de} \ge \de^{\eta_1}\frac{\De}{u\de}$ for each $T_P \in \T_{P,R}$. This implies that for each $T \in \T_1(R)$, \[|Y_4(T)|_{\de} \ge |\calP_R|\de^{\eta_1}\frac{\De}{u\de} \ge \de^{2\eta_1}\de^{-1}.\]Since $\T(R)$ satisfies the $t$-Frostman Convex Wolff Axiom in $R$ with error $\de^{-\eta}$ and $|\T_1(R)| \ge \de^{\eta_1}|\T(R)|$, we see that $\T_1(R)$ satisfies the $t$-Frostman Convex Wolff Axiom in $R$ with error $\de^{-\eta_1-\eta}$. We then apply Lemma \ref{lem:RFPA} to $(\T_1(R), Y_4)$ to see that \begin{equation}\label{eqn:4shadingbound}|Y_4(R)|_{\de} \gtrapprox_{\de} \de^{5\eta_1 + \eta}\De^t\de^{-1-t}.\end{equation}Since $|Y_3(R)| \lesssim \De^t\de^{-1-t-\eta}$, conclude that \begin{equation}\label{eqn:34shadingbound}|Y_4(R)|_{\de} \gtrapprox_{\de} \de^{5\eta_1 + 2\eta}|Y_3(R)|_{\de} \gtrapprox_{\de} \de^{6\eta_1}|Y_3(R)|_{\de}\end{equation}

    \subsection*{Reduction IV}In this reduction, we upgrade the broadness at each point to have multiplicity $1$, at the cost of an acceptable loss in the shading.
    
    We first apply another constant multiplicity refinement between prisms to arrive at a shading $Y_5$ with constant multiplicity $\mu_5$ and \begin{equation}\label{eqn:45shadingbound}\sum_{R \in \calR_3} |Y_5(R)|_{\de} \gtrapprox_{\de} \sum_{R \in \calR_3} |Y_4(R)|_{\de}.\end{equation}Recall that $Y_5$ is a subshading of $Y_2$, $(\calR_1, Y_2)$ was $\beta_1$-broad in $\T_u$ with error $\lessapprox_{\de} 1$, and $Y_2$ has multiplicity $\mu_2$. Applying (\ref{eqn:45shadingbound}), (\ref{eqn:34shadingbound}) and (\ref{eqn:bunny2}), we conclude that \begin{equation}\label{eqn:52shadingbound}\sum_{R \in \calR_3} |Y_5(R)|_{\de} \gtrapprox_{\de} \de^{6\eta_1}\sum_{R \in \calR_1} |Y_2(R)|_{\de}.\end{equation}Since $\bigcup_{R \in \calR_3} Y_5(R) \subset \bigcup_{R \in \calR_1} Y_2(R)$, we conclude that $\mu_5 \gtrapprox_{\de} \de^{6\eta_1}\mu_2$. 
    
    At each $\de$-cube $Q \in \bigcup_{R \in \calR_3} Y_5(R)$ and each $T_u \in \T_u$, by (\ref{eqn:dragon5}) we know that $\{\dir(R) : R \in \calR_1[T_u], Q \in Y_2(R)\}$ is either empty or broad in a $u$-ball with error $\lessapprox_{\de} 1$ and has $\gtrapprox_{\de} \de^{\beta_1}\mu_2$ many elements. Since \[\sum_{T_u \in \T_u} |\{R \in \calR_3[T_u] : B \in Y_5(R)\}| = \mu_5 \gtrapprox_{\de}\de^{6\eta_1} \mu_2 = \de^{6\eta_1}\sum_{T_u \in \T_u} |\{R \in \calR_1[T_u]: B \in Y_2(R)\}|,\]for each $Q \in Y_5(\calR_3)$, we can find a choice $T_u(Q) \in \T_u$ such that \[|\{R \in \calR_3[T_u(Q)] : Q \in Y_5(R)\}| \gtrapprox_{\de} \de^{6\eta_1} |\{R \in \calR_1[T_u(Q)]: Q \in Y_2(R)\}|\gtrapprox_{\de} \de^{6\eta_1 + \beta_1}\mu_2.\]Define a new shading by $Q \in Y_6(R)$ if $R \subset T_u(Q)$ and note that this shading contains each $Q \in Y_5(\calR_3)$. Our new shading also satisfies $\mu_6(Q) \gtrapprox_{\de} \de^{\beta_1 + 5\eta_1}\mu_2 \gtrapprox_{\de} \de^{\beta_1 + 6\eta_1}\mu_5(Q)$\footnote{The last inequality follows from the easy bound $\mu_2(Q) \ge \mu_5(Q)$.}, so \begin{equation}\label{eqn:56shadingbound}\sum_{R \in \calR_3} |Y_6(R)|_{\de} \gtrapprox_{\de} \de^{\beta_1 + 5\eta_1} \sum_{R \in \calR_3} |Y_5(R)|_{\de}.\end{equation}Moreover, $Y_6$ is a subshading of $Y_2$ and for each $\de$-cube $Q$, \[|\{R \in \calR_3[T_u] : Q \in Y_5(R)\}| \gtrapprox_{\de} \de^{5\eta_1 }|\{R \in \calR_1[T_u]: Q \in Y_2(R)\}|.\]By Remark \ref{rmk2}, we know that $(\calR_3, Y_6)$ is $\beta_1$-broad in $\T_u$ with error $\lessapprox_{\de} \de^{-5\eta_1}$ and broadness multiplicity $1$.

    \subsection*{Reduction V} In this reduction, we break the shading in each $T_u$ up into short tubelets incident to the same number of $\de$-cubes, which allows us to control the mass of the shading after rescaling. 
    
    This is similar to the argument in Proposition \ref{prop:largeprism}, Reduction III, except that the shading is now supported on a union of tubes, rather than the entire prism. We only sketch the parts that are identical to Proposition \ref{prop:largeprism}. Recall that for a $u$-tube $T_u$, $\phi_{T_u}$ is the affine linear map sending $T_u$ to the unit ball. We start by covering $\calR_3[T_u]$ by a disjoint union $\mathcal{T}'_{T_u}$ consisting of $\de \times \de \times \frac{\de}{u}$ tube segments, which rescale under $\phi_{T_u}$ to dyadic $\frac{\de}{u}$-cubes. Refine $\mathcal{T}'_{T_u}$ and $Y_6(\calR_4[T_u])$ so that each $Q \in Y'_6(\calR_4[T_u])$ intersects exactly one $\lambda \in \mathcal{T}'_{T_u}$ and \begin{equation}\label{eqn:66prime}\sum_{R \in \calR_4[T_u]} |Y'_6(R)|_{\de} \gtrsim \sum_{R \in \calR_4[T_u]} |Y_6(R)|_{\de}.\end{equation} Let $E = \{(\lambda, R) \in \mathcal{T}'_{T_u} \times \calR_4[T_u]: \lambda \cap Y'_6(R) \neq \emptyset\}$. Pigeonhole on $E$ to some $E'$ to ensure that for each $(\lambda, R) \in E'$, $|\lambda \cap Y'_6(R)|_{\de} \sim M_{T_u}$ while \begin{equation}\label{eqn:76prime}\sum_{R \in \calR_3[T_u]} \left|\bigcup_{(\lambda, R) \in E'} \lambda \cap Y'_6(R)\right|_{\de} \gtrapprox_{\de} \sum_{R \in \calR_3[T_u]} |Y'_6(R)|_{\de}.\end{equation}Set \[Y_7(R) = \{ Q \in Y'_6(R) : Q \cap \lambda \neq \emptyset \text{ for some } \lambda \text{ with } (\lambda, R) \in E'\}.\]Combining (\ref{eqn:66prime}) and (\ref{eqn:76prime}), we see that \begin{equation}\label{eqn:pre67shadingbound}\sum_{R \in \calR_3} |Y_7(R)|_{\de} \gtrapprox_{\de} \sum_{R \in \calR_3} |Y_6(R)|_{\de}.\end{equation}

    Decompose $\T_u$ into $\T_{u,M} = \{T \in \T_u : M_{T_u} \in [M, 2M)\}$ for dyadic choices of $M \in \N$. Decompose $\calR_3$ into $\calR_M = \{R \in \calR_3 : R \subset T_u \text{ for some } T_u \in \T_M\}$. Dyadically pigeonholing, we see that there is a choice of $M$ such that \begin{equation}\label{eqn:mid7shadingbound}\sum_{R \in \calR_M} |Y_7(R)|_{\de} \gtrapprox_{\de} \sum_{R \in \calR_3} |Y_7(R)|_{\de}.\end{equation}We claim that $|\calR_M| \ge \de^{\beta_1 + 15\eta_1}|\calR_3|$. If not, then since $\eta_1 > 3\eta$ and $|Y(R)|_{\de} \le \de^{-1-t}\left(\frac{\De}{\de}\right)^t$, we have \[\sum_{R \in \calR_M} |Y_7(R)|_{\de} \le |\calR|\de^{\beta_1 + 15\eta_1} \sup_{R \in \calR}|Y(R)|_{\de} \le  |\calR|\de^{\beta_1 + 15\eta_1 - \eta} \de^{-1-t}\De^t \le  |\calR|\de^{\beta_1 +15 \eta_1} \de^{-1-t}\De^t.\]On the other hand, combining (\ref{eqn:pre67shadingbound}), (\ref{eqn:56shadingbound}), (\ref{eqn:52shadingbound}), (\ref{eqn:dragon4}), (\ref{eqn:Y10}), and (\ref{eqn:L2}), we see that \[\sum_{R \in \calR_M} |Y_7(R)|_{\de} \gtrapprox_{\de} \de^{\beta_1 + 11\eta_1}\sum_{R \in \calR}|Y(R)|_{\de} \gtrapprox_{\de} \de^{\beta_1 + 11\eta_1 + 5\eta}\de^{-1-t}\De^t |\calR|.\]So long as $\de_0$ is sufficiently small, we conclude that $|\calR_M| \ge \de^{\beta_1 + 15\eta_1}|\calR_3|$. We relabel $\T'_u = \T_{u,M}$ and $\calR_4 = \calR_M$. We see that \begin{equation}\label{eqn:43card}
        |\calR_4| \ge \de^{\beta_1 + 15\eta_1}|\calR_3|.
    \end{equation}Combining (\ref{eqn:pre67shadingbound}) and (\ref{eqn:mid7shadingbound}), we also conclude that \begin{equation}\label{eqn:67shadingbound}
        \sum_{R \in \calR_4} |Y_7(R)|_{\de} \gtrapprox_{\de} \sum_{R \in \calR_3} |Y_6(R)|_{\de}.
    \end{equation}
    
    The last step of this reduction is to bound $M$ from below. Fix a choice of $T_u \in \T'_u$ and note that $\phi_{T_u}$ maps $\T(R)$ to a collection of essentially $\frac{\de}{u}$-tubes. Each $\frac{\de}{u}$-tube intersects $\lesssim \frac{u}{\de}$ many dyadic $\frac{\de}{u}$-cubes. After undoing the rescaling, we see that each $T \in \T(R)$ intersects $\lesssim \frac{u}{\de}$ many tube segments. Since $|\T(R)| \le \de^{-\eta_1}\left(\frac{\De}{\de}\right)^t$, $\T(R)$ intersects $\lesssim |\T(R)| \frac{u}{\de} \lesssim \de^{-\eta}\left(\frac{\De}{\de}\right)^t\frac{u}{\de}$ many tube segments from $\mathcal{T}_{T_u}$. We conclude that \begin{equation}\label{eqn:bunny4}\sum_{R \in \calR_4} |Y_7(R)|_{\de} \lesssim |\calR_4|\de^{-\eta_1}\left(\frac{\De}{\de}\right)^t\frac{u}{\de}M.\end{equation}
    
    We know by (\ref{eqn:67shadingbound}), (\ref{eqn:56shadingbound}), (\ref{eqn:45shadingbound}) we have that $\sum_{R \in \calR_4} |Y_7(R)|_{\de} \gtrapprox_{\de} \de^{\beta_1 + 6\eta_1}\sum_{R \in \calR_3} |Y_4(R)|_{\de}$, so by (\ref{eqn:4shadingbound}) it follows that \begin{equation}\label{eqn:7shadingbound}\sum_{R \in \calR_4} |Y_7(R)|_{\de} \gtrapprox_{\de} \de^{\beta_1 + 6\eta_1}|\calR_3|\de^{-1-t}\De^t \gtrapprox_{\de} \de^{\beta_1 + 6\eta_1}|\calR_4|\de^{-1-t}\De^t.\end{equation}Comparing with (\ref{eqn:bunny4}), we conclude that \begin{equation}\label{eqn:Mlowerbound}M \gtrapprox_{\de} \de^{\beta_1 + 7\eta_1}u^{-1}.\end{equation}On the other hand, we certainly have \begin{equation}\label{eqn:Mupperbound}M \lesssim u^{-1}.\end{equation}

    For $R^{T_u} \in \calR_4^{T_u},$ we define $Y^{T_u}(R^{T_u}) = \{ \phi_{T_u}(\lambda) : \lambda \in \mathcal{T}_{T_u}, \lambda \cap Y_7(R) \neq \emptyset\}$, where $T_u$ is the unique $u$-tube from $\T'_u$ containing $R$. This is a well-defined shading, since $\mathcal{T}_{T_u}$ is the preimage under $\phi_{T_u}$ of a collection of dyadic $\frac{\de}{u}$-cubes. By construction, $Y^{T_u}(R^{T_u}) \supset \phi_{T_u}(Y_7(R))$, each $\de$-cube in $Q \in Y_7(R)$ intersects exactly one tube segment from $\mathcal{T}_{T_u}$, and each $\lambda \in \phi^{-1}_{T_u}(Y^{T_u}(R^{T_u}))$ intersects $\sim M$ many $\de$-cubes from $Y_7(R)$.
    
    \subsection*{Reduction VI}In this reduction, we find a transversality parameter $\theta$ so that at each $\de$-cube $Q \in Y^{T_u}(\calR^{T_u})$, the normal directions of prisms incident to $Q$ are broadly spread in a $\theta$-ball.
    
    This is similar to Reduction III from Proposition \ref{prop:largeprism}, so we again only sketch the argument. For each $T_u$ and each $Q \in Y^{T_u}(\calR^{T_u})$, set $\tilde{V}_Q = \{n(R^{T_u}) : Q \in Y^{T_u}(R^{T_u})\}$. Apply Corollary \ref{cor:localbroad} to $\tilde{V}_Q = \{n(R^{T_u}) : Q \in Y^{T_u}(R^{T_u})\}$ with $\beta = \beta_2$, where $\beta_2$ is a parameter depending on $\omega$ and $\beta_1$ to be determined over the course of the proof. Denote the value returned by Corollary \ref{cor:localbroad} by $\theta_Q$ and the resulting multiset of directions by $V_Q$. We define a new shading $\tilde{Y}_1^{T_u}$ by setting $Q\in \tilde{Y}_1^{T_u}(R^{T_u})$ if $n(R^{T_u}) \in V_Q$. Since $|V_Q| \gtrapprox_{\de} \de^{\beta_2} |\tilde{V}_Q|$, $\mu^{T_u}_1(Q) \gtrapprox_{\de} \de^{\beta_2} \mu^{T_u}(Q)$ for every $Q \in Y^{T_u}(\calR_4^{T_u})$, while $Y^{T_u}(\calR_4^{T_u}) = Y_1^{T_u}(\calR_4^{T_u})$, so \[\sum_{R \in \calR_4^{T_u}} |\tilde{Y}^{T_u}_1(R^{T_u})|_{\de/u} \gtrapprox_{\de} \de^{\beta_2} \sum_{R \in \calR_4^{T_u}} |Y^{T_u}(R^{T_u})|_{\de/u}.\]After dyadically pigeonholing on the value of $\theta_Q$ and on the multiplicity over all $T_u \in \T_u$, we arrive at a single value $\theta$, a single multiplicity $\mu_1^{T_u}$, and a shading $Y_1^{T_u} \subset \tilde{Y}_1^{T_u}$ so that \begin{equation}\label{eqn:rescaled12}\sum_{T_u \in \T'_u}\sum_{R^{T_u} \in \calR_4^{T_u}} |Y_1^{T_u}(R^{T_u})|_{\de/u} \gtrapprox_{\de} \sum_{T_u \in \T'_u}\sum_{R^{T_u} \in \calR_4^{T_u}} |\tilde{Y}_1^{T_u}(R^{T_u})|_{\de/u}\gtrapprox_{\de} \de^{\beta_2}\sum_{T_u \in \T'_u}\sum_{R^{T_u} \in \calR_4^{T_u}} |Y^{T_u}(R^{T_u})|_{\de/u}.\end{equation}Also, for each $\frac{\de}{u}$-cube $Q$, $\{n(R^{T_u}) : Q \in Y_1^{T_u}(R^{T_u})\}$ is $\beta_2$-broadly spread in a $\theta$-ball and $Y_1^{T_u}$ has constant multiplicity $\mu^{T_u}_1$. 

    The refinement of the rescaled shading induces a refinement $Y_8$ of $Y_7$. To construct this, take the unique $T_u \in \T_u$ such that $R \subset T_u$ and we define \[Y_8(R) = \{Q \in Y_7(R) : Q \cap \lambda \neq \emptyset \text{ for some }\lambda \in \phi_{T_u}^{-1}(Y_1^{T_u}(R^{T_u}))\}. \]Each $\lambda \in \phi_{T_u}^{-1}(Y_1^{T_u}(R^{T_u}))$ intersects $\sim M$ $\de$-cubes from $Y_8(R)$ and each $\de$-cube $Y_8(R)$ intersects at most one $\lambda \in \phi_{T_u}^{-1}(Y_1^{T_u}(R^{T_u}))$. The analogous statement holds for $Y^{T_u}$ and $Y_7$, so (\ref{eqn:rescaled12}) implies that \begin{align*}\sum_{R \in \calR_4} |Y_8(R)|_{\de} &\sim M \sum_{T_u \in \T'_u}\sum_{R^{T_u} \in \calR_4^{T_u}} |Y_1^{T_u}(R^{T_u})|_{\de/u} \nonumber\\ &\gtrapprox_{\de} M\de^{\beta_2}\sum_{T_u \in \T'_u}\sum_{R^{T_u} \in \calR_4^{T_u}} |Y^{T_u}(R^{T_u})|_{\de/u} \sim  \de^{\beta_2}\sum_{R \in \calR_4} |Y_7(R)|_{\de}\numberthis \label{eqn:78density}.\end{align*}
    
    \subsection*{Recap of the proof so far.} 

    We start by reviewing the myriad objects defined in the proof so far; subsequently we review their useful properties. We began with a collection of $\de \times \De \times 1$ prisms $\calR$, which we successively refined it to families $\calR_i$ for $i$ from $1$ to $4$. Each prism $R \in \calR$ contains a collection of tubes $\T(R)$, which we refine to a collection of tubes $\T_1(R)$. We began with a shading $Y$ on $\calR$ supported on $\T(R)$, and successively refined to shadings $Y_i$ for $i$ from $1$ to $8$ with $Y_1, \dots, Y_3$ supported on $\T(R)$ and $Y_4, \dots, Y_8$ supported on $\T_1(R)$. We defined a set of $u$-tubes $\T_u$ and refined it to a set $\T'_u$. For each $T_u$, we defined a set of essentially $\frac{\de}{u} \times \frac{\De}{u} \times 1$ prisms $\calR_4^{T_u}$, supporting shadings $Y^{T_u}$ and $Y_1^{T_u}$. For each $R \in \calR_3$, we have a collection $\calP_R$ of $\de \times \De \times \frac{\De}{u}$ prisms. Finally, for each $P \in \calP_R$, we have a collection $\T_{P,R} \subset \{T \cap P : T \in \T_1(R), P \in \calP_R\}$. 

    We also defined small parameters $\eta, \eta_1, \beta_2, \beta_1$, and $\zeta$. These parameters have the following dependencies: $\eta_1 \sim \eta = \eta(\beta_2, \beta_1, \zeta, \omega)$, $\beta_2 = \beta_2(\beta_1, \omega)$, $\beta_1 = \beta_1(\omega)$, and $\zeta = \zeta(\omega)$. The values of these parameters will be finalized over the remainder of the proof.

    We use the following properties of these objects in the remainder of the proof. As an aid to the reader, we repeat the bounds proven up to this point which we need in the rest of the proof, keeping the same equation numbering for those bounds. The properties are presented here roughly in the order they will be used.

    \begin{enumerate}[label = (\alph*)]
        \item There exists a value $\mu_1^{T_u}$ (not depending on $T_u$) such that for each $T_u \in \T'_u$, $Y^{T_u}_1$ is a shading on $\calR_4^{T_u}$ with constant multiplicity $\mu_1^{T_u}$. For each $Q \in Y_1^{T_u}(\calR_4^{T_u})$, $\{n(R^{T_u}) : Q \in Y_1^{T_u}(\calR_4^{T_u})\}$ is $\beta_2$-broadly spread in a $\theta$-ball. These properties were established in Reduction VI. 
        \item For each $T_u \in \T'_u$, each $R \in \calR_4[T_u]$, and each $\frac{\de}{u}$-cube $Q \in Y_1^{T_u}(R^{T_u})$, $\phi^{-1}_{T_u}(Q)$ intersects $\sim M$ $\de$-cubes from $Y_8(R)$ and each $\de$-cube from $Y_8(R)$ intersects $\phi^{-1}_{T_u}(Q)$ for one choice of $Q \in Y_1^{T_u}(\calR_4^{T_u})$. Moreover, \[M \lesssim u^{-1}\tag{\ref{eqn:Mupperbound}}\]and \[M \gtrapprox_{\de} \de^{\beta_1 + 7\eta_1}u^{-1}.\tag{\ref{eqn:Mlowerbound}}\]
        \item We break (\ref{eqn:78density}) into two useful bounds: \begin{align*}
            \sum_{R \in \calR_4} |Y_8(R)|_{\de} &\sim M \sum_{T_u \in \T'_u} \sum_{R^{T_u} \in \calR_4^{T_u}} |Y_1^{T_u}(R^{T_u})|_{\de/u} \tag{\ref{eqn:78density}.A}\\
            \sum_{R \in \calR_4} |Y_8(R)|_{\de} &\gtrapprox_{\de} \de^{\beta_2} \sum_{R \in \calR_4} |Y_7(R)|_{\de} \tag{\ref{eqn:78density}.B}
        \end{align*}
        Combining (\ref{eqn:7shadingbound}) and (\ref{eqn:78density}.B), we see that \begin{equation}\label{eqn:8shadingbound}
            \sum_{R \in \calR_4} |Y_8(R)|_{\de} \gtrapprox_{\de} \de^{\beta_2 + \beta_1 + 6\eta_1} |\calR_4|\de^{-1-t}\De^t.
        \end{equation}
        \item For each $R \in \calR_3$ (and hence each $R \in \calR_4$), there exists a collection of $\de \times \De \times \frac{\De}{u}$ prisms $\calP_R$ such that for each $P \in \calP_R$, $\T_{P, R} = \{P \cap T : T \in \T_1(R)\}$ satisfies the $t$-Frostman Convex Wolff Axiom in $P$ with error $\de^{-\zeta}$ and multiplicity $\le \de^{-\eta}\left(\frac{\De}{\de}\right)^t$. Moreover, $Y_4(R) \cap P$ is supported on $\T_{P, R}$. These properties were established in Reduction III. We note two important consequences of this: if $R \subset T_u$, then $\T_{P,R}^{T_u} = \{\phi_{T_u}(P \cap T) : T \in \T_1(R)\}$ satisfies the $t$-Frostman Convex Wolff Axiom in $P^{T_u}:= \phi_{T_u}(P)$ with error $\de^{-\zeta}$ and $Y_1^{T_u}(R^{T_u}) \cap P^{T_u}$ is supported on $\T_{P,R}^{T_u}$. 
        \item By (\ref{eqn:43card}), (\ref{eqn:R31}), and (\ref{eqn:R10}), we have \begin{equation}\label{eqn:R40}
            |\calR_4| \gtrapprox_{\de} \de^{\beta_1 + 13\eta_1 + 2\eta}|\calR|.
        \end{equation}We also recall our assumption that $\calR$ satisfies the $t$-Frostman Convex Wolff Axiom with error $\de^{-\eta}$.
        \item Recall that at each $\de$-cube $Q_{\de} \in Y_8(\calR_4)$, $\{R \in \calR_4 : Q_{\de} \in Y_8(R)\}$ are all contained in a $u$-tube in $\T'_u$. For each $\de$-cube $Q_{\de} \in Y_8(\calR_4[T_u])$, there exsts a $\frac{\de}{u}$-cube $Q_{\de/u}$ such that $Q_{\de/u} \in Y_1^{T_u}(\calR_4^{T_u})$ and $Q_{\de/u} \cap \phi_{T_u}(Q_{\de}) \neq \emptyset$. For each $R \in \{R \in \calR_4 : Q_{\de} \in Y_8(R)\}$, $Q_{\de/u} \in Y_1^{T_u}(R^{T_u})$. Recap item (a) implies that $\{R^{T_u} \in \calR_4^{T_u} : Q_{\de/u} \in Y_1^{T_u}(R^{T_u})\}$ are all contained in a $2\theta \times 1 \times 1$ slab. Undoing the rescaling through $T_u$ maps each such slab to a $2u \theta \times u \times 1$ prism, which contains each $R \in \phi_{T_u}^{-1}(\{R^{T_u} \in \calR_4^{T_u} : Q_{\de/u} \in Y_1^{T_u}(R^{T_u})\})$ and hence each $R \in\{R \in \calR_4 : Q_{\de} \in Y_8(R)\}$. We conclude that $\{R \in \calR_4 : Q_{\de} \in Y_8(R)\}$ are all contained in a common $2u\theta \times u \times 1 $ prism. An important consequence is that for any fixed $R_0$, each $R \in \calR_4$ with $Y_8(R) \cap Y_8(R_0)$ is contained in a common $2u\theta \times u \times 1$ prism concentric with $R_0$. We denote the $2u\theta \times u \times 1$ prism by $\env(R_0)$ and call the fact that each prism $R$ with $Y_8(R) \cap Y_8(R_0)$ is contained in $\env(R_0)$ the \emph{envelope property}\footnote{ The idea being that the prisms in $\calR_4$ are pages being put into larger envelopes. Yixuan Pang suggested this name to the author in a similar context for Kakeya problem.}. 
        \item Recall that the set of prisms in $\calR_3$ with the shading $Y_6$ is $\beta_1$-broad in $\T_u$ with error $\lessapprox_{\de} \de^{-5\eta_1}$ and broadness multiplicity $1$. This was established at the end of Reduction IV. By (\ref{eqn:67shadingbound}) and (\ref{eqn:78density}.B), we see that \begin{equation}\label{eqn:68shadingbound}\sum_{R \in \R_4} |Y_8(R)|_{\de} \gtrapprox_{\de} \de^{\beta_2} \sum_{R \in \calR_3} |Y_6(R)|_{\de}.\end{equation}
    \end{enumerate}

    This concludes the recap. We now consider three possibilities, at least one of which must occur:
    \begin{enumerate}[label = \arabic*.]
        \item $\theta \ge \left(\frac{\de}{\De}\right)^{\omega/(10t + 5)}$. This is referred to as the very transverse case in Section \ref{subsec:bpop3}.
        \item $u \le \De \de^{-\omega/(10t + 5)}$.  This is referred to as the very narrow case in Section \ref{subsec:bpop3}.
        \item $u \ge \De \de^{-\omega/(10t + 5)}$ and $\theta \le \left(\frac{\de}{\De}\right)^{\omega/(10t +5)}$. This is referred to as the brush case in Section \ref{subsec:bpop3}.
    \end{enumerate}
    We show the proof closes for each of these possibilities, finishing our argument.

    \subsection*{Possibilities 1 and 2 common setup} Possibilities 1 and 2 are both variations of the argument in the slightly transverse case of Proposition \ref{prop:largeprism}. They start in the same way, which requires no assumptions on $u$ or $\theta$, so we present their common set-up before breaking into the two cases.

    For each $T_u \in \T_u$ and each $R \in \calR_4[T_u]$, we know the prisms $P \in \calP_R$ are aligned in the long direction of $\T_u$, so $\phi_{T_u}(P)$ is essentially a $\frac{\de}{u} \times \frac{\De}{u} \times \frac{\De}{u}$ prisms, which we denote $P^{T_u}$. Define by $\calP^{T_u}_R = \{P^{T_u} : P \in \calP_{R}\}$ and $\calP^{T_u} = \bigsqcup_{R \in \calR_4[T_u]} \calP_R^{T_u}$. Set $A = \bigsqcup_{T_u \in \T'_u} Y_1^{T_u}(\calR_4^{T_u})$, $B = \bigsqcup_{R \in \calR_4} \calP_R$, and \[E = \{(Q, P^{T_u}) \in A \times B : Q \in Y_1^{T_u}(R^{T_u}), Q \cap P^{T_u} \neq \emptyset, P \in \calP_R\}.\]This defines a bipartite graph on $A  \sqcup B $. Apply Lemma \ref{lem:WZ2547} to $(A \sqcup B, E)$. Let $E'$ denote the remaining collection of edges. For $P \in \calP^{T_u}_R$, define a shading $Y_2^{T_u}(P) \subset P^{T_u} \cap Y_1^{T_u}(R^{T_u})$ by $Q \in Y_2^{T_u}(P^{T_u})$ if $(Q, P^{T_u}) \in E'$. Define $\calP^{T_u}_{1} = \{P \in \calP^{T_u} : Y_2^{T_u}(P^{T_u}) \neq \emptyset\}$ and $\calP_1 = \bigsqcup_{T_u \in \T'_u} \calP_1^{T_u}$. By construction, \[|E| = \sum_{T_u \in \T'_u}\sum_{P^{T_u} \in \calP^{T_u}} |Y_1^{T_u}(P^{T_u})|_{\de/u} \sim \sum_{T_u \in \T'_u}\sum_{R^{T_u} \in \calR^{T_u}} |Y_1^{T_u}(R^{T_u})|_{\de/u} = \mu_1^{T_u}\sum_{T_u \in \T'_u} |Y_1^{T_u}(\calR^{T_u})|_{\de/u}.\]Clearly, $|A| = \sum_{T_u \in \T'_u} |Y_1^{T_u}(\calR_4[T_u])|_{\de/u}$, so the conclusion of Lemma \ref{lem:WZ2547} and Remark \ref{rmk} imply that \begin{equation}\label{eqn:mu2}\mu_2^{T_u}(Q) \gtrsim \frac{|E|}{\sum_{T_u \in \T'_u} |Y_1^{T_u}(\calR_4[T_u])|_{\de/u}} \gtrsim \mu_1^{T_u}.\end{equation}We know that $|\calP_R| \lesssim \frac{u}{\De}$ for each $R \in \calR_4$, and hence $|B| \le |\calR_4|\frac{u}{\De}$. Since $M \le \frac{1}{u}$ by (\ref{eqn:Mupperbound}), we can use (\ref{eqn:78density}.A) to bound \[|E| = \sum_{T_u \in \T'_u}\sum_{P^{T_u} \in \calP^{T_u}} |Y_1^{T_u}(P^{T_u})|_{\de/u} \sim \frac{1}{M}\sum_{R \in \calR_4} |Y_8(R)|_{\de} \gtrsim u\sum_{R \in \calR_4} |Y_8(R)|_{\de}.\]Using (\ref{eqn:8shadingbound}), we see that \[\sum_{R \in \calR_4} |Y_8(R)|_{\de} \gtrapprox_{\de} \de^{\beta_1 + \beta_2 + 6\eta_1}|\calR_4|\de^{-1-t}\De^t.\]Combining the previous two inequalities, we arrive at another bound for $E$: \begin{equation}\label{eqn:Ebound}|E| \gtrapprox_{\de}\de^{\beta_1 + \beta_2 + 6\eta_1} \frac{u}{\de} \left(\frac{\De}{\de}\right)^t|\calR_4|.\end{equation}Therefore, for each $P^{T_u} \in \calP^{T_u}_1$, \begin{equation}\label{eqn:Pbound}|Y_2^{T_u}(P^{T_u})|_{\de/u} \gtrsim \frac{|E|}{|\calP^{T_u}|} \gtrapprox_{\de} \frac{\de^{\beta_1 + \beta_2 + 6\eta_1}\frac{u}{\de}\left(\frac{\De}{\de}\right)^t|\calR_4|}{|\calR_4|\frac{u}{\De}} = \de^{\beta_1 + \beta_2 + 6\eta_1}\left(\frac{\De}{\de}\right)^{t+1}.\end{equation}

    For each $T_u \in \T'_u$ and each $P_0^{T_u} \in \calP_1^{T_u}$ with $P_0^{T_u} \in \calP_R^{T_u},$ we apply Lemma \ref{lem:brush} with $P_0^{T_u}$ taking the place of $P_0$, $\calP_1^{T_u}$ taking the place of $\calP$, $\T^{T_u}_{P^{T_u}_0, R}$ standing in for $\T(P)$ (recall this was defined in Recap (d)), and $Y_2^{T_u}(\T(P_0^{T_u}))$ the shading on $P_0^{T_u}$. Let $G^{T_u}(P^{T_u}_0)$ be the $\frac{3\theta\De}{u} \times \frac{3\De}{u} \times \frac{3\De}{u}$ prism concentric with $P^{T_u}_0$. We apply Lemma \ref{lem:brush} with $\tau$ replaced by some $\gamma  = \gamma(\omega)> 0$ to be determined over the course of the proof. The lemma gives us a choice of $\eta'$. We choose $\beta_1, \beta_2$ and $\eta_1$ all less than $\eta'/10$. We then take $\beta_2$ as itself, which means that we must check that for each $Q \in Y_2^{T_u}(\calP^{T_u})$, $\{n(P^{T_u}) : Q \in Y_2^{T_u}(P^{T_u})\}$ is $\beta_2$-broadly spread in a $\theta$-ball with error $\lesssim 1$. This is the case because, as noted in Recap (a), $\{n(R^{T_u}) : Q \in Y_1^{T_u}(R^{T_u})\}$ is $\beta_2$-broadly spread in a $\theta$-ball with error $\lesssim 1$, which implies the same for $\{n(P^{T_u}) : Q \in Y_1^{T_u}(P^{T_u})\}$. By (\ref{eqn:mu2}), $\mu_2^{T_u}(Q) \gtrsim \mu_1^{T_u}(Q)$ for each $Q$, so the same holds for $\{n(P^{T_u}) : Q \in Y_2^{T_u}(P^{T_u})\}.$ The lemma finally gives a value $\de'_0 > 0$. The objects we apply Lemma \ref{lem:brush} to have ambient scale $\frac{\de}{u}$, which we need to be $\le \de'_0$. We certainly know that $\frac{\de}{u} \le \frac{\de}{\De} \le \de^{\e}$. If necessary, we decrease $\de_0$ so that $ \de_0^{\e}\le \de'_0$. We can indeed apply Lemma \ref{lem:brush}, which gives a family $\mathscr{S}^{T_u}[P_0^{T_u}] \subset \mathcal{P}_1^{T_u}$ satisfying the $t$-Frostman Convex Wolff Axiom in $G^{T_u}(P^{T_u}_0)$ with error $\de^{-\gamma}$, and for each $P^{T_u} \in \mathscr{S}^{T_u}[P_0^{T_u}]$, \begin{equation}\label{eqn:env}Y_2^{T_u}(\T(P^{T_u})) \cap Y_2^{T_u}(\T(P_0^{T_u})) \neq \emptyset.\end{equation}
    
    We then apply Lemma \ref{lem:RFPB} with $\mathscr{S}^{T_u}[P^{T_u}_0]$ taking the place of $\mathcal{P}$, $\T^{T_u}_{P, R}$ standing in for $\T(P)$ for each $P \in \mathscr{S}^{T_u}[P^{T_u}_0]$, and $Y^{T_u}_2$ taking the place of $Y$. We have error terms $C_1 = \de^{-\gamma}$ from the previous paragraph, $C_2 = \de^{-\zeta}$ from Recap (d), $C_3 = \de^{-\eta}$ since $|\T^{T_u}_{P, R}| \le |\T(R)| \le \de^{-\eta}\left(\frac{\De}{\de}\right)^t$, and $C_4 \lessapprox_{\de} \de^{-\beta_1 -\beta_2 - 6\eta_1}$ from (\ref{eqn:Pbound}). We conclude that \begin{equation}\label{eqn:GBound}\left|\bigcup_{P^{T_u}\in \mathscr{S}^{T_u}[P_0^{T_u}]} Y^{T_u}_2(P^{T_u})\right|_{\de/u} \gtrapprox_{\de} \de^{2\beta_1 + 2\beta_2 + 12 \eta_1 + \eta + \gamma + \zeta}\left(\frac{\theta\De}{\de}\right)^t\left(\frac{\De}{\de}\right)^{t+1}.
    \end{equation}Denote the right-hand side of (\ref{eqn:GBound}), including the implicit logarithmic term, by $N$.

    For one choice of $T_u$ and one $P^{T_u} \in \calP^{T_u}_R$, we have a $\frac{3\theta \De}{u} \times \frac{3\De}{u} \times \frac{3\De}{u}$ prism $G^{T_u}(P^{T_u})$, which is mapped under $\phi^{-1}_{T_u}$ to a $3\theta \De \times 3\De \times \frac{3\De}{u}$ prisms $G(P)$. We define $\mathscr{S}[P] = \phi_{T_u}^{-1}(\mathscr{S}^{T_u}[P^{T_u}])$, a set of $\de \times \De \times \frac{\De}{u}$ prisms satisfying the $t$-Frostman Convex Wolff Axiom in $G(P)$. Since $\phi^{-1}_{T_u}(Y^{T_u}_2(P))$ contains a collection of $N$ $\de \times \de \times \frac{\de}{u}$ prisms, each intersecting $\gtrsim M$ distinct $\de$-cubes from $Y_8(\calR_4[T_u])$, we conclude that for each $P \in \calP_1$, \begin{equation}\label{eqn:GBound2}\left| Y_8(\mathscr{S}[P])\right|_{\de} \gtrsim MN.\end{equation}

    Our choices of $T_u$ and $P \in \calP_1^{T_u}$ were arbitrary throughout the common setup, so we conclude that (\ref{eqn:GBound}) and (\ref{eqn:GBound2}) holds for all choice of $T_u$ and all $P \in \calP^{T_u}_1$. This concludes the common setup for Possibilities 1 and 2. We now close the argument for both.
    \subsection*{Finishing possibility 1}We assume $\theta \ge \left(\frac{\de}{\De}\right)^{\omega/(10t + 5)}$ and prove Conclusion A holds.

    Cover $G(P)$ with $\sim \frac{1}{u}$ many $3\theta \De \times 3\De \times 3\De$ prisms and pigeonhole to conclude that one of them, which we label $\tilde{G}$, satisfies $\left| \tilde{G} \cap Y_8(\calR[T_u])\right|_{\de} \gtrsim uMN$. Let $B_{\De}$ be a $3\De$-ball containing $\tilde{G}$. We see that \begin{equation}\label{eqn:MNbound}\left| B_{\De} \cap Y_8(\calR[T_u])\right|_{\de} \gtrsim uMN.\end{equation}

    Suppose we can prove \begin{equation}\label{eqn:MNbound2}uMN \gtrapprox_{\de} \left(\frac{\De}{\de}\right)^{2t+1-\omega/2}.\end{equation}Then (\ref{eqn:MNbound}) implies that \[\left| B_{\De} \cap Y_8(\calR[T_u])\right|_{\de} \gtrapprox_{\de} \left(\frac{\De}{\de}\right)^{\omega/2}\left(\frac{\De}{\de}\right)^{2t+1-\omega} \gtrapprox_{\de} \de^{-\e\omega/2}\left(\frac{3\De}{\de}\right)^{2t+1-\omega}.\]Taking $\de_0$ sufficiently small so that $\de^{-\e\omega}$ is larger than the reciprocal of the implicit logarithmic loss, we see that Conclusion A holds. 
    
    It remains to prove (\ref{eqn:MNbound2}). Recall by (\ref{eqn:Mlowerbound}) that $M \gtrapprox_{\de} \de^{\beta_1 + 7\eta_1}u^{-1}$. We have from (\ref{eqn:GBound}) that $N \gtrapprox_{\de} \de^{2\beta_1 + 2\beta_2 + 12 \eta_1 + \eta + \gamma + \zeta}\left(\frac{\theta\De}{\de}\right)^t\left(\frac{\De}{\de}\right)^{t+1}$. Since $\theta \ge \left(\frac{\de}{\De}\right)^{\omega/(10t + 5)}$, we know that \[\left(\frac{\theta\De}{\de}\right)^{t} \ge \left(\frac{\De}{\de}\right)^{-\omega/(10 + 5/t)}\left(\frac{\De}{\de}\right)^t \ge \left(\frac{\De}{\de}\right)^{-\omega/10}\left(\frac{\De}{\de}\right)^t.\]We conclude that \[uMN \gtrapprox_{\de} \de^{3\beta_1 + 2\beta_2 + 19\eta_1 + \eta + \gamma + \zeta}\left(\frac{\De}{\de}\right)^{2t+1-\omega/10}.\]We can freely choose $\beta_1, \beta_2, \eta_1, \eta, \gamma$, and $\zeta$ to be less than $\e\omega/1000$ and conclude that \[uMN \gtrapprox_{\de} \de^{\e\omega/5}\left(\frac{\De}{\de}\right)^{2t+1-\omega/10}.\]Since $\de^{\e} \ge \frac{\de}{\De}$, we finally have \[uMN \gtrapprox_{\de} \left(\frac{\De}{\de}\right)^{2t+1-\omega/2}.\]This completes the Possibility I case.

    \subsection*{Finishing possibility 2}We assume that $u \le \De \de^{-\omega/(10t + 5)}$ and prove Conclusion A. 

    We first aim to bound $|\calP_1|$. Recall that this is the vertices one side of a bipartite graph with edge set $E'$, where the other side of the bipartite graph comprises a set of $\de$-cubes. As an immediate consequence of how we defined $E'$, we see that if $(Q, P^{T_u}) \in E'$, then $Q \in Y_1^{T_u}(P^{T_u}) \subset \T^{T_u}_{P, R}$. We can therefore bound the maximal degree of $P^{T_u} \in \calP_1$ by $|Y_{\text{full}}(\T^{T_u}_{P,R})|_{\de/u}$. This is a collection of $\le \de^{-\eta}\left(\frac{\De}{\de}\right)^t$ many $\frac{\de}{u} \times \frac{\de}{u} \times \frac{\De}{u}$-tubes, so $|Y_{\text{full}}(\T^{T_u}_{P,R})|_{\de/u} \le \de^{-\eta}\left(\frac{\De}{\de}\right)^{t+1}$. It follows that \[|E'| \le |\calP_1|\max_{P^{T_u} \in \calP_1} \deg_{E'}(P^{T_u}) \le |\calP_1|\de^{-\eta}\left(\frac{\De}{\de}\right)^{t+1},\]which together with (\ref{eqn:Ebound}) implies \begin{equation}\label{eqn:P1bound}|\calP_1| \gtrapprox_{\de} \de^{\beta_1 + \beta_2 + 6\eta_1 + \eta}\frac{u}{\De}|\calR_4|.\end{equation}Going forward, we denote by $\gamma_1 = \beta_1 + \beta_2 + 6\eta_1 + \eta$.

    Define $\tilde{\calP}_1 = \{\phi_{T_u}^{-1}(\calP^{T_u}) : \calP^{T_u} \in \calP_1\} \subset \bigsqcup_{R \in \calR_4} \calP_R$. Clearly $|\tilde{\calP}_1| = |\calP_1|$. Define \[\calR_5 = \{R \in \calR_4: \calP_R \cap \tilde{\calP}_1 \neq \emptyset\}.\]Since $|\calP_R| \le \frac{u}{\De}$ for each $R \in \calR_5$, (\ref{eqn:P1bound}) implies that $|\calR_5| \ge \de^{\gamma_1}|\calR_4|$. Together with (\ref{eqn:R40}), we see that $|\calR_5| \ge \de^{3\gamma_1}|\calR|$. Since $\calR$ satisfies the $t$-Frostman Convex Wolff Axiom with error $\de^{-\eta} \le \de^{-\gamma_1}$, we see that $\calR_5$ satisfies the $t$-Frostman Convex Wolff Axiom with error $\le \de^{-4\gamma_1}$. Denote by $\calR'_{\text{env}} = \{\text{env}(R) : R \in \calR_5\}$. With a Vitali-style covering argument\footnote{ Vitali covering-style reductions are more common for families of tubes or of slabs, where they correspond nicely to the actual Vitali covering theorem applied on the affine Grassmannian. The Vitali covering reduction for prisms we use here follows from the the case of tubes and of slabs by applying a Vitali covering argument to $\T'_u$ and then on the set of $\theta \times 1 \times 1$ prisms concentric with elements of $\calR_5^{T_u}$ for each remaining $T_u$.}, we refine this set to some $\calR_{\text{env}}$ so that $\calR_5[3R_{\text{env}}] \cap \calR_5[3R'_{\text{env}}] = \emptyset$ for each $R_{\text{env}} \neq R'_{\text{env}}$ and if we set \[\calR_6 = \{R \in \calR_5 : R \subset R_{\text{env}} \text{ for some }R_{\text{env}} \subset \calR_{\text{env}}\},\]then $|\calR_6| \gtrsim |\calR_5|$. Since $\calR_5$ satisfies the $t$-Frostman Convex Wolff Axiom with error $\le \de^{-4\gamma_1}$, we see that $\calR_6$ satisfies the $t$-Frostman Convex Wolff Axiom with error $\lesssim \de^{-4\gamma_1}$. This implies that for each $R_{\text{env}}$ in $\calR_{\text{env}}$, \begin{equation}\label{eqn:letitend}|\calR_6[R_{\text{env}}]|\le \de^{-4\gamma_1}(\theta u^2)^t |\calR_6|.\end{equation}Since $\calR_{\text{env}}$ is a $1$-partitioning cover of $\calR_6$, $\sum_{R_{\text{env}} \in \calR_{\text{env}}} |\calR_6[R_{\text{env}}]| = |\calR_6|$. Then summing (\ref{eqn:letitend}) over $\calR_{\text{env}}$ and rearranging, we conclude that \begin{equation}\label{eqn:Renvlowerbound}|\calR_{\text{env}}| \gtrsim \de^{4\gamma_1} (\theta u^2)^{-t}.\end{equation}
        
    Each $R_{\text{env}} \in \calR_{\text{env}}$ contains a non-empty family \[\tilde{\calP}_1(R_{\text{env}}) = \{ P \in \tilde{\calP}_1 : P \in \calP_R \text{ for some } R \subset R_{\text{env}}\}.\]Choose one $P_{R_{\text{env}}} \in \tilde{\calP}_1(R_{\text{env}})$ for each $R_{\text{env}} \in \calR_{\text{env}}$. We claim that each $P \in \mathscr{S}[P_{R_{\text{env}}}]$, $P \in \calP_R$ for some $R \subset 2R_{\text{env}}$. Let $R_0 \in \calR_6$ be such that $P_{R_{\text{env}}} \in \calP_{R_0}$, which by the definition of $\tilde{\calP}_1(R_{\text{env}})$ implies that $R_0 \subset R_{\text{env}}$. By (\ref{eqn:env}), $Y_2^{T_u}(P^{T_u}_{R_{\text{env}}}) \cap Y_2^{T_u}(P^{T_u}) \neq \emptyset$, so $Y_2^{T_u}(R^{T_u}_0) \cap Y_2^{T_u}(R^{T_u}) \neq \emptyset$. By the envelope property (see Recap (f)) and the fact that $R_0 \subset 2R_{\text{env}}$, we conclude that $R \subset \text{env}(R_0) \subset 2R_{\text{env}}$, as desired. 
    
    Now suppose there exists $R_{\text{env}} \neq R'_{\text{env}}$ with $Y_8(P) \cap Y_8(P') \neq \emptyset$ for some $P \in \mathscr{S}[P_{R_{\text{env}}}]$ and $P' \in \mathscr{S}[P_{R'_{\text{env}}}]$. We know that $P \in \calP_R$ for some $R \subset 2R_{\text{env}}$ and $P' \in \calP_{R'}$ for some $R' \subset 2R'_{\text{env}}$. Furthermore, since $Y_8(P) \subset Y_8(R)$ and $Y_8(P') \subset Y_8(R')$, again by the envelope property we conclude that $R' \subset \env(R) \subset 3R_{\text{env}}$. It follows that $R' \in \calR_5[3R_{\text{env}}] \cap \calR_5[3R'_{\text{env}}]$, contradicting construction of $\calR_{\env}$. Recalling our notation $Y_8(\mathscr{S}[P_{R_{\env}}]) = \bigcup_{P \in \mathscr{S}[P_{R_{\env}}]} Y_8(P)$, we conclude that $Y_8(\mathscr{S}[P_{R_{\env}}]) \cap Y_8(\mathscr{S}[P_{R'_{\env}}]) = \emptyset$ and hence \[|Y(\calR)|_{\de} \ge \sum_{R_{\env} \in \calR_{\env}} |Y_8(\mathscr{S}[P_{R_{\env}}])|_{\de}.\]Together with (\ref{eqn:GBound2}), this implies that \begin{equation}\label{eqn:we4}|Y(\calR)|_{\de} \ge MN|\calR_{\text{env}}|.\end{equation}
    
    Combining (\ref{eqn:Mlowerbound}), our definition of $N$ from (\ref{eqn:GBound}), and (\ref{eqn:Renvlowerbound}), we conclude that \[MN|\calR_{\text{env}}| \gtrapprox_{\de}\de^{3\beta_1 + 2\beta_2 + 19\eta_1 +\eta + \gamma + \zeta + 4\gamma_1}\left(\frac{\De}{u}\right)^{2t+1}\de^{-2t-1}\]Recalling our assumption that $u \le \de^{-\omega/(10t + 5)}\De$ and choosing error terms so that \[\de^{3\beta_1 + 2\beta_2 + 19\eta_1 +\eta + \gamma + \zeta + 4\gamma_1} \ge \de^{\omega/5},\]we see that $MN|\calR_{\text{env}}| \gtrapprox_{\de} \de^{\omega/2-2t-1}$. So long as $\de_0$ is sufficiently small, we have $MN|\calR_{\text{env}}| \ge \de^{-2t-1+\omega}$. Using this bound in (\ref{eqn:we4}), we reach Conclusion A.

    \subsection*{Possibility 3}We use a brush argument similar to the second subcase of the slightly broad case in Proposition \ref{prop:largeprism} to reach Conclusion B. 

    We assume $u \ge \De \de^{-\omega/(10t+5)}$ and $\theta \le \left(\frac{\de}{\De}\right)^{\omega/(10t + 5)}$ Set $\theta_0 := \left(\frac{\de}{\De}\right)^{\omega/(10t + 5)} \ge \theta$, $\overline{\de} = 3u\theta_0$ and $\overline{\De} = 3u$. These are the final choices of scales in Conclusion B. As promised, we have $\frac{\overline{\De}}{\overline{\de}} = \theta_0^{-1} = \left(\frac{\De}{\de}\right)^{\omega/(10t + 5)}$. 

    Refine $Y_8$ to some $Y_9 = \{Q \in Y_8(\calR_4) : \mu_8(Q) \gtrapprox_{\de} \de^{\beta_2}\mu_6(Q)\}$. Choosing the implicit logarithmic term sufficiently small, we ensure that $\sum_{R \in \calR_4} |Y_9(R)|_{\de} \gtrsim \sum_{R \in \calR_4} |Y_8(R)|_{\de}$, since otherwise (\ref{eqn:68shadingbound}) would be impossible. Define \begin{equation*}\calR_{\text{stem}} = \left\{R \in \calR_4 : |Y_9(R)|_{\de} \ge \de^{2\beta_1} \De^{t}\de^{-1-t}\right\}.\end{equation*}We claim that $|\calR_{\text{stem}}| \ge \de^{2\beta_1}|\calR|$. Suppose otherwise. Then since $|Y_9(R)|_{\de} \le |\T(R)|\de^{-1} \le \De^{t}\de^{-1-t-\eta}$, we know that \[\sum_{R \in \calR_{\text{stem}}} |Y_9(R)|_{\de} \le \de^{2\beta_1}|\calR_4|\De^t \de^{-1-t-\eta}.\]Clearly the same is true if we replace $\calR_{\text{stem}}$ with $\calR_4 \setminus \calR_{\text{stem}}$. We conclude that $\sum_{R \in \calR_4} |Y_9(R)|_{\de} \lesssim \De^{t}\de^{2\beta_1-1-t-\eta}|\calR_4|$ and hence $\sum_{R \in \calR_4} |Y_8(R)|_{\de} \lesssim \De^{t}\de^{2\beta_1-1-t-\eta}|\calR_4|$. We can choose $\beta_2, \eta_1,$ and $\eta$ to be much smaller than $\beta_1$ and reach a contradiction with (\ref{eqn:8shadingbound}), so long as $\de_0$ is sufficiently small. Therefore, $|\calR_{\text{stem}}| \ge \de^{2\beta_1}|\calR|$, as desired.

    Define $\overline{\calR} = \{N_{\overline{\de}}(\env(R)) : R \in \calR_{\text{stem}}\}$. We claim $\overline{\calR}$ satisfies the $t$-Frostman Convex Wolff Axiom with error $\le \de^{-\tau} \le \overline{\de}^{-\tau}$. Since $|\calR_{\text{stem}}| \ge \de^{2\beta_1}|\calR|$ and $\calR$ satisfies the $t$-Frostman Convex Wolff Axiom with error $\de^{-\eta}$, we know that $\calR_{\text{stem}}$ satisfies the $t$-Frostman Convex Wolff Axiom with error $\le \de^{-2\beta_1-\eta}$. Since the $t$-Frostman Convex Wolff Axiom is inherited upwards, we conclude that $\overline{\calR}$ satisfies the $t$-Frostman Convex Wolff Axiom with error $\le \de^{-2\beta_1 - \eta}$. We can freely take $\beta_1, \eta < \tau/3$ and conclude that $\overline{\calR}$ satisfies the $t$-Frostman Convex Wolff Axiom with error $\le \de^{-\tau}$.

    It remains to construct the shading $\overline{Y}$ and prove it has the desired density in $\overline{\calR}$. Since $u \ge \de^{-\omega/(10t + 5)}\De$, the definition of $\theta_0$ ensures that $\overline{\de} = 3\theta_0u > \De$. It follows that $\T_{\overline{\de}} = \calD_{\overline{\de}}(\calR)$ consists of $\overline{\de}$-tubes, and if $T = N_{\de}(R) \in \calD_{\overline{\de}}(\calR)$, then $\overline{Y}(T) = N_{\overline{\de}}(Y(R))$ is an $\de^{-\eta}$-dense shading on $\T_{\overline{\de}}$. We claim that for each $\overline{R} \in \overline{\calR}$, we can find a set $\T_{\overline{\de}}(\overline{R})$ satisfying the Frostman Convex Wolff Axiom in $\overline{R}$ with error $\le \de^{-\tau/2}$. If the claim holds, then we apply Lemma \ref{lem:RFPA} with $\T_{\overline{\de}}$ the set of tubes, $\overline{Y}$ the shading, and $\overline{R}$ the ambient prism. We conclude that \[\left|\bigcup_{T \in \T_{\overline{\de}}(\overline{R})} \overline{Y}(R)\right|_{\overline{\de}} \gtrapprox_{\de} \de^{-2\eta -\tau/2}\frac{\overline{\De}}{\overline{\de}^2}.\]So long as $\eta < \tau/5$ and $\de_0$ is sufficiently small so that the implicit logarithmic term is $\le \de^{-\tau/10}$, we ensure that $\overline{Y}(\overline{R})$ is $\ge \de^{\tau}$ dense, as desired.

    It remains to construct the set of tube $\T_{\overline{\de}}(\overline{R}_0)$ for an arbitrary $\overline{R}_0 \in \overline{\calR}$. We know that $\overline{R}_0 = N_{\overline{\de}}(\env(R_0))$ for some $R_0 \in \calR_{\text{stem}}$. By the construction of $\calR_{\text{stem}}$, we know that $|Y_9(R_0)|_{\de} \ge \de^{2\beta_1}\De^t \de^{-1-t}$. Since $Y_9(R_0)$ is supported on a union of $\le \de^{-\eta}\left(\frac{\De}{\de}\right)^t$ many $\de$-tubes, we know one tube $T_0$ has $|Y_9(T_0)|_{\de} \ge \de^{2\beta_1 + \eta}\de^{-1}$. Each $\de$-cube $Q \in Y_9(T_0)$ is also contained in $Y_6(\calR_3)$, so $\{\dir(R) : R \in \calR_3, Q \in Y_6(R)\}$ is $\beta_1$-broad in a $u$-ball with error $\lessapprox_{\de} \de^{-5\eta_1}$. Since $\mu_9(Q) \gtrapprox_{\de} \de^{\beta_2} \mu_6(Q)$, we see that $\mu_9(Q)$ is $\beta_1$-broad in a $u$-ball with error $\lessapprox_{\de} \de^{-5\eta_1 -\beta_2}$. Following the discussions in Remark \ref{rmk2}, there exists some $R \in \calR_4$ with $Q \in Y_9(R)$ and $\angle(\dir(R), \dir(R_0)) \gtrapprox_{\de} \de^{C}u$, where $C = \frac{5\eta_1 + \beta_2}{\beta_1}$. 

    Set $\tilde{u} = \de^Cu$. Since $|Y_9(T_0)|_{\de} \ge \de^{2\beta_1 + \eta}\de^{-1}$, \begin{equation}\label{eqn:J}J = |Y_9(T_0)|_{\overline{\de}/\tilde{u}} \ge \de^{2\beta_1 + \eta} \frac{\tilde{u}}{\overline{\de}} \ge \de^{2\beta_1 + \eta + C}\frac{u}{\overline{\de}}.\end{equation}Enumerate the $J$ many $\frac{\overline{\de}}{\tilde{u}}$-cubes incident to $Y_9(T_0)$ as $Q_1, \dots Q_J$. For each $j$ from $1$ to $J$, there exists some prism $R_j$ with $Q_j \cap Y_9(R_j) \neq \emptyset$ and $\angle(\dir(R_0), \dir(R_j)) \gtrapprox_{\de} \de^C u$. Let $\overline{T}_j = N_{\overline{\de}}(R_j) \in \T_{\overline{\de}}$. By the envelope property, $R_j \subset \env(R_0)$, so $\overline{T}_j \subset \overline{R}_0$. Set $\T_{\De}(\overline{R}_0 ) = \{\overline{T}_j : j = 1, \dots J\}$. 
    
    It remains to prove $\T_{\overline{\de}}(\overline{R}_0 )$ satisfies the $1$-Frostman Convex Wolff Axiom in $\overline{R}_0$ with error $\le \de^{-\tau/2}$. Suppose otherwise. It follows that $\T_{\overline{\de}}(\overline{R}_0)^{\overline{R}_0}$, a set of $\frac{\overline{\de}}{u} \times 1 \times 1$ slabs, has $1$-Frostman Convex Wolff Axiom error $> \de^{-\tau/2}$. Moreover, each slab intersects $T_0^{\overline{R}_0}$ with angle $\ge \de^C$. By (\ref{eqn:J}), so long as $\beta_1, \eta, C < \tau/100$, we must have that $J \ge \de^{\tau/4}\frac{u}{\overline{\de}}$ and hence $C_{1-\text{KT-CW}}(\T_{\overline{\de}}(\overline{R}_0)) \ge \de^{-\tau/4}$. Apply Lemma \ref{lem:WZ2546} to $\T_{\overline{\de}}(\overline{R}_0)^{\overline{R}_0}$ and let $\T'$ denote the resulting subset of $\T$ and $\calW$ the larger cover of $\T'$, by $\alpha \times 1 \times 1$ prisms. 
    
    We know that $|\T'| \gtrapprox_{\de} |\T|$, so using (\ref{eqn:WZ2546+}), we conclude that \begin{equation}\label{eqn:W}|\calW| \lessapprox_{\de} \de^{\tau/4}|\T'| \frac{\overline{\de}}{u\alpha}\end{equation}If $\alpha \ge \de^{C}$, (\ref{eqn:W}) implies $|\calW| = 0$, so long as $C$ is sufficiently small relative to $\tau$ and $\de_0$ is sufficiently small, so assume $\alpha < \de^C$. Then since each $W \in \calW$ contains some slab in $\T_{\overline{\de}}(\overline{R}_0)^{\overline{R}_0}$, $W$ intersects $T_0^{\overline{R}_0}$ with angle $\gtrsim \de^C$. If we denote by $\ell$ the center line of $T_0^{\overline{R}_0}$, we see that $\vol_{\R^1}(W \cap \ell) \lesssim \de^{-C}\alpha$. Together with \ref{eqn:W}, we conclude that \begin{equation}\label{eqn:almost1}\vol_{\R^1}\left(\bigcup_{T \in \T'} T \cap \ell\right) \lessapprox_{\de} \de^{\tau/4-C}|\T'|\frac{\overline{\de}}{u}.\end{equation}On the other hand, by construction $T \cap \ell$ is disjoint for different $T \in \T'$ and intersects $\ell$ with length $\ge \frac{\overline{\de}}{u}$, so \begin{equation}\label{eqn:almost2}\vol_{\R^1}\left(\bigcup_{T \in \T'} T \cap \ell\right) \ge |\T'| \overline{\de}{u}.\end{equation}Comparing (\ref{eqn:almost1}) with (\ref{eqn:almost2}), we reach a contradiction, so long as $C$ is sufficiently small relative to $\tau$ and $\de_0$ is sufficiently small. We conclude that $\T_{\overline{\de}}(\overline{R}_0)$ satisfies the $1$-Frostman Convex Wolff Axiom, completing the proof.
    
\end{proof}

\section{The proofs Lemmas \ref{lem:RFPA}, \ref{lem:RFPB}, \ref{lem:RFPC}, and \ref{lem:brush}}\label{sec:lem}

\subsection{The proofs of Lemmas \ref{lem:RFPA}, \ref{lem:RFPB}, and \ref{lem:RFPC}}\label{subsec:RFP}

These three lemmas are various $L^2$ arguments for incidences. In some cases, it may be convenient to know the precise relationship of each parameter in the error terms, so we keep track of those in these results. We recall the statement of each lemma and give the proof of each lemma in numerical order. For Lemmas \ref{lem:RFPA} and \ref{lem:RFPC}, the proofs are largely the same as results already in the literature, respectively \cite[Appendix A]{HSY21} and \cite[Lemma 4.4]{WZ24}. Lemma \ref{lem:RFPB} is a little more difficult and may be novel. Recall the initial common set-up for these lemmas: we have error terms $C_1, C_2, C_3,$ and $C_4 \ge 1$, a choice of $t \in (0, 1]$, and some scale $\de > 0$.

\begin{lemma*}[\textbf{\ref{lem:RFPA}}]
    Fix $\De \in (\de, 1]$. Suppose $(\T, Y)_{\de}$ has $|Y(T)|_{\de} \ge C_1^{-1}\de^{-1}$. Also, suppose for some $\de \times \De \times 1$ prism $R$ and all $T \in \T$, $T \subset R$  and $\T^R$ satisfies the $t$-Frostman Convex Wolff Axiom with error $C_2$. Then \[\left|\bigcup_{T \in \T} Y(T)\right|_{\de}\gtrsim |\log \de|^{-1}C_1^{-1}C_2^{-2}(\de/\De)^{-t}\de^{-1}.\]
    In particular, if $C_1, C_2 \le \de^{-\eta}$ and $\de$ is sufficiently small, then \[\left|\bigcup_{T \in \T} Y(T)\right|_{\de} \ge \de^{4\eta}(\de/\De)^{-t}\de^{-1}.\]
\end{lemma*}
\begin{proof}
    By Cauchy-Schwarz, we know that \begin{equation}\label{eqn:611}\left|\bigcup_{T \in \T} Y(T)\right|_{\de} \ge \frac{|\T|^2 (\inf_{T \in \T} |Y(T)|)^2_{\de}}{\sum_{T_1, T_2 \in \T} |Y(T_1) \cap Y(T_2)|_{\de}}.\end{equation}We see that \[\sum_{T_1, T_2 \in \T} |Y(T_1) \cap Y(T_2)|_{\de} \le |\T| \sup_{T_1 \in \T} \sum_{j = 0}^{|\log \de |} \sum_{\substack{T_2 \in \T: \angle(T_1, T_2) \in [2^j\de, 2^{j+1}\de]}} |Y(T_1) \cap Y(T_2)|_{\de}.\]
        Fix a choice of $T \in \T$. For any $T_2 \in \T$ with $\angle(T_1, T_2) \in [2^j\de, 2^{j+1}\de)$, we see that $T_2 \subset N_{2^{j+1}\de}(T_1) \cap R$, a prism with dimensions $\de \times 2^j\de \times 1$. Since $\T$ satisfies the $t$-Frostman Convex Wolff Axiom in $R$ with error $C_2$, \[|\{T_2 : \angle(T_1, T_2) \le 2^{j+1}\de\}| \le C_2\left(\frac{2^{j+1}\de}{\De}\right)^t|\T|.\]For each such $T_2$, we have that $Y(T_1) \cap Y(T_2)$ is contained in a $\de \times \de \times \frac{\de}{2^{j}\de}$ subset of $T_2$ and hence consists of $\le (2^j\de)^{-1} \le 2^{-tj}\de^{-1}$ $\de$-balls. It follows that \[\sum_{\substack{T_2 \in \T: \angle(T_1, T_2) \in [2^j\de, 2^{j+1}\de]}} |Y(T_1) \cap Y(T_2)|_{\de} \lesssim C_2 \de^{-1+t}\De^{-t}|\T|.\]
        As there are $\lesssim |\log \de|$ many choices of $j$, we conclude that \[\sum_{T_1, T_2 \in \T} |Y(T_1) \cap Y(T_2)|_{\de} \le |\T|^2|\log \de| \de^{-1+t}\De^{-t}.\] Since $|Y(T)|_{\de} \ge C_1^{-1}\de^{-1}$ for each $T \in \T$, plugging the above bound into (\ref{eqn:611}), we conclude that \[\left|\bigcup_{T \in \T} Y(T)\right|_{\de} \gtrsim |\log \de|^{-1} C_1^{-2}C_2^{-1} \de^{-1}(\de/\De)^{-t}.\]
\end{proof}

\begin{lemma*}[\textbf{\ref{lem:RFPB}}]
        Fix $\De \in (\de, 1]$ and $\nu \in (\De, 1]$. Suppose $\calP$ is a collection of $\de \times \De \times \De$ prisms satisfying $t$-Frostman Convex Wolff Axiom with error $C_1$ inside some $G$ a $\nu \times \De \times \De$ prism. Suppose furthermore that for each $P \in \calP$, we have a collection of tubes $\T(P)$ satisfying the $t$-Frostman Convex Wolff Axiom with error $C_2$ inside $P$ and $|\T(P)| \le C_3(\frac{\De}{\de})^t$. Finally, $\T(P)$ is equipped with a shading $Y$ such that $\left|\bigcup_{T \in \T(P)} Y(T)\right|_{\de} \ge C^{-1}_4 (\frac{\De}{\de})^{t+1}$.Then \[\left|\bigcup_{T \in \T} Y(T)\right|_{\de} \gtrsim |\log \de|^{-2}C_{4}^{2}C_3^{-1}C_2^{-1}C_1^{-1} \left(\frac{\nu}{\de}\right)^t \left(\frac{\De}{\de}\right)^{t+1}.\]
        In particular, if $C_i \le \de^{\eta}$ for $i = 1, 2, 3, 4$, then taking $\de$ sufficiently small, we have that \[\left|\bigcup_{T \in \T} Y(T)\right|_{\de} \gtrsim \de^{6\eta} \left(\frac{\nu}{\de}\right)^t \left(\frac{\De}{\de}\right)^{t+1}.\]
\end{lemma*}

We start the proof of (B) following the standard Cordoba argument, but to estimate the size of the intersection of the shadings between two different prisms requires a little more work. We must bound $|Y(P) \cap Y(P')|_{\de}$ from above for two prisms intersecting in some angle $\theta$. Denote the intersection of the prisms by $V$. It suffices to bound $|Y(P) \cap V|_{\de}$ from above. After pigeonholing on the set of tubes supporting $Y(P)$, we may assume they all make the same angle with $V$. If that angle is close to $\de$, then they are all contained in $V$ but we have fewer tubes intersecting $V$, since the tubes satisfy the $t$-Frostman Convex Wolff Axiom. If the angle is close to $1$, we have all tubes intersecting $V$ but the size of the intersection of each with $V$ is smaller. In both extremal cases, we arrive at the desired upper bound for the size of intersection. The intermediate angle case follows by combining the reasoning of the two extremal cases.

\begin{proof}
    We begin by reducing the set of tubes in each prism to ensure the shading is uniformly dense. Let $\T'(P) = \{T \in \T(P) : |Y(T)|_{\de} \ge 2(C_3C_4)^{-1}\frac{\De}{\de}\}$. With a union bound and our assumption on $|Y(T)|_{\de}$, we have that \[\left|\bigcup_{T \in \T \setminus \T'(P)} Y(T)\right|_{\de} \le \frac{1}{2}C_3^{-1}C_4^{-1}\frac{\De}{\de}|\T(P)| \le \frac{1}{2}\left|\bigcup_{T \in \T(P)} Y(T)\right|_{\de}.\]
        We conclude that $\left|\bigcup_{T \in \T'(P)} Y(T)\right|_{\de} \ge \frac{1}{2} C^{-1}_4 \left(\frac{\De}{\de}\right)^{t+1}$ and hence $|\T'(P)| \gtrsim C^{-1}_4 \left(\frac{\De}{\de}\right)^{t}$. As $\T(P)$ has multiplicity $\le C_3 \left(\frac{\De}{\de}\right)^t$, we see that $|\T'(P)| \ge (C_3 C_4)^{-1} |\T(P)|$, so since $\T(P)$ satisfies the $t$-Frostman Convex Wolff Axiom with error $C_2$, $\T'(P)$ satisfies the $t$-Frostman Convex Wolff Axiom with error $\sim C_2C_3C_4$. We carry this argument out for each prism $P$.

        Now, we apply a Cordoba style argument. Denote by $Y(P) = \bigcup_{T \in \T(P)} Y(T)$. By Cauchy-Schwarz, we know that \begin{equation}\label{eqn:meerkat1}\left|\bigcup_{P \in \calP} Y(P)\right|_{\de} \ge \frac{\left(\sum_{P \in \calP }|Y(P)|_{\de} \right)^2}{\sum_{P_1, P_2} |Y(P_1) \cap Y(P_2)|_{\de}} \gtrsim C_{4}^{2} |\calP|^2 \left(\frac{\De}{\de}\right)^{2t+2}\left(\sum_{P_1, P_2} |Y(P_1) \cap Y(P_2)|_{\de}\right)^{-1}.\end{equation}
             
        As is standard for these sorts of arguments, we dyadically pigeonhole on the incidence angle as follows: \begin{align*}\sum_{P_1, P_2} |Y(P_1) \cap Y(P_2)|_{\de} &= \sum_{P_1 \in \calP} \sum_{j = 0}^{~|\log (\de/\De)|^{-1}} \sum_{\substack{P_2 \in \calP \\ \angle(n(P_1), n(P_2)) \in [2^j\frac{\de}{\De}, 2^{j+1}\frac{\de}{\De}) }} |Y(P_1) \cap Y(P_2)|_{\de}\\ &\le |\calP||\log \de|^{-1}\sup_{0 \le j \le |\log \de/\De|^{-1}} \sum_{\substack{P_2 \in \calP \\ \angle(n(P_1), n(P_2)) \in [2^j\frac{\de}{\De}, 2^{j+1}\frac{\de}{\De}) }} |Y(P_1) \cap Y(P_2)|_{\de}.\end{align*}
        If two prisms $P_1, P_2 \in \calP$ intersecting with angle $\sim 2^{j}\frac{\de}{\De}$, then $P_2$ is contained in the $\sim 2^j\de \times \De \times \De$-neighborhood of $P_1$. Applying the fact that $\calP$ satisfies the $t$-Frostman Convex Wolff Axiom in $G$ with error $C_1$, we conclude that \[\left|\left\{ P_2 \in \calP : Y(P_1) \cap Y(P_2) \neq \emptyset, \angle(n(P_1), n(P_2)) \lesssim 2^j \frac{\de}{\De}\right\}\right| \le C_1|\calP|\left(\frac{2^j\de}{\nu}\right)^t.\]

        It follows that for each $j$ between $0$ and $|\log \de|^{-1}$, \[\sum_{\substack{P_2 \in \calP \\ \angle(n(P_1), n(P_2)) \in [2^j\frac{\de}{\De}, 2^{j+1}\frac{\de}{\De}) }} |Y(P_1) \cap Y(P_2)|_{\de} \le C_1|\calP|\left(\frac{2^j\de}{\nu}\right)^t \sup_{\substack{P_2 \in \calP \\ \angle(n(P_1), n(P_2)) \in [2^j\frac{\de}{\De}, 2^{j+1}\frac{\de}{\De}) }} |Y(P_1) \cap Y(P_2)|_{\de}\]
        Comparing with (\ref{eqn:meerkat1}), we see that the conclusion of the lemma would follow from proving that for each $P_1, P_2$ intersecting at angle $\theta \sim \frac{2^j\de}{\De}$, \begin{equation}\label{eqn:621}|Y(P_1) \cap Y(P_2)|_{\de} \le |\log \de| C_2C_3 2^{-jt} \left(\frac{\De}{\de} \right)^{t+1}.\end{equation}It remains to prove (\ref{eqn:621}).

        Since $P_1$ intersects $P_2$ at angle $\theta$, $P_1 \cap P_2$ is contained in a $\de \times \frac{\de}{\theta} \times \De$ prism $V$. Denote by $\ell$ a line of length $\De$ in this prism. Dyadically pigeonholing, we refine $\T'(P_2)$ to some subcollection $\T''(P_2)$ for which $\angle(T, \ell) \sim \zeta$ for each $T \in \T''(P_2)$ and \[\sum_{T \in \T''(P_2)} |V \cap Y(P_2)|_{\de} \ge |\log \de|^{-1} \sum_{T \in \T'(P_2)} |V \cap Y(P_2)|_{\de}.\]
        Note that the summands of $\sum_{T \in \T''(P_2)} |V \cap Y(P_2)|_{\de}$ must fall in the $\De \zeta + \frac{\de}{\theta} \sim \max\left(\De \zeta, \frac{\de}{\theta}\right)$-neighborhood $W$ of $\ell$ and each tube intersects $V$ in a $\de$-tube segment of length $L = \min\left(\frac{\de}{\theta \zeta}, \De\right)$. First, suppose that $\De \zeta \ge \frac{\de}{\theta}$. Equivalently, $\De \ge \frac{\de}{\theta \zeta}$, so for each $T \in \T''(P_2)$, $|T \cap V|_{\de} \sim \frac{L}{\de} = \frac{1}{\theta\zeta}$. Since $\T(P_2)$ satisfies the Frostman Convex Wolff Axiom in $P$ with multiplicity $\le C_3\left(\frac{\De}{\de}\right)^t$ and error $C_2$, we conclude that \begin{equation}\label{eqn:622}|\T''(P_2)[W]| \le |\T(P_2)| \lesssim C_2C_3\left(\frac{\De\zeta}{\de}\right)^t.\end{equation}Since we assumed $\De \zeta \ge \frac{\de}{\theta}$, we know that $(\theta \zeta)^{t-1} \le \left(\frac{\De}{\de}\right)^{1-t}$. Recall also that $\theta \sim 2^j\frac{\de}{\De}$. Applying these bounds and (\ref{eqn:622}), we union bound to see \[|Y(\T''(P_2)[W]) \cap V|_{\de} \le C_2C_3\left(\frac{\De\zeta}{\de}\right)^t(\theta \zeta)^{-1} \le \theta^{-t}(\zeta\theta)^{t-1}\left(\frac{\De}{\de}\right)^t \le \theta^{-t}\frac{\De}{\de} \le 2^{-jt}\left(\frac{\De}{\de}\right)^{t+1}.\]We confirm (\ref{eqn:621}) in this case. 
        
        Now suppose that $\De \zeta < \frac{\de}{\theta}$. Then the tubes incident to $V$ must lie in $2V$ so similarly to the previous paragraph, we have $|\T(P_2)[2V]| \lesssim C_2 C_3 \theta^{-t}$, while $|T \cap V|_{\de} \le |T|_{\de} = \frac{\De}{\de}$. We conclude that $|Y(\T(P_2)[2V])|_{\de} \lesssim C_2 C_3 \theta^{-t} \frac{\De}{\de} \le C_2 C_3 2^{-jt} \left(\frac{\De}{\de} \right)^{t+1}$. Either way, we have (\ref{eqn:621}).
\end{proof}
\begin{lemma*}[\textbf{\ref{lem:RFPC}}]
    Fix $\De \in (\de, 1]$ and $\nu \in (\De, 1]$. Suppose that $\calP$ is a collection of $\de \times \De \times \De$ prisms satisfying the $t$-Frostman Convex Wolff Axiom with error $C_1$ inside some $G$ a $\nu \times \De \times \De$ prism. Suppose $\calP$ is equipped with a $C^{-1}_2$ dense shading $Y$. Then \[\left|\bigcup_{P \in \calP} Y(P)\right|_{\de} \gtrsim |\log \de|^{-1} C_1^{-1}C_2^{-2} \left(\frac{\nu}{\de}\right)^{t}\left(\frac{\De}{\de}\right)^2.\]
\end{lemma*}

\begin{proof}
    This follows a similar but fortunately simpler Cordoba style argument as Lemma \ref{lem:RFPB}. We again see that \begin{equation}\label{eqn:elk2}\left|\bigcup_{P \in \calP} Y(P)\right|_{\de} \ge \frac{|\calP|^2 (\inf_{P \in \calP} |Y(P)|_{\de})^2}{\sum_{P_1, P_2 \in \calP} |Y(P_1) \cap Y(P_2)|_{\de}} \ge C_2^{-2}\left(\frac{\De}{\de}\right)^2 \frac{|\calP|^2(\De/\de)^2}{\sum_{P_1, P_2 \in \calP} |Y(P_1) \cap Y(P_2)|_{\de}}.\end{equation}As in Lemmas \ref{lem:RFPA} and \ref{lem:RFPB}, we have that \begin{equation}\label{eqn:elk1}\sum_{P_1, P_2 \in \calP} |Y(P_1) \cap Y(P_2)|_{\de} \lessapprox_{\de} |\calP| \sup_{P_1 \in \calP} \sup_{\theta \in [\de/\De, 1] \text{ dyadic}} \sum_{\substack{P_2 \in \calP \\ \angle(n(P_1), n(P_2)) \in [\theta, 2\theta]}} |P_1 \cap P_2|_{\de}.\end{equation}For a fixed choice of $P_1 \in \calP$ and dyadic value $\theta \in [\de/\De, 1]$, at most $C_1\frac{(\theta \De^3)^{t}}{(\nu \De^2)^t}|\calP| = C_1 \theta^t \left(\frac{\De}{\nu}\right)^t |\calP|$ many prisms $P_2$ intersect $P_1$ with $\angle(n(P_1), n(P_2)) \in [\theta, 2\theta]$. For each such prism, $|P_1 \cap P_2|_{\de} \le \frac{\De}{\theta\de}$. Using this bound in (\ref{eqn:elk1}) and the fact that $\theta \ge \frac{\de}{\De}$, we conclude that \[\sum_{P_1, P_2 \in \calP} |Y(P_1) \cap Y(P_2)|_{\de} \lessapprox_{\de} C_1|\calP|^2 \theta^t \left(\frac{\De}{\nu}\right)^t \frac{\De}{\theta\de} \le C_1|\calP|^2 \left(\frac{\de}{\nu}\right)^t\left(\frac{\De}{\de}\right)^{2}.\]Comparing with (\ref{eqn:elk2}), we see that \[\left|\bigcup_{P \in \calP} Y(P)\right|_{\de} \gtrapprox_{\de} C_1^{-1}C_2^{-2}\left(\frac{\nu}{\de}\right)^t\left(\frac{\De}{\de}\right)^{2}.\]
\end{proof}

\subsection{The proof of Lemma \ref{lem:brush}}\label{subsec:brush}

Lemma \ref{lem:brush} is an example of a brush argument, where we construct a family of prisms satisfying the $t$-Frostman Convex Wolff Axiom with small error by taking almost parallel prisms all incident to a common ``stem'' prism. Here, we apply it for families of prisms with a shading supported on unions of tubes. Unlike the usual brush argument, we only aim to build the set of bristles of the brush and prove they satisfy the $t$-Frostman Convex Wolff Axiom. 

\begin{lemma*}[\textbf{\ref{lem:brush}}]
    For any $\tau > 0$ there exists $\eta' > 0$ such that for any $\beta_2 > 0$ there exists $\de_0> 0$ such that the following holds for all $\de \in (0, \de_0)$, any $\De \in (\de, 1)$ and any $\theta \ge \frac{\de}{\De}$.

    Let $\calP$ be a set of $\de \times \De \times \De$ prisms, $\T$ a set of $\de \times \de \times \De$ tube segments, and $Y$ a shading on $\T$. Suppose that $\T = \bigcup_{P \in \calP} \T(P)$, where $\T(P)$ satisfies the $t$-Frostman Convex Wolff Axiom in $P$ with error $\de^{-\eta'}$ and multiplicity $\le \de^{-\eta'}\left(\frac{\De}{\de}\right)^t$ and $|Y(\T(P))|_{\de} \ge \de^{\eta'}\left(\frac{\De}{\de}\right)^{t+1}$. Suppose furthermore that for each $\de$-cube $Q \in Y(\T)$, $\{n(P) : Q \in Y(\T(P))\} \subset S^2$ is $\beta_2$-broadly spread in a $\theta$-ball. 

    Then for any $P_0 \in \calP$, if $G(P_0)$ is the concentric $3\theta\De \times 3\De \times 3\De$ prism, there exists some $\mathscr{S}_{P_0} \subset \calP$ so that each $P' \in \mathscr{S}_{P_0}$ is contained in $G(P_0)$, $\mathscr{S}_{P_0}$ satisfies the $t$-Frostman Convex Wolff Axiom in $G(P_0)$ with error $\de^{-\tau}$, and $Y(\T(P')) \cap Y(\T(P_0)) \neq \emptyset$ for each $P' \in \mathscr{S}_{P_0}$.
\end{lemma*}

We now outline the proof of Lemma \ref{lem:brush}. The proof requires some modification from the standard brush argument, since we need to make sure the bristles of the brush (which are slabs in this setting) are essentially parallel to ensure we can find enough space on the ``stem'' of the brush for all the bristles we need. This requires a round of pigeonholing. We also must take additional care in how (or if) we define a ``stem'' for the brush. Let $P$ be the base slab we are building the brush from. The family of parallel incident prisms mentioned previously intersect $P$ in a collection of $\de \times \frac{\de}{\theta} \times \De$ prisms $\mathcal{S}$ contained in $P$. If elements of $\mathcal{S}$ are all essentially perpendicular to each tube in $P$, then we can use one tube as the ``stem'' on the brush. If the elements of $\mathcal{S}$ are all essentially parallel to each tube in $P$, then they look like the $\frac{\de}{\theta}$-neighborhood of the set of tubes in $P$. We can use the fact that the tube set satisfies the $t$-Frostman Convex Wolff Axiom to find enough prisms. This describes the two extreme cases. The intermediate cases combine the reasoning of the extreme cases. 

\begin{proof}
    Take some $P_0 \in \calP$. For any $\de$-cube $Q \in Y(P_0)$, since $\{n(P) : Q \in Y(\T(P))\}$ is $\beta_2$-broadly spread in a $\theta$ ball for each $Q \in Y(P_0)$, there exists a prism $P_Q$ with $\angle(n(P_Q), n(P_0)) \sim \theta$ and $Q \in Y(P_Q)$. Since $P_Q$ is a $\de \times \De \times \De$ prism, $P_Q \cap P_0$ is contained in a $\sim\de \times \frac{\de}{\theta}\times \De$ \footnote{ The $\sim$ denotes that each constant side length can vary by a $\sim 1$ factor depending on the broadness error} prism $S_A$. We denote by $\ell_Q$ a line in this prism of length $\De$. Take some $\T' \subset \T[P_0]$ on which $|Y(T)|_{\de} \ge \de^{3\eta'}\frac{\De}{\de}$ for each $T \in \T'$ with $|\T'| \ge \de^{3\eta'}|\T|$. this can be done since the shading is $\de^{\eta'}$ dense and supported on $\T[P_0]$. For each tube $T \in \T'$ and each $Q \in Y(\T')$, let $\zeta(T, Q) = \max\left( \angle(T, \ell_Q), \frac{\de}{\De\theta}\right)$. After pigeonholing, we may assume that $\zeta(T, Q) \in [\zeta, 2\zeta)$ for some fixed $\zeta \in [\de, 1]$ and for each $T \in \T^* \subset \T'$ and $Q \in Y^*(\T^*) \subset Y'(\T^*)$ with $|\T^*|\gtrapprox_{\de} |\T'|$ and $|Y^{\ast}(\T^*)| \gtrapprox_{\de} |Y(\T')|$. After another refinement of $\T'$, we can ensure that $|\T^{\ast}| \ge \de^{3\eta'}|\T'|$ and $|Y^{\ast}(T)|_{\de} \gtrapprox_{\de} \de^{6\eta'}\frac{\De}{\de}$ for each $T \in \T^{\ast}$. It follows that $\T^{\ast}$ satisfies the $t$-Frostman Convex Wolff Axiom with error $\de^{-7\eta'}$.
    
    Using a Vitali-style covering argument and dyadic pigeonholing, we can find a set of tubes $\T^{\ast}_{\text{small}} \subset \T^{\ast}$ and a set of $\De\zeta$ tubes $\T^*_{\text{big}}$ a $\sim 1$-balanced, partitioning cover of $\T^{\ast}_{\text{small}}$ such that $|\T^{\ast}_{\text{small}}| \sim |\T^{\ast}|$ and the elements of $\T^{\ast}_{\text{big}}$ are $5\De\zeta$-separated. Since $|\T^{\ast}_{\text{small}}| \sim |\T^{\ast}|$ and $\T^{\ast}$ satisfies the $t$-Frostman Convex Wolff Axiom with error $\de^{-7\eta'}$, we know that as long as $\de$ is sufficiently small, $\T^{\ast}_{\text{small}}$ satisfies the $t$-Frostman Convex Wolff Axiom with error $\sim \de^{-7\eta'} \le \de^{-8\eta'}$. Since the $t$-Frostman Convex Wolff Axiom is inherited upwards, $\T^{\ast}_{\text{big}}$ satisfies the $t$-Frostman Convex Wolff Axiom with error $\le \de^{-8\eta'}$ as well. We denote $M = |\T^{\ast}_{\text{big}}|$ and enumerate the tubes in $\T^{\ast}_{\text{big}}$ as $T_1, T_2, \dots, T_M$. 

    Each $T_i \in \T^*_{\text{big}}$ contains some $T_{\text{small}} \in \T^{\ast}_{\text{small}}$, which supports a $\de^{6\eta'}$-dense shading $Y^*(T)$. It follows that we can find $\frac{\de}{\theta \zeta}$-separated points $x_{i,1}, \dots, x_{i,N}$ with each $x_{i,j}$ contained in $Y^*(T)$ and $N = \max\left(\left\lfloor \de^{6\eta'}\frac{\theta \zeta}{\de} \right\rfloor, 1\right)$. We record a useful Frostman non-concentration condition for these points.

    \begin{equation}\label{eqn:theendofbeginnings}|\{j : x_{i, j} \in Y^*(\lambda) \}| \le \de^{-6\eta'}\rho^tN \text{ for any }\de \times \de \times \rho \text{ tubelet }\lambda \subset T.\end{equation}

    After carrying this out for each of $M$ tubes in $\T^*_{\text{big}}$, we arrive at a collection of $NM$ $\frac{C\de}{\theta \zeta}$-points $x_{i,j}$, each of which is incident to a prism $P_{x_{i,j}}$ making angle $\gtrsim \theta$ with $P_0$. 

    Suppose $P_{x_{i,j}} = P_{x_{k, l}}$. If $i \neq k$ (that is, the points are drawn from different tubes), note that $\ell_{x_{i,j}}$ makes angle $\approx \zeta$ with $T_i$ and with $T_k$. If $\zeta \le \frac{\de}{\theta\De}$, since $S_{x_{k,l}} \cap S_0 = S_{x_{i,j}} \cap S_0$ is a $\de \times \frac{\de}{\theta} \times 1$ prism and $\De \zeta \le \frac{\de}{\theta}$, $T_i, T_k \subset N_{2\De\zeta}(\ell_{x_{i,j}})$, so $T_i \subset N_{5\De\zeta}(T_k)$, a contradiction. If $\zeta > \frac{\de}{\theta\De}$, then $\ell_{x_{i,j}}$ intersects both $T_i$ and $T_k$. It then follows that both $T_i, T_k \subset N_{2\De\zeta}(\ell_{x_{i,j}})$, leading in the same way to a contradiction.
    If $i = k$ but $j \neq l$, then $S_{x_{i, k}} \cap T_i$ and $S_{x_{j, l}} \cap T_i$ are both intervals of length $\frac{\de}{\theta\zeta}$, so since $j \neq l$, they must be disjoint. It follows that $(i, j) = (k, l)$, hence $\mathscr{S}_{P_0} = \left\{S_{x_{i,j}} : i = 1, \dots, M; j = 1, \dots, N\right\}$ consists of $NM$ distinct prisms. This serves as our choice of $\mathscr{S}_{P_0}$. As desired, for each $P_x \in \mathscr{S}_{P_0}$, $x \in Y(\T(P_Q)) \cap Y(\T(P_0))$ and in particular $Y(\T(P_Q)) \cap Y(\T(P_0)) \neq \emptyset$.

    It remains to prove that $\mathscr{S}_{P_0}$ satisfies the $t$-Frostman Convex Wolff Axiom in $G(P)$ with error $\de^{-\tau}$. Suppose $W \subset G(P)$ is convex. For convenience, we denote $\nu = 3\theta \De$ for the rest of the proof. If $\mathscr{S}_{P_0}[W] = \emptyset$, then certainly $|\mathscr{S}_{P_0}[W]| \le \de^{-\tau}(|W|/|P|)^t|\mathscr{S}_{P_0}|$, so assume $\mathscr{S}_{P_0}[W] \neq \emptyset$. This means that $W$ is essentially a $\kappa \times \De \times \De$ prism with $\kappa \in [\de, \nu]$ and $\angle(n(W), n(S_0)) \sim \theta$. It follows that $W \cap S_0 := \tilde{W}$ is approximately a $\de \times \kappa/\theta \times \De$ prism. Furthermore, there exists a length $\De$ line segment $\ell_W$ in $W \cap S_0$ which makes angle $\sim \zeta$ with each $T \in \T^{\ast}_{\text{small}}$. It follows that the tubes from $\T^{\ast}_{\text{big}}$ intersecting $\tilde{W}$ are contained in $N_{\De\zeta}(\tilde{W})$, a $\de \times \max(\De\zeta, \kappa/\theta) \times \De$ prism. 

    If $\max(\De\zeta, \kappa/\theta) < \De^2$, then there are $\lesssim \de^{-8\eta'}\max\left(\De\zeta, \frac{\kappa}{\theta}\right)^{t}\De^{-2t}M$ many such tubes, since $\T^{*}_{\text{big}}$ satisfies the $t$-Frostman Convex Wolff Axiom in $P_0$ with multiplicity $M$ wotj error $\de^{-8\eta'}$. We also see that for each $T \in \T^{\ast}_{\text{small}}$ intersecting $\tilde{W}$, $\tilde{W} \cap T$ has length $\min\left(\frac{\kappa}{\theta \zeta}, \De\right)$ and by (\ref{eqn:theendofbeginnings}) intersects fewer than $\de^{-6\eta'}\min\left(\frac{\kappa}{\theta \zeta}, \De\right)^tN$ many $Q_{i,j}$. From this and the fact that $S_{Q_{i,j}} \subset W$ implies the $Q_{i,j}$ intersects $W$, we conclude that \begin{align*}|\mathscr{S}_{P_0}[W]| &\lesssim \de^{-14\eta'}\min\left(\frac{\kappa}{\theta \zeta}, \De\right)^tN\max\left(\De\zeta, \frac{\kappa}{\theta}\right)^{t}\De^{-2t}M \\ &\quad\quad\le \de^{-14\eta'} \left(\frac{\kappa}{\nu}\right)^t MN = \de^{-14\eta'}(|W|/|P|)^t|\mathscr{S}_{P_0}|.\end{align*}

    If $\max\left(\De\zeta, \frac{\kappa}{\theta}\right) > \De^2$, we trivially bound the number of incident tubes from $\T^*_{\text{small}}$ by $M$. As in the previous case, we use (\ref{eqn:theendofbeginnings}) to bound the number of cubes on each tube by $\de^{-6\eta'}\min\left(\frac{\kappa}{\theta \zeta}, \De\right)^tN$. If $\max\left(\De\zeta, \frac{\kappa}{\theta}\right) = \De\zeta$, then $\zeta \ge \De$ and so $\min\left(\frac{\kappa}{\theta \zeta}, \De\right) = \frac{\kappa}{\theta \zeta} \le \frac{\kappa}{\nu}$ (recall $\nu = \theta \De$). If $\max\left(\De\zeta, \frac{\kappa}{\theta}\right) = \frac{\kappa}{\theta}$, then $\frac{\kappa}{\nu} \ge \De$, so $\min\left(\frac{\kappa}{\theta \zeta}, \De\right) = \De \le \frac{\kappa}{\nu}$. In either case, we see as in the previous paragraph that \[|\mathscr{S}_{P_0}[W]| \lesssim \de^{-14\eta'}\min\left(\frac{\kappa}{\theta \zeta}, \De\right)^tNM \le \de^{-14\eta'} \left(\frac{\kappa}{\nu}\right)^t MN = \de^{-14'\eta}(|W|/|P|)^t|\mathscr{S}_{P_0}|.\]
    Hence, as long as $14\eta' \le \tau$, $\mathscr{S}_{P_0}$ satisfies the $t$-Frostman Convex Wolff Axiom in $G(P)$ with error $\le \de^{-\tau}$. 
\end{proof}

\printbibliography
\end{document}